\documentclass[12pt]{article}
\pdfoutput=1

\usepackage[lofdepth,lotdepth,caption=false]{subfig}
\usepackage{fancyhdr}
\usepackage{hyperref}
\usepackage{amsthm}
\usepackage{amsmath,  amssymb, graphicx,harvard}
\usepackage{xspace}
\usepackage{braket}
\usepackage{color}
\usepackage{setspace}
\usepackage{float}
\usepackage{bm}
\usepackage{multirow}
\usepackage{lscape}
\usepackage{algorithm,algcompatible}
\usepackage{mathrsfs}
\usepackage[bottom]{footmisc}
\usepackage{appendix}

\newcommand{\bftheta}{\theta}

\newcommand{\str}{^*}

\newcommand{\CL}{{\cal L}}

\newcommand{\CH}{{\cal H}}

\newcommand{\CX}{{\cal X}}

\DeclareMathOperator{\Var}{Var}
\DeclareMathOperator{\SKL}{SKL}
\DeclareMathOperator{\E}{E}

\newcommand{\bb}[1]{\mathbb{#1}}
\newcommand{\mf}[1]{\mathfrak{#1}}
\newcommand{\R}{\ensuremath{\bb{R}}}
\newcommand{\Q}{\mathcal{Q}}
\newcommand{\N}{\mathcal{N}}
\newcommand{\X}{\mathcal{X}}
\newcommand{\abs}[1]{\left\lvert#1\right\rvert}
\newcommand{\brac}[1]{\left\{#1\right\}}
\newcommand{\ip}[1]{\left\langle#1\right\rangle}
\newcommand{\norm}[1]{\left\lVert#1\right\rVert}
\DeclareMathOperator*{\argmin}{arg\,min}

\newtheorem{theorem}{Theorem}

\newtheorem{assumption}{Assumption}
\newtheorem{lemma}{Lemma}
\newtheorem{remark}{Remark}
\newtheorem{cor}{Corollary}

\newtheoremstyle{appendixtheorem}
  {\topsep}   
  {\topsep}   
  {\itshape}  
  {0pt}       
  {\bfseries} 
  {.}         
  {5pt plus 1pt minus 1pt} 
  {#1 A.#2}          
\theoremstyle{appendixtheorem}
\newtheorem{atheorem}{Theorem}
\newtheorem{alemma}{Lemma}
\newtheorem{aprop}{Proposition}

\title{Smoothing Spline Semiparametric Density Models}
\author{$\hbox{Jian Shi}^1 \footnote{Joint first authors.}, \hbox{Jiahui Yu}^2 \footnotemark[\value{footnote}], \hbox{Anna Liu}^2, \hbox{Yuedong Wang}^1$}
\date{}
\begin{document}
\maketitle

\center{$\hbox{University of California, Santa Barbara}^1$}\\
\center{$\hbox{University of Massachusetts, Amherst}^2$}

\abstract{Density estimation plays a fundamental role in many areas of statistics and machine
learning. Parametric, nonparametric and semiparametric density estimation methods have been proposed in the literature. Semiparametric density models are flexible in incorporating domain knowledge and uncertainty regarding the shape of the density function. Existing literature on semiparametric density models is scattered and lacks  a systematic framework. In this paper, we consider a unified framework based on the reproducing kernel Hilbert space for modeling, estimation, computation and theory. We propose general semiparametric density models for both a single sample and multiple samples which include many existing semiparametric density models as special cases. We develop penalized likelihood based estimation methods and computational methods under different situations. We establish joint consistency and derive convergence rates of the proposed estimators for both the finite dimensional Euclidean parameters and an infinite-dimensional functional parameter. We validate our estimation methods empirically through  simulations and an application.
}

\section{Introduction}
Density estimation plays a fundamental role in many areas of statistics and machine
learning. The estimated density functions are useful for model building and diagnostics,
inference, prediction and clustering. Traditionally there are two distinct approaches for density
estimation: maximum likelihood method within a parametric family and nonparametric
methods. Parametric models are in general parsimonious with meaningful and interpretable
parameters \cite{pearson1902systematic,pearson1902systematic2,kendall1946advanced,fisher1997absolute}. Nonparametric methods based on minimal assumptions are in general more
flexible \cite{Silverman,izenman1991review,Gu2013bk}. Often in practice it is desirable to model some components of the density function
parametrically while leaving other components unspecified. Many semiparametric density
models have been proposed for different purposes, including the partially linear semiparametric density models \cite{yang2009penalized}, the mixture models \cite{olkin1987semiparametric,bordes2006semiparametric,ma2015flexible}, mixture models with exponential tilts for multiple samples \cite{Anderson1972,Qin1999,ZouFineYandell2002,Tan2009}, combinations of parametric and nonparametric functions \cite{hjort1995nonparametric,efron1996using}, semiparametric models for Bayesian testing \cite{lenk2003bayesian}, and transformation models including the location-scale family distributions \cite{wand1991transformations,potgieter2012nonparametric}. However, a systematic framework with a broad formulation of the semiparametric density models is lacking. In this paper, we propose a general  framework that includes the above mentioned semiparametric density models as special cases. We develop a smoothing spline based estimation method for the general model and prove the asymptotic results based on our estimation method. The general framework provides a unified platform for the developments of estimation, computation, and theory.

For a single random sample, $X_{i}$, $i=1,\ldots,n$, from a common probability density $f(x)$ on a general domain $\CX$, we consider the following general semiparametric density model 
\begin{equation}
f(x)=\frac{\exp\left\lbrace\eta(x;\bftheta,h)\right\rbrace}{\int \exp\left\lbrace\eta(x;\bftheta,h)\right\rbrace dx}, \label{Semimodel}
\end{equation}
where $\eta$ is a known function of $x$ given $\bftheta$ and $h$, which will be referred to as the logistic transformation of $f$. The  parameter $\bftheta \in \mathbb{R}^p$ and the nonparametric function $h$ are unknown and need to be estimated. Often certain conditions, which depend on the form of $\eta(x;\theta,h)$, are necessary to make model \eqref{Semimodel} identifiable. We assume that model \eqref{Semimodel} is identifiable, and discuss identifiability conditions for specific models in the following sections.

Many  existing semiparametric models are special cases of model (\ref{Semimodel}).  \citeasnoun{olkin1987semiparametric} proposed a mixture of a parametric and a nonparametric density function, namely,
$f(x)=\theta_0 f_1(x,\theta_1)+(1-\theta_0)f_2(x)$,
where $f_1(x,\theta_1)$ is a known density function up to parameters $\theta_1 \in \R^p$, $\theta_0 \in [0,1]$ is an unknown weight parameter, and $f_2(x)$ is a nonparametric density function.
They showed that the semiparametric density estimate provides a compromise between the parametric and nonparametric estimates. It is a special case of model (\ref{Semimodel}) with $\eta(x)=\log\{\theta_0 f_1(x,\theta_1)+(1-\theta_0)f_2(x)\}$, where $\theta_0$ and $\theta_1$ are the parameters and $f_2$ is the nonparametric function.
 \citeasnoun{hjort1995nonparametric} proposed a density estimation procedure by starting out with a parametric density estimate $f_1(x,\hat{\theta})$, and then multiplying with a nonparametric kernel type estimate of a correction function $r(x)=f(x)/f_1(x,\hat{\theta})$.
It was shown that their semiparametric density model can perform better than a nonparametric fit when the true density is in the neighborhood of the initial parametric density. Their model is a special case of model (\ref{Semimodel}) with $\eta(x)=\log\{f_1(x,\bftheta)\}+ h(x)$, where $\theta$ is the parameter and $h(x) = \log \{r(x)\}$ is the nonparametric function.
\citeasnoun{efron1996using} proposed a specially designed exponential family for density estimation, 
$f(x)=f_{0}(x)\exp\{\theta_{0}+t^T(x)\theta_{1}\}$,
where $f_{0}(x)$ is a carrier density and estimated by kernel density estimation, $t(x)$ is a $p$-dimensional vector of sufficient statistics, $\theta_{1}$ is a $p$-dimensional vector of parameters, and $\theta_{0}$ is a normalizing parameter making $f(x)$ integrate to 1 over $\mathcal{X}$. The proposed method matches the estimated expectation of $t(x)$ with its sample expectation. 
The model is a special case of model (\ref{Semimodel}) with 
$\eta(x)=\log \{f_0(x)\}+\theta_0+ t^T(x)\theta_1$, where $f_0(x)=\exp \{h(x)\} / \int{\exp \{h(x)\} dx}$  is the carrier density given by nonparametric function $h(x)$, and $\theta_0$ and $\theta_1$ are the parameters.
\citeasnoun{lenk2003bayesian} proposed a flexible semiparametric model for Bayesian testing of 
$f (x|\theta , Z)= {\exp \{\alpha^T(x)\theta+Z(x) \} }/{\int_{\mathcal{X}}\exp \{ \alpha^T(x) \theta + Z(x) \}dG(x)}$,
where  $\alpha(x)=[\alpha^{1}(x),\ldots,\alpha^{m}(x)]^T$ is a vector of $m$ nonconstant functions, $Z$ is a zero mean, second-order Gaussian process with bounded, continuous covariance function, and $G$ is a known dominating measure on the support $\mathcal{X}$. The semiparametric model allows the predictive distribution
to deviate from the parametric family. If the parametric family is inadequate, the semiparametric predictive density coherently adapts to the data. 
\citeasnoun{yang2009penalized} also used the logistic transformation of density function as in \citeasnoun{lenk2003bayesian} with $Z(x)$ modeled as an unknown smooth function. The model in \citeasnoun{yang2009penalized} is a special case of model (\ref{Semimodel}) with $\eta(x)=\alpha^T(x)\theta+h(x)$.
\citeasnoun{wand1991transformations} considered density estimation of data with local features, and proposed a data transformation technique so that global smoothing is appropriate.
This transformation model is a special case of model (\ref{Semimodel}) with $\eta(x)=h(t(x; \theta))$ where $t(x; \theta)$ is the transformation, $\theta$ is the parameter and $h$ is the nonparametric function.

In the case of multiple samples, assume there are $m$ independent groups, and in each group $l=1,\cdots, m$, there are $n_l$ iid observations such that $X_{l1},\ldots,X_{ln_l}\overset{iid}{\sim} f_l(x; \bftheta, h)$ on domain $\CX_l$.
We consider the following general semiparametric density model 
\begin{equation}
f_l(x, \bftheta, h)=\frac{\exp\left\lbrace\eta_l(x;\bftheta,h)\right\rbrace}{\int \exp\left\lbrace\eta_l(x;\bftheta,h)\right\rbrace dx},\label{SemimodelM}
\end{equation}
where $\eta_l$, the logistic transformation of $f_l$, is a known function of the parameter $\bftheta$ and the nonparametric function $h$.
We are interested in the estimation of $\bftheta\in \mathbb{R}^p$, $h$, and ultimately the overall density function $f_l(x, \bftheta, h)$ from the observations. 

\citeasnoun{potgieter2012nonparametric} considered a two sample transformation model. Suppose that $X_{11},\ldots,X_{1n_1}\overset{iid}{\sim} f_1$ and $X_{21},\ldots,X_{2n_2}\overset{iid}{\sim}f_2$, and $X_2$'s have the same distribution as $X_1$'s after a certain invertible transformation parametrized by $\bftheta$, denoted as $t(x;\bftheta)$.
They considered the nonparametric estimation of the density function based on asymptotic likelihood, and showed that the estimators are often near optimal when compared to fully parametric methods.  The model is a special case of our model \eqref{SemimodelM} with
$\eta_1(x; \theta, h) = h(x)$, and $\eta_2(x; \theta, h) = h(t(x; \theta))\abs{t^\prime(x;\theta)}$,
where $h$ is the logistic transformation of $f_1$.
\citeasnoun{Anderson1972}, \citeasnoun{Qin1999}, \citeasnoun{ZouFineYandell2002}, and \citeasnoun{Tan2009} considered exponential tilt mixture models for biased sampling and case control studies, where multiple samples are collected and each sample follows an exponential tilt mixture model, that is, for sample $l$, 
$X_{l1},\ldots,X_{ln_1}\overset{iid}{\sim} \lambda_lf(x) +(1-\lambda_l)g(x)$, and 
$\log(f(x)/g(x))=\theta^1+\theta^2 x$.
We can see that this model is also a special case of model \eqref{SemimodelM} with
$
\eta_l(x; \theta, h)=\log\{\lambda_{l}+(1-\lambda_{l})\exp(\theta^1+\theta^2 x)\} +h(x)$,
where $\theta = [\theta^1,\theta^2]$ is the parameter and $h$ is the logistic transformation of $f$.

We consider the smoothing spline based estimation method for models (\ref{Semimodel}) and (\ref{SemimodelM}), that is, we assume that $h \in \CH$, where $\CH$ is a reproducing kernel Hilbert space (RKHS). Our work builds upon the existing literature and extends it to include semiparametric nonlinear density models. We develop computation methods based on profile penalized likelihood and backfitting, and 
joint asymptotic theory of the parametric and nonparametric components. For the single sample case, smoothing spline based nonparametric density estimation has been considered by many authors, including \citeasnoun{silverman_1982}, \citeasnoun{CoxOsullivan}, and \citeasnoun{Gu2013bk}. \citeasnoun{yang2009penalized} extended such methods for nonparametric models to a semiparametric density estimation model, where $\eta$ is linear in both the parametric and nonparametric components $\bftheta$ and $h$, as discussed above. 
Great progress has been made toward joint asymptotic theory in semiparametric models recently, with the seminal work by \citeasnoun{cheng2015joint} and follow-up papers \cite{ZhaoChengLiu16,ChaoVolgushevCheng17}. 
Prior to this, semiparametric asymptotic theory focused on the convergence of the parametric component, whereas the nonparametric component was usually considered as a nuisance parameter \cite{Bickelbk,Tsiatisbk,Kosorokbk}. We note that one approach is to extend the framework for joint asymptotics of the semiparametric linear regression model in \citeasnoun{cheng2015joint} to the semiparametric linear density model, and the same joint rate of convergence as derived by \citeasnoun{yang2009penalized} can be obtained. This approach relies heavily on the extension of the reproducing kenel of $\CH$, where the nonparametric component $h$ lives, to the product parameter space $\R^p \times \CH$ using the linear form of $\eta$ in $\bftheta$ and $h$. However, such an extension for a general nonlinear function $\eta(\bftheta,h)$ is nontrivial to obtain. In order to develop joint asymptotics for model (\ref{Semimodel}) with general $\eta$ and for model (\ref{SemimodelM}) with general $\eta_l$, we extend an approach from \citeasnoun{gu_qiu_1993} for nonparametric density estimation to the general nonlinear semiparametric model through a first order linearization technique motivated by the study of the nonlinear nonparametric regression model in \citeasnoun{o’sullivan_1990}. In particular, we specify the additional assumptions needed for our general framework and we introduce an appropriate metric in which the rate of convergence is derived for the joint estimator. Additional assumptions include, for example, smoothness criteria for $\eta$ and the existence of certain bounded linear operators related to the first Fr\'echet partial derivatives of $\eta$, all of which are redundant when $\eta$ is linear in $(\theta,h)$. In addition to the rate of convergence of the joint estimator, we also obtain the convergence rate of the parametric component in the standard $l^2$ norm as well as the convergence rate of the overall density function in the symmetrized Kullback-Leibler distance.

In Section  \ref{sec::Algm}, we consider models \eqref{Semimodel} and \eqref{SemimodelM} in three scenarios and develop different computation procedures. Asymptotic properties of the proposed estimators are considered in Section \ref{sec::theory}. Simulations are conducted to evaluate the proposed estimation procedures in Section \ref{sec::semiSimulation}. Section \ref{sec::realAnalysis} shows applications to real dataset.

\section{Some notations}
We begin by introducing some notations that will be used throughout the rest of the paper. Unless specified otherwise, let ${a = \left[a^k \right]_{k=1}^{p}}$ denote any $p \times 1$ vector, whose the $k$th element is $a^k$. The standard $l^2$ norm of $a \in \R^p$ is denoted $\norm{a}_{l^2} = [\sum_{k=1}^p (a^k)^2]^{1/2}$. We also denote any $p \times p$ matrix ${M = \left[M^{i,j} \right]_{i,j=1}^p}$, where $M^{i,j}$ represent the $(i,j)$th entry of the matrix. If there exist positive constants $c_1, c_2$, such that $c_1A \leq B \leq c_2A$, we write $A \sim B$. We use the notation $\E(\cdot)$ to represent the expectation taken over the joint sample distribution, whose density function is $f(x_1,\ldots,x_m) = \prod_{l=1}^{m} f_l(x_l)$. For simplicity, we will also sometimes use $\tau$ to represent the pair $(\theta, h)$, i.e., $\tau = (\theta, h)$ and $\eta_l(\tau)=\eta_l(\theta, h)$.

Denote the product parameter space as $\Q \equiv \R^p \times \CH$. Let $D_h$ be the Fr\'echet partial differential operator with respect to $h$, and let $D_{\theta}$ be the Fr\'echet partial differential operator with respect to $\theta$. If $X$ and $Y$ are any (real) Banach spaces, $\mathscr{L}(X,Y)$ represents the space of bounded linear operators from $X$ to $Y$. Note that for any function $f:\Q \rightarrow Y$, the Fr\'echet partial derivatives of $f$ are maps $D_{\theta}f: \Q \rightarrow \mathscr{L}(\R^p,Y)$ and $D_{h}f: \Q \rightarrow \mathscr{L}(\CH,Y)$ by definition. 

\section{Penalized likelihood estimation} \label{sec::Algm}
In this section, we will focus on model \eqref{SemimodelM}. Model \eqref{Semimodel} is a special case with $m=1$. We propose to estimate $\bftheta = [\theta^i]_{i=1}^p \in \R^p$ and $h \in \CH$ in \eqref{SemimodelM} as the minimizer of the following penalized likelihood with respect to $(\bftheta, h) \in \Q$:
\begin{align}
\sum_{l=1}^{m}\brac{ -\frac{1}{n_l} \sum_{i=1}^{n_l} \eta_l(X_{li};\theta,h) +
\log \int_{\X_l} e^{\eta_l(x;\theta,h)}dx } + \frac{\lambda}{2} J(h),
\label{PL}
\end{align}
where the first part of \eqref{PL} is the joint negative log likelihood, $J(h)$ is a square seminorm penalty, and $\lambda $ is the smoothing parameter. Assume that $\mathcal{H}_{0}=\{h:J(h)=0 \}$ is a $k$-dimensional space with basis functions $\psi(x)=[\psi^i(x)]_{i=1}^k$, then $\CH=\CH_0\oplus \CH_1$, where $\CH_1$ is an RKHS with reproducing kernel (RK) denoted by $R_J$.
 
The minimizer of the penalized likelihood (\ref{PL}) does not fall in a finite dimensional space. In the following,  we consider estimation of the model \eqref{SemimodelM} under the following three scenarios: additive, general, and transformation models.

\subsection{The additive model}\label{sec::additive}
The model \eqref{SemimodelM} is additive when $\eta_l$ is the sum of a term involving the parameters and the nonparametric component, that is, 
\begin{equation}
\eta_l(x;\bftheta,h)= \alpha_l(x; \mathbb{\bftheta}) +h(x), \label{additive}
\end{equation}
where $\alpha_l$ is a known function of  $x$ with unknown parameters $\bftheta$.  The model proposed by  \citeasnoun{efron1996using}, \citeasnoun{lenk2003bayesian}, \citeasnoun{yang2009penalized}, and \citeasnoun{hjort1995nonparametric} are special cases of $\eqref{additive}$ with $m=1$. In particular,  the first three of them considered the linear additive model in which $\alpha_l(x;\theta)=\bftheta^T a(x)$,
where $a(x) = [a^i(x)]_{i=1}^p$.
The model proposed in \citeasnoun{ZouFineYandell2002} is an example of the nonlinear additive model with
$\alpha_l(x)=\log\{\lambda_{l}+(1-\lambda_{l})\exp(\theta^1+\theta^2 x)\}$,
where $0\le\lambda_l\le1$ and $\lambda_1\ne \cdots \ne \lambda_m$.  

For the additive semiparametric density model (\ref{additive}), the penalized likelihood objective (\ref{PL}) becomes the following:  
\begin{equation}
	  \sum_{l=1}^{m}  \left\{-\frac{1}{n_l}\sum_{i=1}^{n_l}\alpha_l(X_{li};\bftheta)-\frac{1}{n_l}\sum_{i=1}^{n_l}h(X_{li})+\log\int_{\X_l} \exp\{\alpha_l(x;\bftheta)+h(x)\} dx\right\} +\frac{\lambda}{2}J(h). \label{additivePL2}
\end{equation}



We propose a profiled penalized likelihood based approach to optimize (\ref{additivePL2}).  First, with a fixed ${\bftheta}$, 
optimizing (\ref{additivePL2}) with respect to $h\in\CH$ is equivalent to optimizing the following weighted penalized likelihood:
\begin{equation}
\sum_{l=1}^m\left\lbrace -\frac{1}{n_l}\sum_{i=1}^{n_l} h(X_{li})+\log\int_{\X_l} w_l(x)e^{h(x)}dx \right\rbrace+\frac{\lambda}{2}J(h),\label{plfadditive}
\end{equation}
where $w_l(x)=\exp\{\alpha_l(x,\bftheta)\}$. As in Gu (2013), we approximate the solution by
\begin{equation}
\hat{h}_{\bftheta,\lambda}(x) = \sum_{\nu = 1}^{k}d^{\nu}\psi^{\nu}(x) + \sum_{j=1}^{q}c^{j}R_{J}(Z_{j},x)= \psi^{T}(x)d+\xi^{T}(x)c,\label{solutionOfH2}
\end{equation}
where $\{Z_j,j=1,\ldots,q \}$ is a random sample of $\{X_{li}, i=1,\ldots,n_l, l=1,\ldots,m\}$, $c=[c^{j}]_{j=1}^q$, ${d=[d^{\nu}]_{\nu=1}^k}$, $\psi(x) = [\psi^{\nu}(x)]_{\nu=1}^{k}$, and $\xi(x)=[R_{J}(Z_{1},x),\ldots, R_{J}(Z_{q},x)]^T$.

For $l=1,\ldots, m$ and $i=1,\ldots,n_l$, let $N=\sum_{l=1}^m n_l$, $W$ be the $N \times N$ diagonal matrix whose $r$th diagonal entry is $n_l$ where $r=(l-1)n_l+i$, $S$ be the $N\times k$ matrix whose $(r, \nu$)th entry is $\psi^{\nu}(X_{li})$ for $\nu = 1, \ldots,k$, and $R$ be the $N \times q$ matrix whose $(r, j)$th entry is $\xi^{j}(X_{li})=R_{J}(Z_{j}, X_{li})$ for $j = 1, \ldots, q$. Let $Q$ be the $q\times q$ matrix whose $(i, j)$th entry is $R_{J}(Z_{i}, Z_{j})$. Define for $l=1,\ldots, m$,
$\mu_{l,f}(g)={\int gw_le^{f}dx}/{\int w_le^{f}dx}$,
$V_{l,f}(g, h) = \mu_{l,f}(gh) - \mu_{l,f}(g)\mu_{l,f}(h)$,
and $V_{l,f}(g) = V_{l,f}(g, g)$. These notations extend in the obvious way to vectors and matrices of functions. 
 
We plug (\ref{solutionOfH2}) back into (\ref{plfadditive}) and compute minimizers $c$ and $d$ using the Newton iterative method. Let $\tilde{h}_{\bftheta, \lambda}(x) = \psi^{T}(x)\tilde{d}+\xi^{T}(x)\tilde{c}$,  where $\tilde{c}=[\tilde{c}^{j}]_{j=1}^q$ and $\tilde{d}=[\tilde{d}^{\nu}]_{\nu=1}^{k}$ are $c$ and $d$  vectors calculated in the previous step of the iteration. Denote 
\[\bar{\mu}_{\tilde{h}}({\psi})=\sum_{l=1}^m\mu_{l,\tilde{h}}({\psi})=\left[\sum_{l=1}^m\mu_{l,\tilde{h}}({\psi}^i)\right]_{i=1}^k,\quad V_{\psi, \tilde{h}} =\sum_{l=1}^mV_{l,\tilde{h}}(\psi, \tilde{h})= \left[\sum_{l=1}^mV_{l,\tilde{h}}(\psi^i, \tilde{h})\right]_{i=1}^{k},\]
\[V_{\psi,\psi}=\sum_{l=1}^mV_{l,\tilde{h}}(\psi,\psi^T) = \left[\sum_{l=1}^mV_{l,\tilde{h}}(\psi^i,\psi^j)\right]_{i,j=1}^k, \quad V_{\psi,\xi}=\sum_{l=1}^mV_{l,\tilde{h}}(\psi,\xi^T) = \left[\sum_{l=1}^mV_{l,\tilde{h}}(\psi^i,\xi^j)\right]_{i=1,j=1}^{k,q},\] 
and note $V_{\xi,\psi}$ is the transpose of $V_{\psi,\xi}$. We can also define
$\bar{\mu}_{\tilde{h}}({\xi})$, $V_{\xi,  \tilde{h}}$, and $V_{\xi,\xi}$ similarly to the first three equations above by replacing $\psi$ by $\xi$ and $k$ by $q$. Similar to \citeasnoun{Gu2013bk}, the Newton updating equation is 
\begin{equation}
\left( \begin{array}{cc}
V_{\psi, \psi } & V_{\psi, \xi } \\
V_{\xi, \psi } & V_{\xi, \xi } + \lambda Q
\end{array} \right) 
\left( \begin{array}{c}
d\\
c
\end{array} \right) =
\left( \begin{array}{c}
S^{T}W\boldsymbol{1}_N - \bar{\mu}_{\tilde{h}}({\psi})+V_{\psi, \tilde{h}} \\
R^{T}W\boldsymbol{1}_N- \bar{\mu}_{\tilde{h}}({\xi})+V_{\xi,\tilde{h}}
\end{array} \right).\label{additiveIteration2}
\end{equation}

At convergence of the Newton method, we denote the estimate of $h$ with fixed $\bftheta$ and $\lambda$ as $\hat{h}_{\bftheta, \lambda}$. 
Similar to  \citeasnoun{Gu2013bk}, the smoothing parameter $\lambda$ is selected by optimizing the relative Kullback-Leibler (KL) distance between the density associated with the estimate, $\eta_l(x;\bftheta,\hat{h}_{\bftheta,\lambda})$, and   the true density corresponding to $\eta_l(x;\bftheta,h)$ for $l=1,\ldots, m$,
\begin{equation}
\text{RKL}(\eta_l(x;\bftheta,h),\eta_l(x;\bftheta,\hat{h}_{\bftheta, \lambda}), l=1,\ldots,m)=\sum_{l=1}^m\left\{\log\int_{\X_l} e^{\eta_l(x;\bftheta,\hat{h}_{\bftheta,\lambda})}dx-\mu_{l,h}(\eta_l(x;\bftheta,\hat{h}_{\bftheta, \lambda}))\right\}.
\end{equation}
Following similar derivations in \citeasnoun{Gu2013bk}, an approximated cross-validation estimate of the relative Kullback-Leibler distance is 

\begin{equation}
V(\lambda,\bftheta)=\sum_{l=1}^m\left\{-\frac{1}{n_l}\sum_{i=1}^{n_l}\eta_l(X_i;\bftheta,\hat{h}_{\bftheta,\lambda})+\log\int_{\X_l} e^{\eta_l(x;\bftheta,\hat{h}_{\bftheta,\lambda})}dx+\alpha \frac{\text{tr}(P_{l,\bf1}^{\perp}\breve{R_l}H^{-1}\breve{R_l}^TP_{l,\bf1}^{\perp})}{n_l(n_l-1)}\right\},
\label{cv}\end{equation}
where $P_{l,\bf1}^{\perp}=I_{n_l}-\bold{1}_{n_l}\bold{1}_{n_l}^T/n_l$, $\breve{\xi}=(\psi^T,\xi^T)^T$, ${\breve{R}_l^T=(\breve{\xi} (X_{l1}),\ldots,\breve{\xi}(X_{ln_l}))}$, $H=\sum_{l=1}^mV_{l,\tilde{h}}(\breve{\xi},\breve{\xi}^T)+\text{diag}(0,\lambda Q)$, and $\alpha=1$. We may set a larger $\alpha$ to prevent under-smoothing, such as $\alpha=1.4$ as in \citeasnoun{Gu2013bk}.

The optimal smoothing parameter for  a given $\bftheta$, $\lambda_{\bftheta}$, is defined as the $\lambda$ that minimizes the cross validation criterion (\ref{cv}). The corresponding nonparametric function estimate $\hat{h}_{\bftheta,\lambda_{\bftheta}}$ is denoted as $\hat{h}_{\bftheta}$ for simplicity.
Plugging  $\hat{h}_{\bftheta}$ into \eqref{additivePL2}, we have the  profiled penalized likelihood 
\begin{equation}
\sum_{l=1}^m\left\{-\frac{1}{n_l}\sum_{i=1}^{n_l} (\alpha_l(X_{li};\bftheta)+\hat{h}_{\bftheta}(X_{li}))+\log\int_{\X_l} \exp\{\alpha_l(x;\bftheta)+\hat{h}_{\bftheta}(x)\} dx\right\}+\frac{\lambda_{\bftheta}}{2} J(\hat{h}_{\bftheta}). \label{nppl2}
\end{equation}
Then the estimate of $\bftheta$, $\hat{\bftheta}$, is the minimizer of \eqref{nppl2}. The minimization is achieved by the line search algorithm in \citeasnoun{nelder1965simplex}. The final estimate of $h$ is $\hat{h}_{\hat{\bftheta}}$.

In fact, for a given $\bftheta$, minimizing the  cross validation criterion (\ref{cv}) to obtain $\lambda_{\bftheta}$ and minimizing (\ref{additivePL2}) 
to obtain $\hat{h}_{\bftheta}$ can be easily implemented by utilizing the 
 {\em ssden} function in the R package {\em gss}  with the option {\em bias} to specify the weights \cite{gss}.

\subsection{The general nonlinear  model}\label{sec::nonadditive}
In this section, we consider the general form of $\eta_l$ in model \eqref{SemimodelM} that is not additive.
 For example, when $m=1$, 
 a mixture model of two densities assumes 
 $\eta(x;\theta,h)=\log\left\{\theta_0 f_1(x,\theta_1)+(1-\theta_0)h(x)\right\}$,
 where $0\leq \theta_0 \leq 1$ and $\theta = [\theta_0,\theta_1^1,\ldots,\theta_1^p]^T$. Originally considered by \citeasnoun{olkin1987semiparametric}, this model is often referred to as the two-component mixture model. Different estimation methods and applications can be found in \citeasnoun{bordes2006semiparametric} and \citeasnoun{ma2015flexible}.

We will extend the Gauss-Newton procedure for the estimation of $\bftheta$ and $h$ to the general nonlinear semiparametric density model. For fixed $\bftheta$, let $ \tilde{h}$ be the current estimate of $h$.  We assume that the  Fr\'echet derivative of $\eta_l$ with respect to $h$ at $(\theta,\tilde{h})$ exists, and we denote it by $D_{h}\eta_l(x; \theta, \tilde{h}) \equiv \CL_{l,x}$. Approximating $\eta_l$ by its first order Taylor expansion at $\bftheta$ and $ \tilde{h}$, we have
\begin{equation}
\eta_l(x; \bftheta,h) \approx  \eta_l(x;\bftheta,  \tilde{h})+ \CL_{l,x}(h- \tilde{h})\equiv r_l(x;\bftheta, \tilde{h})+\CL_{l,x} h,\label{ApproxNonadditive}
\end{equation}
where $r_l(x;\bftheta, \tilde{h})=\eta_l(x;\bftheta,  \tilde{h})-\CL_{l,x}  \tilde{h}$. 

Through this linearization, the nonlinear semiparametric density is approximated by an additive model as in Section \ref{sec::additive} but with the
evaluational functional replaced by the bounded linear functional $\CL_{l,x}$. Note that the method in Section \ref{sec::additive} can be extended to bounded linear functionals
with $\psi^\nu(X_{li})$ replaced by $\CL_{l,X_{li}} \psi^\nu$ and $\xi^j(X_{li})$ replaced by $\CL_{l,X_{li}}R_J(Z_j,\cdot)$, $j=1,\cdots, q$.

In summary, for the general nonlinear semiparametric density model, there are two loops of iterations. In the inner loop, for a fixed $\theta$, we estimate $h$ by the extended Gaussian-Newton iteration until convergence, where smoothing parameters are estimated at each iteration. In the outer loop, the profile penalized likelihood (\ref{nppl2}) is optimized with respect to $\bftheta$.

\subsection{The transformation model}\label{sec::powEst}
Model \eqref{SemimodelM} is called   a transformation model if
\begin{equation}
\eta_l(x;\bftheta,h)=h(\alpha_l(x;\bftheta)), \label{nonadditive}
\end{equation}
where $\alpha_l$'s are known differentiable and invertible functions with unknown parameters $\bftheta$.  When $\theta$ is fixed,  the logistic density of $Y_l=\alpha_l(X_l;\theta)$ is $\eta_l(y)=h(y)-\log\left(\vert \alpha_l'(\alpha_l^{-1}(y;\bftheta))\vert\right)$. Then the problem reduces to nonparametric density estimation with different weights.  The two sample location-scale family density model \cite{potgieter2012nonparametric} belongs to this case with $\alpha_1(x;\bftheta)=x$ and $\alpha_2(x;\bftheta)=(x-\mu)/\sigma$, where $\mu$ and $\sigma$ are the location and scale parameters and $\bftheta=(\mu,\sigma)$.

We note that the profiled penalized likelihood method  developed for the general nonlinear model in Section \ref{sec::nonadditive} applies to the 
 transformation model. Nevertheless,  the following backfitting procedure works better for the transformation model.   
First, provide initial value $\bftheta_0$  for  the parameter $\bftheta$. Next we iterate between the following two steps until convergence:
\begin{enumerate}
\item Based on current estimates $\tilde{\theta}$, transform the data using $Y_{li}=\alpha_l(X_{li};\tilde{\theta})$ and we have 
$\eta_Y(y)=h(y)-\log(\vert \alpha'(\alpha^{-1}(y;\tilde{\theta}))\vert)$. With the transformed data, update $h$ by minimizing the penalized likelihood
      \begin{equation}
      \sum_{l=1}^m\left\{-\frac{1}{n_l}\sum_{i=1}^{n_l}h(Y_{li})+\log\int\exp\left\lbrace h(y)\right\rbrace  w_l(y)dy\right\} +\frac{\lambda}{2} J(h)\label{backfittingTransform}
      \end{equation}
with the weight function $w_l(y)=1/\vert \alpha_l'(\alpha_l^{-1}(y;\tilde{\theta}))\vert$. 
\item Based on current estimate $\tilde{h}$ of $h$, update $\bftheta$ as the MLE based on the likelihood of 
$X_{11},\ldots,X_{mn_m}$.
\end{enumerate}
The minimization of the weighted penalized likelihood \eqref{backfittingTransform} is solved by the Newton method proposed in \citeasnoun{Gu2013bk}.


\section{Joint asymptotic theory}
\label{sec::theory}

In this section, we develop a joint asymptotic theory for the penalized likelihood estimators for models (\ref{Semimodel}) and (\ref{SemimodelM}). 

For $1 \le l \le m$, $f_l(x)$ is a density function over the sample space $\X_l$ of the form given by \eqref{SemimodelM}.
We consider $(\theta,h) \in \Q = \R^p \times \CH$ and a known function $\eta_l: \Q \to L^2_{0}(\X_l)$, where $L_{0}^{2}(\X_l) = L_{\tau_0}^{2}(\X_l) \ominus \{1\}$, $L^2_{\tau_0}(\X_l)$ is the space of functions on $\X_l$ with finite second moment with respect to the measure given by the true density $f_l(x;\theta_0,h_0)$, and $\tau_0 = (\theta_0,h_0)$ is the true parameter. The constant functions have been removed for identifiability. 
In general, $\eta_l(\theta,h)$ need not be one-to-one on $\Q$. However, for identifiability, we require that for all $l = 1, \ldots, m$, $\eta_l(\theta,h)$ is one-to-one in a neighborhood $\N_{\theta_0} \times \N_{h_0}$ of the true parameter. We specify the precise assumptions on this neighborhood in Section \ref{Properties of eta}.

Recall the penalized likelihood for estimating the parameter $(\theta, h)$ given by \eqref{PL}, and denote it as $\sum_{l=1}^m \mf{L}_{n_l,\lambda}(\theta,h)$. Let $\mf{L}_{n_l}(\theta,h) \equiv { -n_l^{-1} \sum_{i=1}^{n_l} \eta_l(X_{li};\theta,h) +\log \int_{\X_l} \exp\{\eta_l(x;\theta,h)\}dx }$.
If the joint negative log likelihood $\sum_{l=1}^{m} \mf{L}_{n_l}(\theta,h)$ is convex, under certain conditions (see Theorem~$2.9$ in \citeasnoun{Gu2013bk}), the existence of the penalized semiparametric estimator given by
\[(\hat{\theta}, \hat{h}) \equiv \argmin_{(\theta,h) \in \Q}\ \sum_{l=1}^m \mf{L}_{n_l,\lambda}(\theta,h)\]
is guaranteed. In general, the functions $\eta_l(\theta,h)$ need not satisfy these conditions, so for our analysis on joint consistency, we assume that $\sum_{l=1}^m \mf{L}_{n_l}(\theta,h)$ is convex and achieves a unique local minimum at $(\hat{\theta}, \hat{h})$ in $\N_{\theta_0} \times \N_{h_0}$. We study the consistency of this estimator and obtain the rate of convergence of $(\hat{\theta}, \hat{h})$ to the true parameters $(\theta_0,h_0)$ as $n = \min_{1\le l \le m} \{n_l\} \to \infty$ and $\lambda \to 0$. In our analysis, we may drop the subscript $l=1$ when $m=1$ for convenience.

 
Note that by adapting the methods in \citeasnoun{cox_1988}, \citeasnoun{CoxOsullivan}, and \citeasnoun{o’sullivan_1990}, we have established local existence and uniqueness of $(\hat{\theta}, \hat{h})$. A detailed proof can be found in the Appendix. 

\begin{subsection}{Outline of the proof of consistency} 
By our assumption for $\sum_{l=1}^m \mf{L}_{n_l,\lambda}(\theta,h)$, the estimator $\hat{\tau} = (\hat{\theta},\hat{h})$ is taken to be the unique solution of
\[\begin{cases}D_{\theta}\sum_{l=1}^{m}\mf{L}_{n_l,\lambda}(\theta,h) = 0\\
D_{h}\sum_{l=1}^{m}\mf{L}_{n_l,\lambda}(\theta,h) = 0
\end{cases}\]
in $\N_{\theta_0} \times \N_{h_0}$.
To study the asymptotic behavior of the estimator $\hat{\tau}=(\hat{\theta}, \hat{h})$, we follow a similar outline as in the nonparametric case \cite{gu_qiu_1993}. We first introduce an approximation of $\hat{\tau}$, denoted $\tilde{\tau} = (\tilde{\theta}, \tilde{h})$, which is the minimizer of 
\begin{align}
\sum_{l=1}^{m}\tilde{\mf{L}}_{n_l,\lambda}(\theta,h) &\equiv \sum_{l=1}^{m} \tilde{\mf{L}}_{n_l}(\theta,h) + \frac{\lambda}{2} J(h) \nonumber\\
&\equiv \sum_{l=1}^{m} \left\lbrace -\frac{1}{n_l} \sum_{i=1}^{n_l} \left[ D_{\theta}\eta_l(X_{li};\tau_0)\theta + D_{h}\eta_l(X_{li};\tau_0)h \right] + \mu_{l,\tau_0} \left[D_{\theta}\eta_l(\tau_0)\theta +D_{h}\eta_l(\tau_0)h \right]\right. \nonumber\\
& \quad \left. + \frac{1}{2} V_{l,\tau_0}\left[D_{\theta}\eta_l(\tau_0)(\theta-\theta_0)+D_{h}\eta_l(\tau_0)(h-h_0) \right] \right\rbrace + \frac{\lambda}{2}J(h),
\label{APL}
\end{align}
where $\mu_{l,\tau}(\varphi) = {\int_{\X_l} \varphi(t) \exp(\eta_l(t;\tau)) dt}/{\int_{\X_l} \exp(\eta_l(t;\tau)) dt}$, $V_{l,\tau}(\varphi_1, \varphi_2) = \mu_{l,\tau}(\varphi_1 \varphi_2) - \mu_{l,\tau}(\varphi_1) \mu_{l,\tau}(\varphi_2)$, and $V_{l,\tau}(\varphi) = V_{l,\tau}(\varphi, \varphi)$.
We see that for each $l=1,\ldots,m$, $\tilde{\mf{L}}_{n_l}(\theta,h)$ is almost like a quadratic approximation of $\mf{L}_{n_l}(\theta,h)$ at $\tau_0 = (\theta_0,h_0)$, ignoring terms independent of $(\theta,h)$ and terms involving second derivatives of $\eta_l(\theta,h)$. Since $V_{l,\tau_0}(\cdot)$ and $J(\cdot)$ are quadratic functionals, one can check that $\sum_{l=1}^{m}\tilde{\mf{L}}_{n_l,\lambda}(\theta,h)$ is convex with respect to $(\theta,h)$, and attains its minimum at $\tilde{\tau} = (\tilde{\theta}, \tilde{h})$ if and only if $\tilde{\tau}$ is the solution for the system
\[\begin{cases}D_{\theta}\sum_{l=1}^{m}\tilde{\mf{L}}_{n_l,\lambda}(\theta,h) = 0\\
D_{h}\sum_{l=1}^{m}\tilde{\mf{L}}_{n_l,\lambda}(\theta,h) = 0.
\end{cases}\]

Using a quadratic approximation of the penalized likelihood to attain an approximation of the estimator is a common intermediate step in the study of convergence of smoothing spline estimators (see \citeasnoun{silverman_1982}, \citeasnoun{CoxOsullivan}, \citeasnoun{o’sullivan_1990}, and \citeasnoun{gu_qiu_1993}).  In particular, when $m=1$, our definition of $\tilde{\mf{L}}_{n_1,\lambda}(\theta,h)$ can be seen as a generalized version of the quadratic approximation given by equation (5.2) in \citeasnoun{gu_qiu_1993}.  

In Section \ref{Norm and assumptions}, we define a norm related to $\sum_{l=1}^{m}V_{l,\tau_0}(\cdot)$ and $J(\cdot)$, in which the rate of convergence for $\hat{\tau} - \tau_0$ can be derived. For readability, we restrict our analysis to the special case when $m=1$ and establish the rate of convergence for $\tilde{\tau} - \tau_0$ in Section \ref{Linear approximation}. Together with a bound for the approximation error $\hat{\tau} - \tilde{\tau}$, the rate of convergence for $\hat{\tau} - \tau_0$ then follows by the triangle inequality (Section \ref{approximation error}). In addition, we also establish the rate of convergence of the estimate $\hat{\tau}$ in terms of the symmetrized Kullback-Leibler distance, which is defined as 
\begin{align}
\SKL(\tau_0, \hat{\tau}) &= \text{KL}(\tau_0,\hat{\tau}) + \text{KL}(\hat{\tau},\tau_0) =\E_{f_{\tau_0}} \log \frac{f_{\tau_0}}{f_{\hat{\tau}}} + \E_{f_{\hat{\tau}}} \log \frac{f_{\hat{\tau}}}{f_{\tau_0}} \nonumber\\
&=\mu_{\tau_0}[\eta(\theta_0,h_0) - \eta(\hat{\theta},\hat{h})] + \mu_{\hat{\tau}}[\eta(\hat{\theta},\hat{h}) - \eta(\theta_0,h_0)].
\label{SKL_def}
\end{align}  
Section \ref{result for general m} states results for $m \ge 1$ without proofs since the proofs are straightforward extensions of those in Sections \ref{Linear approximation} and \ref{approximation error}.
 
\end{subsection}

\begin{subsection}{Norm and assumptions}\label{Norm and assumptions}
In this section, we introduce an appropriate norm in which  the convergence rate of our estimator is measured and we provide some assumptions needed for our analysis. 

\begin{subsubsection}{Parameter space}\label{Parameter space}
\begin{assumption}
\label{A2}
The penalty $J(h)$ is a square seminorm in $\CH$ with a finite-dimensional null space $\CH_0 \subset \CH$. Therefore, $J((\theta,h)) \equiv J(h)$ extends $J$ to a square seminorm on $\Q$, and its null space $\R^{p} \times \CH_0$ is again finite-dimensional. Denote by $J(g,h)$ the semi-inner product associated with the seminorm $J(h)$. We also assume that $J(h_0) < \infty$.
\end{assumption}

\begin{assumption}
\label{A3}
For $l = 1, \ldots, m$, there are bounded linear operators $L_{l,\theta} : \R^p \to L^{2}_0(\X_l)$ and $L_{l,h} : \CH \to L^{2}_0(\X_l)$, with zero nullspaces, which satisfy the following conditions:  
\begin{enumerate}
\item[(i)] Suppose $\CH$ is a real Hilbert space of functions, equipped with norm $\norm{\cdot}$. For any $g \in \CH$, there exist positive constants $M_1, M_2$, such that 
\[M_1 \norm{g}^2 \leq \sum_{l=1}^{m} V_{l,\tau_0}(L_{l,h}g) + \lambda J(g) \leq M_2 \norm{g}^2.\]
\item[(ii)] For any $\zeta \in \R^p$ satisfying $\norm{\zeta}_{l^2} =1$ and for any $g \in \CH$, there exists a positive constant $c_{\delta}$ such that 
\begin{align}
\sum_{l=1}^{m} V_{l,\tau_0}(L_{l,\theta}\zeta - L_{l,h}g)=\sum_{l=1}^{m} V_{l,\tau_0}(L_{l,\theta}\zeta - L_{l,h}g, L_{l,\theta}\zeta - L_{l,h}g) > c_{\delta}.
\label{lower_bound_c_delta}
\end{align}

\end{enumerate}
\end{assumption}

By Assumptions \ref{A2} and \ref{A3}$(i)$, we see that $ \langle g_1,g_2 \rangle_{\CH} \equiv \sum_{l=1}^{m} V_{l,\tau_0}(L_{l,h}g_1, L_{l,h}g_2) + \lambda J(g_1,g_2)$
is an inner product on $\CH$, and its induced norm, denoted by $\norm{\cdot}_{\CH}$, is complete on $\CH$. One can also see that $L_{l,\theta}$ can be represented by the $p \times 1$ vector of $L^2_0(\X_l)$ functions $[ L_{l,\theta}^k(x)]_{k=1}^{p}$. We may use $L_{l,\theta}$ to denote the linear operator or its vector form, i.e., $L_{l,\theta}\zeta = \zeta^{T} L_{l,\theta}$. Denote by $V_{l,\tau_0}(L_{l,\theta},L_{l,\theta})$ the $ p \times p$ matrix in which the $(i,j)$th entry is $V_{l,\tau_0}(L_{l,\theta}^{i}(x),L_{l,\theta}^{j}(x))$, and note that one can write $\sum_{l=1}^{m} V_{l,\tau_0}(L_{l,\theta}\zeta,L_{l,\theta}\zeta) = \zeta^{T} \sum_{l=1}^{m} V_{l,\tau_0}(L_{l,\theta},L_{l,\theta})\zeta$. When $g=0$, \eqref{lower_bound_c_delta} implies that $\sum_{l=1}^{m} V_{l,\tau_0}(L_{l,\theta},L_{l,\theta})$ is positive definite. Therefore, $\langle \zeta_1,\zeta_2 \rangle_{\R^p} \equiv \sum_{l=1}^{m} V_{l,\tau_0}(L_{l,\theta}\zeta_1,L_{l,\theta}\zeta_2)$ is an inner product on $\R^p$ and we use $\norm{\cdot}_{\R^p}$ to denote its induced norm. 

For any $(\zeta_1, g_1), (\zeta_2,g_2) \in \Q$, we define an inner product as follows,
\begin{equation}
\label{inner_product}
\langle (\zeta_1,g_1),(\zeta_2,g_2) \rangle_{\Q} \equiv \sum_{l=1}^{m} V_{l,\tau_0}(L_{l,\theta}\zeta_1 + L_{l,h}g_1, L_{l,\theta}\zeta_2 + L_{l,h}g_2) + \lambda J(g_1,g_2),
\end{equation}
and we denote the norm induced by this inner product by $\norm{\cdot}_{\Q}$. In the Appendix, we prove that $\langle \cdot, \cdot \rangle_{\Q}$ as defined above is, in fact, an inner product, and that its induced norm is complete.

\begin{remark} \ 
\label{remark1}
\begin{enumerate}
\item[(i)] We need to use the auxiliary operators $L_{l,\theta}$ and $L_{l,h}$ to define the norm in which the rate of convergence can be conveniently measured. We point out here that $L_{l,\theta}$ and $L_{l,h}$ are related to the Fr\'echet partial derivatives of $\eta_l$ by Assumption \ref{A6} in Section \ref{Properties of eta}. This is analogous to the approach in \citeasnoun{o’sullivan_1990}.
\item[(ii)] It is well known that all norms on $\R^p$, including $\norm{\cdot}_{\R^p}$ as defined above, are equivalent to the Euclidean norm, namely $\norm{\zeta}_{l^2} = (\sum_{i=1}^p \abs{\zeta^i}^2)^{1/2}$ (Theorem 3.1 in \citeasnoun{conway1990}). Therefore $\norm{\cdot}_{\R^p}$ is complete. Moreover, the convergence of $\hat{\theta}$ to $\theta_0$ in $\norm{\cdot}_{\R^p}$ implies the convergence in the Euclidean norm (see Corollary \ref{cor1}). 
\item[(iii)] Assumption \ref{A3}(ii) is a regularity condition similar to Assumption A3.\ in \citeasnoun{cheng2015joint}. This assumption guarantees that $\ip{\cdot,\cdot}_{\Q}$ is indeed an inner product, and its induced norm is complete (see Theorem A.\ref{inner_product_theorem} in the Appendix).
\item[(iv)]  We note that under Assumption \ref{A3}, for any $(\zeta,g) \in \Q$, $\norm{(\zeta,g)}_{\Q,1} \equiv \norm{\zeta}_{\R^p} + \norm{g}_{\CH}$ defines a norm on the product space $\Q$. By a similar argument as in the proof of Theorem A.\ref{inner_product_theorem} in the Appendix, a sequence $\{(\zeta_i,g_i)\}_{i=1}^{\infty} \in \Q$ converges in norm $\norm{\cdot}_{\Q}$ if and only if it converges in norm $\norm{\cdot}_{\Q,1}$. This implies that $\norm{\cdot}_{\Q}$ and $\norm{\cdot}_{\Q,1}$ are equivalent norms on $\Q$.
\end{enumerate}
\end{remark}

\end{subsubsection}

\begin{subsubsection}{Properties of $\eta(\theta, h) \text{ and } (\theta_0, h_0)$} \label{Properties of eta}

Recall $\tau_0 = (\theta_0, h_0)$ is the true parameter in $\Q$. We assume there are neighborhoods $\N_{\theta_0} \subset \R^p$ of $\theta_0$ and $\N_{h_0} \subset \CH$ of $h_0$ such that the following Assumptions \ref{A4} to \ref{A7} hold.

\begin{assumption}
\label{A4}
For all $l = 1, \ldots, m$, $\eta_l(\theta,h)$ is one-to-one and three times continuously Fr\'echet differentiable with respect to $(\theta,h)$ in $\N_{\theta_0} \times \N_{h_0}$. Moreover, $\tau_0 = (\theta_0, h_0)$ is the unique root of ${\E\left[D_{\theta}\sum_{l=1}^{m} \mf{L}_{n_l}(\theta,h)\right]= 0}$ and ${\E\left[D_{h}\sum_{l=1}^{m} \mf{L}_{n_l}(\theta,h)\right]=0}$
in $\N_{\theta_0} \times \N_{h_0}$. 
\end{assumption}

\begin{assumption}
\label{A6}
For any $\theta_{*} \in \N_{\theta_0}, h_{*} \in \N_{h_0}, l \in \{1, \ldots, m\}$, $D_{\theta} \eta_l(\theta_*,h_*)$ is a bounded linear operator from $\R^p$ to $L^2_0(\X_l)$, $D_{h} \eta_l(\theta_*,h_*)$ is a bounded linear operator from $\CH$ to $L^2_0(\X_l)$, and there exist positive constants $C_1,C_2$, such that for all $(\zeta,g) \in \Q$,
\begin{align*}
C_1 \sum_{l=1}^{m} V_{l,\tau_0}(L_{l,\theta} \zeta + L_{l,h} g) &\le \sum_{l=1}^{m}V_{l,\tau_0}( D_{\theta} \eta_l(\theta_*,h_*) \zeta + D_{h} \eta_l(\theta_*,h_*) g)\le C_2 \sum_{l=1}^{m}V_{l,\tau_0}(L_{l,\theta} \zeta + L_{l,h} g).
\end{align*}
\end{assumption}

\begin{assumption}
\label{Ad}
For any $(\theta_1, h_1), (\theta_2, h_2) \in \N_{\theta_0} \times \N_{h_0}$, there exists a positive constant $C_{d} < 2C_1$, where $C_1$ is as defined in Assumption \ref{A6}, such that for any $(\zeta,g) \in \N_{\theta_0} \times \N_{h_0}$,
\begin{align*}
&\sum_{l=1}^{m} V_{l,\tau_0} \left[ (D_{\theta} \eta_l(\theta_1,h_1) - D_{\theta} \eta_l(\theta_2,h_2)) \zeta + (D_{h} \eta_l(\theta_1,h_1)-D_{h} \eta_l(\theta_2,h_2)) g \right] \le C_d \sum_{l=1}^{m} V_{l,\tau_0}(L_{l,\theta} \zeta + L_{l,h} g).
\end{align*}

\end{assumption}

\begin{assumption}
\label{A7}
$\N_{\theta_0} \times \N_{h_0}$ is a convex set that contains $\hat{\tau}$, and there exist $C_3,C_4 >0$ such that
\[C_3 \sum_{l=1}^{m}V_{l,\tau_0}(f) \le \sum_{l=1}^{m}V_{l,\tau}(f) \le C_4 \sum_{l=1}^{m}V_{l,\tau_0}(f)\]
holds uniformly for any $\tau=(\theta,h) \in \N_{\theta_0} \times \N_{h_0}$ and $f \in L^2_0(\X_l)$.
\end{assumption}

\begin{remark} 
We compare the assumptions above to the existing literature on the single sample ($m=1$) nonparametric model and the semiparametric additive model. 
\label{remark2}
\begin{enumerate}
\item[(i)]For the nonparametric model where $\eta(x;\theta,h) = h(x)$, and for the semiparametric additive model where $\eta(x;\theta,h) = \alpha(x;\theta) + h(x)$, we see that $\CH \subset L_0^2(\X)$ and $L_h$ can be chosen to be the inclusion operator $\iota: \CH \to L_0^2(\X)$. It is easy to see that $V_{\tau_0}(g_1,g_2)$ defines an inner product on $L_0^2(\X)$. Moreover, the norm $\norm{g}_{\CH} = [V_{\tau_0}(g,g)+\lambda J(g,g)]^{1/2}$ and its analog for regression models have been widely used in the smoothing spline literature for nonparametric models, including \citeasnoun{silverman_1982}, \citeasnoun{cox_1988}, \citeasnoun{gu_qiu_1993}, \citeasnoun{shang2010convergence}, etc.
\item[(ii)] Consider the linear additive model $\eta(x;\theta,h) = \theta^T a(x) + h(x)$, where $a(x) = [a^k(x)]_{k=1}^p$ is a vector of bounded $L^2_0(\X)$ functions. One may choose $L_\theta = D_{\theta}\eta(\theta,h)= a(x)$ and $L_h = D_{h}\eta(\theta,h) = \iota$, where $\iota$ is the inclusion operator from $\CH$ to $L^2_0(\X)$. Since $V_{\tau_0}(\theta^T a - h)$ measures the variance of the difference between the parametric and nonparametric components, identifiability of this model follows from Assumption \ref{A3}(ii). 
\item[(iii)] As discussed above, by choosing the appropriate operators $L_\theta$ and $L_h$ for either the nonparametric model or the semiparametric linear additive model, our analysis is reduced to that studied by \citeasnoun{gu_qiu_1993} or \citeasnoun{yang2009penalized}, respectively. In both cases, due to the linearity of $\eta$, the neighborhood $\N_{\theta} \times \N_{h}$ can be taken to be the whole parameter space. Furthermore, Assumptions \ref{A4} to \ref{Ad} are satisfied automatically. One may see that such assumptions are simply redundant when $\eta$ is linear in $\theta$ and $h$.
\item[(iv)] Note that Assumption \ref{A7} is similar to Assumption A.3 in \citeasnoun{gu_qiu_1993}, Assumption A.3 in \citeasnoun{gu1995smoothing}, and Condition 3 in \citeasnoun{yang2009penalized}. As mentioned by \citeasnoun{gu1995smoothing}, $V_{\tau}(f-g)$ can be viewed as a kind of weighted mean square error between functions $f$ and $g$ with weight function $\exp\{\eta(x;\tau)\}$. Since $\eta$ is continuous in $\N_{\theta_0} \times \N_{h_0}$, a small change in $\tau$ yields a small change in the weight function, and Assumption \ref{A7} simply guarantees that this also yields a relatively small change in the weighted mean square error.
\end{enumerate}
\end{remark}

\end{subsubsection}

\begin{subsubsection}{Spectral decomposition}\label{Spectral decomposition}
We now construct an eigensystem for the functionals $g \mapsto \sum_{l=1}^{m} V_{l,\tau_0}(L_{l,h} g)$ and  $J$. 
\begin{assumption}
\label{A8}
The quadratic functional $g \mapsto \sum_{l=1}^{m} V_{l,\tau_0}(L_{l,h} g)$ is completely continuous with respect to the quadratic functional $J$. 
\end{assumption}

Under Assumption \ref{A8}, Theorem 3.1 of \citeasnoun{weinberger_1974}
yields a sequence $\lbrace \phi_{\nu} : \nu = 1,2,\ldots \rbrace$ of eigenfunctions and a sequence $\lbrace \rho_{\nu} : \nu = 1,2,\ldots \rbrace$ of eigenvalues such that 
\[ \sum_{l=1}^{m} V_{l,\tau_0}(L_{l,h} \phi_{\nu}, L_{l,h} \phi_{\mu}) = \delta_{\nu \mu} \quad\text{ and }\quad J(\phi_{\mu},\phi_{\nu}) = \rho_{\nu} \delta_{\nu \mu}\]
for all pairs $\nu, \mu$ of positive integers, where $\delta_{\nu \mu}$ is the Kronecker delta and $0 \le \rho_{\nu} \rightarrow \infty$. 

\begin{assumption}
\label{A9}
$\rho_{\nu} = \kappa_{\nu} \nu^{r}$, where $r > 1$ and $\kappa_{\nu} \in (\beta_1,\beta_2) \subset (0, \infty)$.
\end{assumption}

By Assumptions \ref{A6}  and \ref{A8}, for any $(\theta_{*}, h_{*}) \in \N_{\theta_0} \times \N_{h_0}$, there exist sequences of eigenfunctions $\lbrace \phi_{*,\nu} : \nu = 1,2,\ldots \rbrace$ and eigenvalues $\lbrace \rho_{*,\nu} : \nu = 1,2,\ldots \rbrace$ such that 
\[ \sum_{l=1}^{m} V_{l,\tau_0}(D_{h} \eta_l(\theta_*,h_*) \phi_{*,\nu}, D_{h} \eta_l(\theta_*,h_*) \phi_{*,\mu}) = \delta_{\nu \mu} \quad\text{ and }\quad J(\phi_{*,\mu},\phi_{*,\nu}) = \rho_{*,\nu} \delta_{\nu \mu}\]
for all pairs $\nu, \mu$ of positive integers, where $0 \le \rho_{*,\nu} \rightarrow \infty$. Assumption \ref{A6} implies that there exist positive constants $c_1,c_2$ such that for $\nu$ large enough,
$ c_1 \rho_{\nu} \le \rho_{*,\nu} \le c_2 \rho_{\nu}$.
By Assumption \ref{A9}, $\rho_{*,\nu} \sim \nu^{r}$ for large enough $\nu$, where $r>1$.

Note that for every $h \in \CH$ and any $(\theta_{*}, h_{*}) \in \N_{\theta_0} \times \N_{h_0}$, we can have a Fourier expansion 
\[h = \sum_{\nu=1}^{\infty} \sum_{l=1}^{m} V_{l, \tau_0}(D_{h} \eta_l(\theta_*,h_*)h, D_{h} \eta_l(\theta_*,h_*) \phi_{*,\nu}) \phi_{*,\nu}.\]
 \end{subsubsection}
 
\end{subsection}

\begin{subsection}{Linear approximation}\label{Linear approximation}
For this section and the next, we present our analysis and provide the rate of convergence for the single sample case when $m=1$. 

Consider the eigensystem discussed in Section \ref{Spectral decomposition} with $(\theta_*,h_*) = (\theta_0, h_0)$. We have the Fourier series expansions $h = \sum_{\nu} h_{\nu} \phi_{0,\nu}$ and $h_0 = \sum_{\nu} h_{0,\nu} \phi_{0,\nu}$, where 
\[h_{\nu} = V_{\tau_0}(D_{h} \eta(\tau_0)h, D_{h} \eta(\tau_0)\phi_{0,\nu}) \quad\text{ and }\quad h_{0,\nu} = V_{\tau_0}(D_{h} \eta(\tau_0)h_0, D_{h} \eta(\tau_0)\phi_{0,\nu})\]
are the Fourier coefficients of $h$ and $h_0$ with respect to the base $\phi_{0,\nu}$. Plugging these into equation \eqref{APL}, we get 
\begin{align}
\tilde{\mf{L}}_{n,\lambda}(\theta,\sum_{\nu} h_{\nu} \phi_{0,\nu}) 
&= -\theta^{T} \brac{\frac{1}{n} \sum_{i=1}^{n} D_{\theta}\eta(X_i;\tau_0) - \mu_{\tau_0} \left[D_{\theta}\eta(\tau_0) \right]} \nonumber \\
&\quad - \sum_{\nu} h_{\nu} \brac{\frac{1}{n} \sum_{i=1}^{n} D_{h}\eta(X_i;\tau_0)\phi_{0,\nu} - \mu_{\tau_0} \left[D_{h}\eta(\tau_0)\phi_{0,\nu} \right]} \nonumber\\
&\quad + \frac{1}{2}(\theta-\theta_0)^{T}V_{\tau_0}\left[D_{\theta}\eta(\tau_0)\right](\theta-\theta_0) + \frac{1}{2}\sum_{\nu}(h_{\nu} - h_{0,\nu})^{2} \nonumber\\
&\quad + \sum_{\nu} (h_{\nu} - h_{0,\nu})(\theta-\theta_0)^{T}V_{\tau_0}\left[D_{\theta}\eta(\tau_0), D_{h}\eta(\tau_0)\phi_{0,\nu} \right] \nonumber \\
&\quad + \frac{\lambda}{2} \sum_{\nu} \rho_{0,\nu} h_{\nu}^2.
\label{FAPL}
\end{align}
In equation \eqref{FAPL}, using the fact that $D_{\theta}\eta(x;\tau_0)$ can be represented by a $p \times 1$ vector  whose $k$th entry is $D_{\theta^k}\eta(x;\tau_0)$, we denote by $D_{\theta}\eta(x;\tau_0)$ both the linear operator and its vector form, i.e., $D_{\theta}\eta(x;\tau_0)\theta = \theta^{T}D_{\theta}\eta(x;\tau_0)$.  We also have $\mu_{\tau_0}[D_{\theta}\eta(\tau_0)] = [ \mu_{\tau_0} [D_{\theta^1}\eta(\tau_0)],\ldots, \mu_{\tau_0}[D_{\theta^p}\eta(\tau_0)]] ^T$,
$V_{\tau_0}[D_{\theta}\eta(\tau_0)]$ is the $p \times p$ covariance matrix whose $(i,j)$th entry is $V_{\tau_0}[D_{\theta^i}\eta(\tau_0),D_{\theta^j}\eta(\tau_0)]$, and $V_{\tau_0}[D_{\theta}\eta(\tau_0), D_{h}\eta(\tau_0)\phi_{0,\nu}]$ is the $p \times 1$ vector whose $i$th entry is $V_{\tau_0}[D_{\theta^i}\eta(\tau_0), D_{h}\eta(\tau_0)\phi_{0,\nu}]$.
Let
\begin{align*}
\bar{\alpha}_n = \frac{1}{n}\sum_{i=1}^{n} D_{\theta}\eta(X_i;\tau_0) - \mu_{\tau_0} \left[D_{\theta}\eta(\tau_0) \right],\quad \bar{\beta}_{\nu,n} = \frac{1}{n} \sum_{i=1}^{n} D_{h}\eta(X_i;\tau_0)\phi_{0,\nu} - \mu_{\tau_0} \left[D_{h}\eta(\tau_0)\phi_{0,\nu} \right].
\end{align*}
The Fourier coefficients $\tilde{h}_{\nu}$ and $\tilde{\theta}$ that minimize equation \eqref{FAPL} are therefore given by setting the partial derivatives of equation \eqref{FAPL} to 0 and solving the resulting system. We get
\begin{align*}
\tilde{\theta} &= \theta_0 + \Omega_{\lambda}^{-1} \brac{\bar{\alpha}_n - \sum_{\nu}\frac{\bar{\beta}_{\nu,n} - \lambda \rho_{0,\nu} h_{0,\nu}}{1+\lambda \rho_{0,\nu}}V_{\tau_0}[D_{\theta}\eta(\tau_0),D_{h}\eta(\tau_0)\phi_{0,\nu}]},\\
\tilde{h}_{\nu} &= \frac{\bar{\beta}_{\nu,n} + h_{0,\nu}}{1+\lambda \rho_{0,\nu}} - (\tilde{\theta}-\theta_0)^{T} \frac{V_{\tau_0}[D_{\theta}\eta(\tau_0),D_{h}\eta(\tau_0)\phi_{0,\nu}]}{1+\lambda \rho_{0,\nu}},
\end{align*}
where
\[\Omega_{\lambda} =  V_{\tau_0}[D_{\theta}\eta(\tau_0)] - \sum_{\nu} \frac{V_{\tau_0}[D_{\theta}\eta(\tau_0),D_{h}\eta(\tau_0)\phi_{0,\nu}]^{\otimes 2}}{1+\lambda \rho_{0,\nu}},\]
and $a^{\otimes 2} = a a^T$ ($a$ is a vector or matrix).
Therefore, $(\tilde{\theta},\tilde{h})$, where $\tilde{h} = \sum_{\nu} \tilde{h}_{\nu} \phi_{0,\nu}$, is the minimizer of equation \eqref{APL}. Using the terminology from the nonparametric setting \cite{Gu2013bk}, $\tilde{h}$ can be called a linear approximation of $\hat{h}$ since $\tilde{h}_{\nu}$ is linear in $\phi_{0,\nu}(X_i)$ given $\tilde{\theta}$. We will slightly abuse the terminology for our nonlinear case here and also call $\tilde{\tau} = (\tilde{\theta},\tilde{h})$ the linear approximation of $\hat{\tau} = (\hat{\theta},\hat{h})$.

Using the fact that $\E(\bar{\alpha}_n) = \E(\bar{\beta}_{\nu,n}) = 0$, $\E(\Vert \bar{\alpha}_n \Vert_{l^2}^2) = O(n^{-1})$, $\E(\bar{\beta}_{\nu,n}^2) =n^{-1}$, and after tedious calculation, we get the following lemma and theorem, whose proofs are given in the Appendix.

\begin{lemma}\label{lemma1}
Under Assumptions \ref{A2} to \ref{A4}, \ref{A8}, and \ref{A9}, as $n \rightarrow \infty$ and $\lambda \rightarrow 0$,
\begin{align*}
&\E\brac{ V_{\tau_0} [D_{\theta} \eta(\tau_0) (\tilde{\theta} - \theta_0)]} \leq c \E \left[ (\tilde{\theta}-\theta_0)^{T}(\tilde{\theta}-\theta_0) \right] = O(n^{-1}\lambda^{-\frac{1}{r}}),\\
&\E\brac{ V_{\tau_0}[D_{h}\eta(\tau_0)(\tilde{h}-h_0)] + \lambda J(\tilde{h}-h_0)} = O(n^{-1}\lambda^{-\frac{1}{r}} + \lambda).
\end{align*}
\end{lemma}

\begin{theorem}\label{theorem1}
Under Assumptions \ref{A2} to \ref{A6}, \ref{A8}, and \ref{A9}, as $n \rightarrow \infty$ and $\lambda \rightarrow 0$, 
\[V_{\tau_0}[L_{\theta} (\tilde{\theta} - \theta_0) + L_h(\tilde{h}-h_0)] + \lambda J(\tilde{h}-h_0) = O_{p}(n^{-1}\lambda^{-\frac{1}{r}} + \lambda).\]
\end{theorem}

Note that in addition to $n \rightarrow \infty$ and $\lambda \rightarrow 0$, if also $n^{-1}\lambda^{-\frac{1}{r}} \rightarrow 0$, then the probability that $\tilde{\tau} \in \N_{\theta_0} \times \N_{h_0}$ tends to $1$ (since $\tilde{\tau} \rightarrow \tau_0$ under these conditions). We can restrict our attention to this event for the rest of our analysis, or for simplicity, we assume that $\tilde{\tau} \in \N_{\theta_0} \times \N_{h_0}$.

\end{subsection}

\begin{subsection}{Approximation error and main results for case $m=1$ }\label{approximation error}

In this section, we first find a bound for the approximation error $\hat{\tau} - \tilde{\tau} = (\hat{\theta}-\tilde{\theta}, \hat{h}-\tilde{h})$, which will then imply the convergence of $\hat{\tau} - \tau_0 = (\hat{\theta} - \theta_0, \hat{h}-h_0)$. 

Define
\begin{align*}
A^{\theta}_{\tau_1,\tau_2}(\alpha) &\equiv \mf{L}_n(\theta_1+\alpha\theta_2, h_1) + \frac{\lambda}{2}J(h_1), &
A^{h}_{\tau_1,\tau_2}(\alpha) &\equiv \mf{L}_n(\theta_1, h_1 + \alpha h_2) +\frac{\lambda}{2}J(h_1+\alpha h_2),\\
B^{\theta}_{\tau_1,\tau_2}(\alpha) &\equiv \tilde{\mf{L}}_n(\theta_1+\alpha\theta_2, h_1) + \frac{\lambda}{2}J(h_1), &
B^{h}_{\tau_1,\tau_2}(\alpha) &\equiv \tilde{\mf{L}}_n(\theta_1, h_1 + \alpha h_2) +\frac{\lambda}{2}J(h_1+\alpha h_2).
\end{align*}
By straightforward calculations, we have 
\begin{align}
\dot{A}^{\theta}_{\tau_1,\tau_2}(0) &= \frac{dA^{\theta}_{\tau_1,\tau_2}}{d\alpha}(0) = -\frac{1}{n}\sum_{i=1}^{n} D_{\theta}\eta(X_i;\tau_1)\theta_2 + \mu_{\tau_1} [D_{\theta}\eta(\tau_1)\theta_2], \label{dAp}\\
\dot{A}^{h}_{\tau_1,\tau_2}(0) &= \frac{dA^{h}_{\tau_1,\tau_2}}{d\alpha}(0) = -\frac{1}{n}\sum_{i=1}^{n} D_{h}\eta(X_i;\tau_1)h_2 + \mu_{\tau_1} [D_{h}\eta(\tau_1)h_2]+ \lambda J(h_1,h_2), \label{dAn}\\
\dot{B}^{\theta}_{\tau_1,\tau_2}(0)&=\frac{dB^{\theta}_{\tau_1,\tau_2}}{d\alpha}(0) = -\frac{1}{n} \sum_{i=1}^{n} D_{\theta}\eta(X_i;\tau_0)\theta_2 + \mu_{\tau_0} [D_{\theta}\eta(\tau_0)\theta_2] \nonumber \\
&\quad + V_{\tau_0}[D_{\theta}\eta(\tau_0)(\theta_1-\theta_0), D_{\theta}\eta(\tau_0)\theta_2] + V_{\tau_0}[D_{h}\eta(\tau_0)(h_1-h_0), D_{\theta}\eta(\tau_0)\theta_2], \label{dBp}\\
\dot{B}^{h}_{\tau_1,\tau_2}(0) &=\frac{dB^{h}_{\tau_1,\tau_2}}{d\alpha}(0) = -\frac{1}{n} \sum_{i=1}^{n} D_{h}\eta(X_i;\tau_0)h_2 + \mu_{\tau_0}[D_{h}\eta(\tau_0)h_2] \nonumber \\
&\quad + V_{\tau_0}[D_{h}\eta(\tau_0)(h_1-h_0), D_{h}\eta(\tau_0)h_2] + V_{\tau_0}[D_{\theta}\eta(\tau_0)(\theta_1-\theta_0), D_{h}\eta(\tau_0)h_2] +\lambda J(h_1,h_2). \label{dBn}
\end{align}
Equation \eqref{dAp} (or equation \eqref{dAn}) can be understood as the Fr\'echet partial derivative of $\mf{L}_{n,\lambda}(\theta,h)$ with respect to $\theta$ (or $h$) at $\tau_1 = (\theta_1, h_1)$ in the direction of $\theta_2$ (or $h_2$). Likewise, equation \eqref{dBp} (or equation \eqref{dBn}) is the Fr\'echet partial derivative of $\tilde{\mf{L}}_{n,\lambda}(\theta,h)$ with respect to $\theta$ (or $h$) at $\tau_1 = (\theta_1, h_1)$ in the direction of $\theta_2$ (or $h_2$). Therefore, for any $\tau_2$, we note that $\dot{A}^{\theta}_{\hat{\tau},\tau_2}(0) = \dot{A}^{h}_{\hat{\tau},\tau_2}(0) = 0$, and $\dot{B}^{\theta}_{\tilde{\tau},\tau_2}(0) = \dot{B}^{h}_{\tilde{\tau},\tau_2}(0) = 0$.

Setting $\tau_1 = \hat{\tau}$, $\tau_2 = \hat{\tau} - \tilde{\tau}$ in \eqref{dAp} + \eqref{dAn}, we get 
\begin{align}
-\frac{1}{n}&\sum_{i=1}^{n} \left[D_{\theta}\eta(X_i;\hat{\tau})(\hat{\theta} - \tilde{\theta})+D_{h}\eta(X_i;\hat{\tau})(\hat{h} - \tilde{h})\right] + \mu_{\hat{\tau}} [D_{\theta}\eta(\hat{\tau})(\hat{\theta} - \tilde{\theta})+D_{h}\eta(\hat{\tau})(\hat{h} - \tilde{h})] \nonumber\\
&+ \lambda J(\hat{h},\hat{h} - \tilde{h}) = 0, 
\label{dA0}
\end{align}
and setting $\tau_1 = \tilde{\tau}$, $\tau_2 = \hat{\tau} - \tilde{\tau}$ in  \eqref{dBp} + \eqref{dBn} implies
\begin{align}
-\frac{1}{n}& \sum_{i=1}^{n} \left[D_{\theta}\eta(X_i;\tau_0)(\hat{\theta} - \tilde{\theta})+D_{h}\eta(X_i;\tau_0)(\hat{h} - \tilde{h})\right] + \mu_{\tau_0}[D_{\theta}\eta(\tau_0)(\hat{\theta} - \tilde{\theta})+D_{h}\eta(\tau_0)(\hat{h} - \tilde{h})] \nonumber\\
&+ V_{\tau_0}[D_{\theta}\eta(\tau_0)(\tilde{\theta}-\theta_0)+D_{h}\eta(\tau_0)(\tilde{h}-h_0), D_{\theta}\eta(\tau_0)(\hat{\theta} - \tilde{\theta})+D_{h}\eta(\tau_0)(\hat{h} - \tilde{h})] \nonumber \\
&+\lambda J(\tilde{h},\hat{h} - \tilde{h})=0.
\label{dB0}
\end{align}
Equating \eqref{dA0} and \eqref{dB0}, rearranging terms, and subtracting ${\mu_{\tilde{\tau}}[D_{\theta}\eta(\hat{\tau})(\hat{\theta} - \tilde{\theta})+D_{h}\eta(\hat{\tau})(\hat{h} - \tilde{h})]}$ on both sides, we have 
\begin{align}
&\mu_{\hat{\tau}} [D_{\theta}\eta(\hat{\tau})(\hat{\theta} - \tilde{\theta})+D_{h}\eta(\hat{\tau})(\hat{h} - \tilde{h})]- \mu_{\tilde{\tau}}[D_{\theta}\eta(\hat{\tau})(\hat{\theta} - \tilde{\theta})+D_{h}\eta(\hat{\tau})(\hat{h} - \tilde{h})] + \lambda J(\hat{h} - \tilde{h},\hat{h} - \tilde{h}) \nonumber\\
& = V_{\tau_0}[D_{\theta}\eta(\tau_0)(\tilde{\theta}-\theta_0)+ D_{h}\eta(\tau_0)(\tilde{h}-h_0), D_{\theta}\eta(\tau_0)(\hat{\theta}-\tilde{\theta})+D_{h}\eta(\tau_0)(\hat{h} - \tilde{h})] \nonumber\\
& + \brac{\mu_{\tau_0}[D_{\theta}\eta(\hat{\tau})(\hat{\theta} - \tilde{\theta})+ D_{h}\eta(\hat{\tau})(\hat{h} - \tilde{h})] - \mu_{\tilde{\tau}}[D_{\theta}\eta(\hat{\tau})(\hat{\theta} - \tilde{\theta})+D_{h}\eta(\hat{\tau})(\hat{h} - \tilde{h})]} \nonumber\\
& + \brac{\frac{1}{n}\sum_{i=1}^{n} \left[ D_{\theta}\eta(X_i;\hat{\tau})(\hat{\theta} - \tilde{\theta})+D_{h}\eta(X_i;\hat{\tau})(\hat{h} - \tilde{h}) \right] - \mu_{\tau_0}[D_{\theta}\eta(\hat{\tau})(\hat{\theta} - \tilde{\theta}) + D_{h}\eta(\hat{\tau})(\hat{h} - \tilde{h})]} \nonumber\\
& -\brac{\frac{1}{n} \sum_{i=1}^{n}  \left[D_{\theta}\eta(X_i;\tau_0)(\hat{\theta}-\tilde{\theta})+ D_{h}\eta(X_i;\tau_0)(\hat{h} - \tilde{h})\right] - \mu_{\tau_0}[ D_{\theta}\eta(\tau_0)(\hat{\theta}-\tilde{\theta})+ D_{h}\eta(\tau_0)(\hat{h} - \tilde{h})]}. 
\label{III} 
\end{align}

Note that for any function $f \in L^2_0(\X)$ and $\tau_1, \tau_2 \in \N_{\theta_0} \times \N_{h_0}$, it is easy to show using the mean value theorem that
\begin{align}
\label{MVT}
\mu_{\tau_2}[f] - \mu_{\tau_1}[f] = V_{\tau_1 + \alpha(\tau_2-\tau_1)}[f,
D_{\theta} \eta(\tau_1+\alpha(\tau_2-\tau_1))(\theta_2-\theta_1) +  D_{h} \eta(\tau_1+\alpha(\tau_2-\tau_1))(h_2-h_1)]
\end{align}
for some $0 \leq \alpha \leq 1$. We are now ready to prove the following theorem.

\begin{theorem}\label{theorem2}
Under Assumptions \ref{A2} to \ref{A9}, as $n \rightarrow \infty$, $\lambda \rightarrow 0$, and $n\lambda^{\frac{1}{r}} \rightarrow \infty$,
\[V_{\tau_0} [ L_{\theta} (\hat{\theta} - \tilde{\theta}) + L_{h}(\hat{h}-\tilde{h} )] + \lambda J(\hat{h}-\tilde{h}) = O_{p}\left(n^{-1}\lambda^{-\frac{1}{r} }+ \lambda\right).\]
Therefore,
\[V_{\tau_0} [ L_{\theta} (\hat{\theta} - \theta_0) + L_{h}(\hat{h}-h_0 )] + \lambda J(\hat{h}-h_0) = O_{p}\left(n^{-1}\lambda^{-\frac{1}{r} }+ \lambda\right).\]
\end{theorem}

\begin{proof}
We first obtain a lower bound for the first two terms in the LHS of \eqref{III} and an upper bound for the second term in the RHS of \eqref{III}. For some $0 \leq \alpha \leq 1$, let $\tau_*= \tilde{\tau} + \alpha (\hat{\tau}-\tilde{\tau})$. By Assumptions \ref{A6}, \ref{Ad}, and \ref{A7}, and equation \eqref{MVT}, the first two terms in the LHS of \eqref{III} yield
\begin{align*}
&2 \brac{ \mu_{\hat{\tau}}[D_{\theta}\eta(\hat{\tau})(\hat{\theta} - \tilde{\theta})+D_{h}\eta(\hat{\tau})(\hat{h} - \tilde{h})]- \mu_{\tilde{\tau}}[D_{\theta}\eta(\hat{\tau})(\hat{\theta} - \tilde{\theta})+D_{h}\eta(\hat{\tau})(\hat{h} - \tilde{h})] }\\
&= 2 V_{\tau_*} [D_{\theta}\eta(\hat{\tau})(\hat{\theta} - \tilde{\theta})+D_{h}\eta(\hat{\tau})(\hat{h} - \tilde{h}), D_{\theta} \eta(\tau_*)(\hat{\theta}-\tilde{\theta}) +  D_{h} \eta(\tau_*)(\hat{h}-\tilde{h})]\\ 
&= V_{\tau_*} [D_{\theta}\eta(\hat{\tau})(\hat{\theta} - \tilde{\theta})+D_{h}\eta(\hat{\tau})(\hat{h} - \tilde{h})] + V_{\tau_*} [D_{\theta} \eta(\tau_*)(\hat{\theta}-\tilde{\theta}) +  D_{h} \eta(\tau_*)(\hat{h}-\tilde{h})] \\
& \quad \quad - V_{\tau_*}[\left(D_{\theta}\eta(\hat{\tau})- D_{\theta} \eta(\tau_*)\right)(\hat{\theta} - \tilde{\theta})+ \left(D_{h}\eta(\hat{\tau}) - D_{h} \eta(\tau_*)\right)(\hat{h} - \tilde{h})]\\
&\ge C_3 \left\{ 2 C_1 V_{\tau_0}[L_{\theta}(\hat{\theta} - \tilde{\theta})+L_h(\hat{h} - \tilde{h})] - C_d V_{\tau_0} [L_{\theta}(\hat{\theta} - \tilde{\theta})+L_h(\hat{h} - \tilde{h})] \right\}\\
&\ge c_3 V_{\tau_0} [L_{\theta}(\hat{\theta} - \tilde{\theta})+L_h(\hat{h} - \tilde{h})],
\end{align*}
for some $0 \le c_3 \le C_3 (2C_1 - C_d)$. 
For the second term in the RHS of \eqref{III}, let $\tau_u= \tilde{\tau} + u (\tau_0-\tilde{\tau})$ for some $0 \leq u \leq 1$. By equation \eqref{MVT}, the Cauchy-Schwarz inequality, and Assumptions \ref{A6} and \ref{A7}, we have 
\begin{align*}
&\mu_{\tau_0}[D_{\theta}\eta(\hat{\tau})(\hat{\theta} - \tilde{\theta})+ D_{h}\eta(\hat{\tau})(\hat{h} - \tilde{h})] - \mu_{\tilde{\tau}} [D_{\theta}\eta(\hat{\tau})(\hat{\theta} - \tilde{\theta})+D_{h}\eta(\hat{\tau})(\hat{h} - \tilde{h})]\\
&= V_{\tau_u}[D_{\theta}\eta(\hat{\tau})(\hat{\theta} - \tilde{\theta})+D_{h}\eta(\hat{\tau})(\hat{h} - \tilde{h}), D_{\theta} \eta(\tau_u)(\theta_0-\tilde{\theta}) +  D_{h} \eta(\tau_u)(h_0-\tilde{h})]\\
&\leq V_{\tau_u}^{\frac{1}{2}}[D_{\theta}\eta(\hat{\tau})(\hat{\theta} - \tilde{\theta})+D_{h}\eta(\hat{\tau})(\hat{h} - \tilde{h})]  V_{\tau_u}^{\frac{1}{2}} [ D_{\theta} \eta(\tau_u)(\theta_0-\tilde{\theta}) +  D_{h} \eta(\tau_u)(h_0-\tilde{h})]\\
&\leq C_2C_4 V_{\tau_0}^{\frac{1}{2}}[L_{\theta}(\hat{\theta} - \tilde{\theta})+ L_{h}(\hat{h} - \tilde{h})]V_{\tau_0}^{\frac{1}{2}}[ L_{\theta} (\theta_0-\tilde{\theta}) +  L_h(h_0-\tilde{h})].
\end{align*}
Next, for any random sample $X_1, \ldots, X_n$ and any function $f \in L^2_0(\X)$, it is well known that $\E \brac{[ n^{-1} \sum_{i=1}^{n} f(X_i) - \E(f(X))]^2} = n^{-1} \Var[f(X)]$, and this implies $n^{-1} \sum_{i=1}^{n} f(X_i) - \E[f(X)] = O_{p} ( n^{-\frac{1}{2}})\Var^{\frac{1}{2}}[f(X)]$. Thus, the last two terms in the RHS of \eqref{III} are
\[O_{p} (n^{-\frac{1}{2}})V_{\tau_0}^{\frac{1}{2}}[D_{\theta}\eta(\hat{\tau})(\hat{\theta} - \tilde{\theta}) + D_{h}\eta(\hat{\tau})(\hat{h} - \tilde{h})]\quad \text{and} \quad O_{p} (n^{-\frac{1}{2}})V_{\tau_0}^{\frac{1}{2}}[D_{\theta}\eta(\tau_0)(\hat{\theta}-\tilde{\theta})+ D_{h}\eta(\tau_0)(\hat{h} - \tilde{h})],\]
respectively.

Putting everything together, as $\lambda \rightarrow 0$, $n \rightarrow \infty$, and $n\lambda^{\frac{1}{r}} \rightarrow \infty$, by all assumptions and Theorem \ref{theorem1}, we get
\begin{align*}
& \quad \frac{c_3}{2} V_{\tau_0} [ L_{\theta} (\hat{\theta} - \tilde{\theta}) + L_{h}(\hat{h}-\tilde{h} )] + \lambda J(\hat{h}-\tilde{h}) \nonumber \\
&\leq \mu_{\hat{\tau}}[D_{\theta}\eta(\hat{\tau})(\hat{\theta} - \tilde{\theta})+D_{h}\eta(\hat{\tau})(\hat{h} - \tilde{h})]- \mu_{\tilde{\tau}} [D_{\theta}\eta(\hat{\tau})(\hat{\theta} - \tilde{\theta})+D_{h}\eta(\hat{\tau})(\hat{h} - \tilde{h})] + \lambda J(\hat{h}-\tilde{h}) \nonumber\\
&\leq  V_{\tau_0}^{\frac{1}{2}}[D_{\theta}\eta(\tau_0)(\tilde{\theta}-\theta_0)+ D_{h}\eta(\tau_0)(\tilde{h}-h_0)] V_{\tau_0}^{\frac{1}{2}} [ D_{\theta}\eta(\tau_0)(\hat{\theta}-\tilde{\theta})+D_{h}\eta(\tau_0)(\hat{h} - \tilde{h})] \nonumber \\
&\quad + C_2C_4 V_{\tau_0}^{\frac{1}{2}}[L_{\theta}(\hat{\theta} - \tilde{\theta})+ L_{h}(\hat{h} - \tilde{h})]V_{\tau_0}^{\frac{1}{2}} [ L_{\theta} (\theta_0-\tilde{\theta}) +  L_h(h_0-\tilde{h})] \nonumber\\
&\quad + O_{p} (n^{-\frac{1}{2}})\brac{V_{\tau_0}^{\frac{1}{2}}[D_{\theta}\eta(\hat{\tau})(\hat{\theta} - \tilde{\theta}) + D_{h}\eta(\hat{\tau})(\hat{h} - \tilde{h})] + V_{\tau_0}^{\frac{1}{2}}[D_{\theta}\eta(\tau_0)(\hat{\theta}-\tilde{\theta})+ D_{h}\eta(\tau_0)(\hat{h} - \tilde{h})]}\nonumber\\
&\leq C_2(1+C_4) V_{\tau_0}^{\frac{1}{2}}[ L_{\theta} (\tilde{\theta}- \theta_0) + L_{h}(\tilde{h}-h_0 )]V_{\tau_0}^{\frac{1}{2}} [ L_{\theta} (\hat{\theta} - \tilde{\theta}) + L_{h}(\hat{h}-\tilde{h} )] \nonumber \\
& \quad +  O_{p}\left(n^{-\frac{1}{2}}\right)V_{\tau_0}^{\frac{1}{2}} [L_{\theta}(\hat{\theta}-\tilde{\theta}) + L_{h}(\hat{h}-\tilde{h})]  \nonumber \\
& \leq O_{p}\left(n^{-\frac{1}{2}}\lambda^{-\frac{1}{2r}} + \lambda^{\frac{1}{2}} \right)V_{\tau_0}^{\frac{1}{2}} [L_{\theta}(\hat{\theta}-\tilde{\theta}) + L_{h}(\hat{h}-\tilde{h})]. 
\end{align*}
The result then follows after trivial manipulation. 
\end{proof}

The following corollaries as direct results from Theorem \ref{theorem2}.
\begin{cor}Under Assumptions \ref{A2} to \ref{A9}, as $n \rightarrow \infty$, $\lambda \rightarrow 0$, and $n\lambda^{\frac{1}{r}} \rightarrow \infty$,  we have 
\[\Vert \hat{\theta}-\theta_0\Vert_{l^2} \sim \Vert\hat{\theta} - \theta_0\Vert_{\R^p} = V_{\tau_0}^{\frac{1}{2}} [ L_{\theta} (\hat{\theta} - \theta_0)] = O_{p}\left(n^{-\frac{1}{2}}\lambda^{-\frac{1}{2r}}+ \lambda^{\frac{1}{2}}\right).\]
\label{cor1}
\end{cor}

\begin{cor}If $\eta(\theta,h)= \alpha(x;\theta) + h(x)$, $L_h$ can be chosen to be the inclusion operator from $\CH$ to $L_0^2(\X)$. Under the same conditions as in Theorem \ref{theorem2}, we have 
\[\Vert\hat{h} - h_0\Vert_{\CH} =  [V_{\tau_0}(\hat{h} - h_0)+\lambda J(\hat{h} - h_0)]^{1/2} = O_{p}\left(n^{-\frac{1}{2}}\lambda^{-\frac{1}{2r} }+ \lambda^{\frac{1}{2}}\right).\]
\label{cor2}
\end{cor}
Corollaries \ref{cor1} and \ref{cor2} hold since $\norm{\cdot}_{\R^p}$ is equivalent to the Euclidean norm $\norm{\cdot}_{l^2}$ on $\R^p$, and $\norm{\cdot}_{\Q}$ is equivalent to the product norm $\norm{\cdot}_{\Q,1}$ on the joint parameter space $\Q$ (see Remark \ref{remark1}$(ii)$ and $(iv)$). We now derive a convergence rate of the overall density function, measured by the symmetrized Kullback-Leibler distance, defined in \eqref{SKL_def}.

\begin{theorem}\label{SKL}Under Assumptions \ref{A2} to \ref{A9}, as $n \rightarrow \infty$, $\lambda \rightarrow 0$, and $n\lambda^{\frac{1}{r}} \rightarrow \infty$,
\[\SKL(\tau_0, \hat{\tau})= O_{p}\left(n^{-1}\lambda^{-\frac{1}{r} }+ \lambda\right).\]
\end{theorem}
\begin{proof}By the definition of the Fr\'echet derivative of $\eta(\theta,h)$ at $(\theta_0,h_0)$, we have 
\[\eta(\hat{\theta},\hat{h}) =  \eta(\theta_0,h_0) + D_{\theta}\eta(\theta_0,h_0)(\hat{\theta}-\theta_0) + D_{h}\eta(\theta_0,h_0)(\hat{h}-h_0) + o(\Vert(\hat{\theta}-\theta_0, \hat{h}-h_0)\Vert_{\Q}). \]
By Theorem \ref{theorem2}, $\Vert(\hat{\theta}-\theta_0, \hat{h}-h_0)\Vert_{\Q} \rightarrow 0$ with probability tending to 1 as $\lambda \rightarrow 0$, and $n\lambda^{\frac{1}{r}} \rightarrow \infty$. Therefore,
\begin{align*}
\SKL(\tau_0, \hat{\tau}) =& \mu_{\tau_0}[\eta(\theta_0,h_0) - \eta(\hat{\theta},\hat{h})] + \mu_{\hat{\tau}}[\eta(\hat{\theta},\hat{h}) - \eta(\theta_0,h_0)]\\
\rightarrow & \mu_{\tau_0}[D_{\theta}\eta(\tau_0)(\theta_0 - \hat{\theta}) + D_h\eta(\tau_0)(h_0-\hat{h})]-\mu_{\hat{\tau}}[D_{\theta}\eta(\tau_0)(\theta_0 - \hat{\theta}) + D_h\eta(\tau_0)(h_0-\hat{h})].
\end{align*}
By equation \eqref{MVT}, for some $0\leq u \leq 1$, we have 
\begin{align}
&\mu_{\tau_0}[D_{\theta}\eta(\tau_0)(\theta_0 - \hat{\theta}) + D_h\eta(\tau_0)(h_0-\hat{h})]-\mu_{\hat{\tau}}[D_{\theta}\eta(\tau_0)(\theta_0 - \hat{\theta}) + D_h\eta(\tau_0)(h_0-\hat{h})] \nonumber\\
&= V_{\hat{\tau} + u(\tau_0-\hat{\tau})} [D_{\theta}\eta(\tau_0)(\theta_0 - \hat{\theta}) + D_h\eta(\tau_0)(h_0-\hat{h}),\nonumber\\
&\quad \quad \quad \quad \quad \quad D_{\theta}\eta(\hat{\tau} + u (\tau_0-\hat{\tau}))(\theta_0 - \hat{\theta}) + D_h\eta(\hat{\tau} + u (\tau_0-\hat{\tau}))(h_0-\hat{h})].\label{SKL_1}
\end{align}
By the Cauchy-Schwarz inequality, Assumptions \ref{A6} and \ref{A7}, and Theorem \ref{theorem2}, \eqref{SKL_1} is bounded above by
\begin{align*}
C_2C_4 V_{\tau_0}[L_{\theta} ( \theta_0 - \hat{\theta}) + L_h(h_0 - \hat{h})]= O_{p}\left(n^{-1}\lambda^{-\frac{1}{r} }+ \lambda\right).
\end{align*} 

\end{proof}
\begin{remark}\
Note that results in Corollary \ref{cor1}, Corollary \ref{cor2}, and Theorem \ref{SKL} are independent of the linear operators $L_{\theta}$ and $L_{h}$. 
\end{remark}

\end{subsection}

\begin{subsection}{Joint asymptotics for case $m\ge1$}\label{result for general m}
In this section, we will discuss the joint consistency for $m\ge1$. The proof for general $m$ is a simple extension of the proof for $m=1$ discussed in Sections \ref{Linear approximation} and \ref{approximation error}. 

The linear approximation estimate $\tilde{\tau} = (\tilde{\theta},\tilde{h})$ of $\hat{\tau}=(\hat{\theta},\hat{h})$, where $\tilde{h} = \sum_{\nu} \tilde{h}_{\nu} \phi_{0,\nu}$, can be obtained by minimizing equation \eqref{APL}. We get

\begin{align*}
\tilde{\theta} &= \theta_0 + \Omega_{\lambda}^{-1} \brac{\sum_{l=1}^{m}\bar{\alpha}_{n_l} - \sum_{\nu}\frac{\sum_{l=1}^{m} \bar{\beta}_{\nu,n_l} - \lambda \rho_{0,\nu} h_{0,\nu}}{1+\lambda \rho_{0,\nu}}\sum_{l=1}^{m} V_{l,\tau_0}[D_{\theta}\eta_l(\tau_0),D_{h}\eta_l(\tau_0)\phi_{0,\nu}]},\\
\tilde{h}_{\nu} &= \frac{\sum_{l=1}^{m}\bar{\beta}_{\nu,n_l} + h_{0,\nu}}{1+\lambda \rho_{0,\nu}} - (\tilde{\theta}-\theta_0)^{T} \frac{\sum_{l=1}^{m} V_{l,\tau_0}[D_{\theta}\eta_l(\tau_0),D_{h}\eta_l(\tau_0)\phi_{0,\nu}]}{1+\lambda \rho_{0,\nu}},
\end{align*}
where 
\begin{align*}
h_{0,\nu} &= \sum_{l=1}^{m} V_{l,\tau_0}(D_{h} \eta_l(\tau_0)h_0, D_{h} \eta_l(\tau_0)\phi_{0,\nu}),\\
\bar{\alpha}_{n_l} &= \frac{1}{n_l}\sum_{i=1}^{n_l} D_{\theta}\eta_l(X_{li};\tau_0) - \mu_{l,\tau_0} \left[D_{\theta}\eta_l(\tau_0) \right],\\
\bar{\beta}_{\nu,n_l} &= \frac{1}{n_l} \sum_{i=1}^{n_l} D_{h}\eta_l(X_{li};\tau_0)\phi_{0,\nu} - \mu_{l,\tau_0} \left[D_{h}\eta_l(\tau_0)\phi_{0,\nu} \right],\\
\Omega_{\lambda} &=  \sum_{l=1}^{m} V_{l,\tau_0}[D_{\theta}\eta_l(\tau_0)] - \sum_{\nu} \frac{\brac{\sum_{l=1}^{m}V_{l,\tau_0}[D_{\theta}\eta_l(\tau_0),D_{h}\eta_l(\tau_0)\phi_{0,\nu}]}^{\otimes 2}}{1+\lambda \rho_{0,\nu}}.
\end{align*}
Recall that $n = \min_{1\leq l \leq m} \{n_l\}$. Since
\[\E(\sum_{l=1}^{m}\bar{\alpha}_{n_l}) = \E(\sum_{l=1}^{m}\bar{\beta}_{\nu,n_l}) = 0, \quad \E[\Vert(\sum_{l=1}^{m}\bar{\alpha}_{n_l})\Vert_{l^2}^2 ] = O(n^{-1}), \quad \E[(\sum_{l=1}^{m}\bar{\beta}_{\nu,n_l})^2] \leq n^{-1},\]
one can get the following results, which are similar to Lemma \ref{lemma1} and Theorems \ref{theorem1} -- \ref{SKL}. Note that one also has corollaries to Theorem \ref{theorem5} which are analogous to Corollaries \ref{cor1} and \ref{cor2}, but for $m>1$. They are omitted to save space.

\begin{lemma}
Under Assumptions \ref{A2} to \ref{A4}, \ref{A8}, and \ref{A9}, as ${n =\min_{1\leq l \leq m} \{n_l\} \rightarrow \infty}$ and $\lambda \rightarrow 0$,
\begin{align*}
&\E\brac{ \sum_{l=1}^{m} V_{l,\tau_0} [D_{\theta} \eta_l(\tau_0) (\tilde{\theta} - \theta_0)]} \leq c \E \left[ (\tilde{\theta}-\theta_0)^{T}(\tilde{\theta}-\theta_0) \right] = O(n^{-1}\lambda^{-\frac{1}{r}}),\\
&\E\brac{ \sum_{l=1}^{m} V_{l,\tau_0}[D_{h}\eta_l(\tau_0)(\tilde{h}-h_0)] + \lambda J(\tilde{h}-h_0)} = O(n^{-1}\lambda^{-\frac{1}{r}} + \lambda).
\end{align*}
\end{lemma}

\begin{theorem}
Under Assumptions \ref{A2} to \ref{A6}, \ref{A8}, and \ref{A9}, as ${n=\min_{1\leq l \leq m} \{n_l\} \rightarrow \infty}$ and $\lambda \rightarrow 0$, 
\[\sum_{l=1}^{m} V_{l,\tau_0}[L_{l,\theta} (\tilde{\theta} - \theta_0) + L_{l,h}(\tilde{h}-h_0)] + \lambda J(\tilde{h}-h_0) = O_{p}(n^{-1}\lambda^{-\frac{1}{r}} + \lambda).\]
\end{theorem}

\begin{theorem}
\label{theorem5}
Under Assumptions \ref{A2} to \ref{A9}, as ${n =\min_{1\leq l \leq m} \{n_l\} \rightarrow \infty}$, $\lambda \rightarrow 0$, and $n\lambda^{\frac{1}{r}} \rightarrow \infty$,
\[\sum_{l=1}^{m} V_{l,\tau_0} [ L_{l,\theta} (\hat{\theta} - \tilde{\theta}) + L_{l,h}(\hat{h}-\tilde{h} )] + \lambda J(\hat{h}-\tilde{h}) = O_{p}\left(n^{-1}\lambda^{-\frac{1}{r} }+ \lambda\right).\]
Therefore,
\[\sum_{l=1}^{m}V_{l,\tau_0} [ L_{l,\theta} (\hat{\theta} - \theta_0 )+ L_{l,h}(\hat{h}-h_0 )] + \lambda J(\hat{h}-h_0) = O_{p}\left(n^{-1}\lambda^{-\frac{1}{r} }+ \lambda\right).\]
\end{theorem}

\begin{theorem}Under Assumptions \ref{A2} to \ref{A9}, as ${n =\min_{1\leq l \leq m} \{n_l\} \rightarrow \infty}$, $\lambda \rightarrow 0$, and $n\lambda^{\frac{1}{r}} \rightarrow \infty$,
\[\SKL(\tau_0, \hat{\tau})= \sum_{l=1}^{m} \brac{\mu_{l,\tau_0}[\eta(\theta_0,h_0) - \eta(\hat{\theta},\hat{h})] + \mu_{l,\hat{\tau}}[\eta(\hat{\theta},\hat{h}) - \eta(\theta_0,h_0)]}=O_{p}\left(n^{-1}\lambda^{-\frac{1}{r} }+ \lambda\right).\]
\end{theorem}
\end{subsection}

\section{ Simulations}\label{sec::semiSimulation}
To evaluate the proposed estimation procedures and algorithms, we conduct several simulations for both additive and nonadditive cases. We use the KL divergence between the true density and the estimated density to evaluate the  performance of density estimation, $\text{KL}(f,\hat{f})=\int_{\cal X} f(x) \log \{ f(x)/\hat{f}(x)\} dx$. 
Furthermore, we will use the generalized decomposition
\begin{equation}
\hbox{E}\{\hbox{KL}(f,\hat{f})\}=\hbox{KL}(f,\bar{f})+ 
\hbox{E}\{\hbox{KL}(\bar{f},\hat{f})\}
=\hbox{bias }+\hbox{ variance}
\label{KLDdecomp}
\end{equation}
proposed by \citeasnoun{heskes1998bias} 
to evaluate the bias-variance trade-off 
where $\bar{f}=\exp\{\hbox{E}(\log \hat{f})\}/Z$ and $Z$ is a normalization 
constant. 

We emphasize that in all the following tables, all numbers are in $10^{-2}$.  The numbers in the parentheses for the KL measures are the biases and the variances in the decomposition of the KL distances. The numbers in the parentheses for the mean squared errors (MSEs) are the squared biases and the variances of the parameter estimators.

\subsection{Additive cases}
\subsubsection{Near Normal distribution}
Consider  the density function
$$
f(x;\mu,\sigma)=\frac{ \exp\{-\frac{x^2}{2\sigma^2}+\frac{\mu x}{\sigma^2} + ax^3\}}{\int_0^1  \exp\{-\frac{x^2}{2\sigma^2}+\frac{\mu x}{\sigma^2} + ax^3\}dx}
$$
where $x\in [0,1]$, $\alpha(x;\mu,\sigma)=-\frac{x^2}{2\sigma^2}+\frac{\mu x}{\sigma^2}$, $h(x)=ax^3$ and $a$ is a constant. The function $\alpha(x;\mu,\sigma)$ is the logistic transformation of the truncated  Normal density function with mean $\mu$ and standard deviation $\sigma$. The constant $a$ controls the departure from the truncated Normal distribution. The choice of this density function is motivated by \citeasnoun{hjort1995nonparametric}, with a truncated Normal as the starting parametric density. 
We will consider the  additive model \eqref{additive} with $\alpha(x;\mu,\sigma)=-\frac{x^2}{2\sigma^2}+\frac{\mu x}{\sigma^2}$
and $h \in W_2^3[0,1] \ominus \{1, x, x^2\}$, where $W_2^s[0,1]$ is the $s$th  order Sobolev space on $\CX=[0,1]$ defined 
as 
\[W_2^s[0,1]= \lbrace g: \X \rightarrow \R \mid\  g^{(j)} \text{ is absolutely continuous, for }j = 0, 1, \ldots, s-1\}.\]
We consider the third order Sobolev space because the null space with quadratic functions corresponds to the probability density functions of the truncated Normal distributions. This null space has been removed from the model space of $h$ to make the model identifiable.  
We use the   profile likelihood approach in Section \ref{sec::additive} to compute estimates of $\bftheta=(\mu,\sigma)$ and $h$. 

For comparison, we also implemented the nonparametric estimation methods based on kernels and polynomial splines. Given a random sample $X_1, \ldots, X_n$ that has density $f(x)$ on $[0,1]$, the polynomial spline based nonparametric estimation method estimates the logistic transformation of $f$, $\eta$, by minimizing the penalized likelihood
$$
-\dfrac{1}{n}\sum_{i=1}^{n}\eta(X_{i}) + \log\int_0^1 e^{\eta}dx + \dfrac{\lambda}{2}\int_0^1(\eta^{(m)})^2 dx, \label{polyPLF}
$$
over $\eta \in W_2^s[0,1]$. We implemented the cubic spline ($s=2$) and the quintic spline ($s=3$).
In the implementation of kernel density estimation and the method by \citeasnoun{hjort1995nonparametric}, we use the Gaussian kernel and compare different methods for bandwidth selection, including the approach in \citeasnoun{scott1992curse}, the unbiased and biased cross-validation procedures,  and the method using pilot estimation of derivatives in  \citeasnoun {sheather1991reliable}.
We find that the bandwidth obtained by
\begin{equation}
h = 0.9 n^{-1/5}\min \{\text{sample standard deviation, sample inter quantile range/1.34}\},\label{kernelBw}
\end{equation}
provides the best overall performance. We report results with bandwidth determined by \eqref{kernelBw} only. 

We set $\mu=0.5$ and $ \sigma=0.2$. We consider three choices of $a$: $a=0.25,1,4$, and three sample sizes, $n=100,200$ and $500$. For each simulation setting, we generate 100 simulation data sets.  For each simulated data set, we compute density estimates using the kernel method, the method by \citeasnoun{hjort1995nonparametric} (HG), cubic spline, quintic spline and our proposed semiparametric (SEMI) method. 
Table \ref{CompareNormalDen} summarizes the performances of these methods in terms of the density estimation.
As expected,  the performance improves as sample size increases, and the semiparametric and quintic spline estimates have the smallest KL than other nonparametric estimates. Our semiparametric approach performs better than the semiparametric approach in \citeasnoun{hjort1995nonparametric}. There is virtually no difference between the quintic spline approach. This is expected since they are fitting the same model using different methods. 

\begin{table}[H]
\center
\begin{tabular}{ccc|c|c}
\hline
$a$&Method& KL & KL &KL\\
\hline
&&$n=100$&$n=200$&$n=500$\\\hline
\multirow{5}{*}{0.25}& Kernel & 2.56(0.04, 2.52) &1.65(0.03, 1.62)& 0.96(0.03, 0.93)   \\
& HG & 2.48(0.47, 2.01) &  1.21(0.18, 1.03)& 0.62(0.12, 0.50)\\ 
 & Cubic & 1.53(0.31, 1.22)& 0.81(0.21,  0.60)&   0.41(0.11, 0.30)\\ 
&Quintic & 1.08(0.04, 1.04) &0.43(0.01, 0.42 )&0.20(0.00, 0.20)    \\ 
&  SEMI & 1.08(0.04, 1.04) & 0.43(0.01, 0.42)& 0.20(0.00, 0.20)\\ 
\hline
\multirow{5}{*}{1} 
&  Kernel & 2.82(0.04, 2.78) &1.84(0.07, 1.77)  &1.25(0.04, 1.21)   \\ 
&  HG & 2.30(0.14, 2.16) &1.38(0.43, 0.95)   &  0.67(0.11, 0.56)  \\ 
&  Cubic & 1.51(0.28, 1.23) & 0.80(0.15, 0.65) &  0.45(0.12, 0.33)\\ 
&Quintic & 0.93(0.03, 0.90) &0.46(0.02, 0.44)&0.22(0.01, 0.21) \\ 
 & SEMI  & 0.93(0.03, 0.90) &0.47(0.02, 0.45)&0.24(0.01, 0.23)   \\ 
\hline
\multirow{5}{*}{4}
&  Kernel & 7.79(0.83, 6.96) &  6.24(0.74, 5.50)& 4.78(0.67, 4.11) \\ 
&  HG & 4.32(2.19, 2.13) &  2.77(1.73, 1.04)&  1.78(1.29, 0.49)\\ 
 & Cubic & 1.54(0.27, 1.27) &  0.79(0.18, 0.61)&  0.34(0.07, 0.27)\\ 
&Quintic & 1.27(0.19, 1.08)& 0.63(0.14 ,0.49)&0.33(0.15, 0.18) \\ 
 & SEMI  & 1.27(0.19, 1.08) &0.64(0.15, 0.49)&  0.35(0.16, 0.19)\\ 
\hline
\end{tabular}
\caption{Kullback-Leibler distances of different methods. Numbers inside parentheses of KLs are bias and variances. All numbers are in $10^{-2}$.}\label{CompareNormalDen}
\end{table}

\subsubsection{Near Gumbel distribution}\label{sec::linearGum}
Consider the density function
$$
f(x;\mu,\sigma)=\frac{ \exp\{-\frac{x-\mu}{\sigma}-\exp(\frac{x-\mu}{\sigma}) + ax^3\}}{\int_0^1   \exp\{-\frac{x-\mu}{\sigma}-\exp(\frac{x-\mu}{\sigma}) + ax^3\}dx}
$$
where $x\in \CX=[0,1]$, $\alpha(x;\mu,\sigma)=-{(x-\mu)}/{\sigma}-\exp({(x-\mu)}/{\sigma})$, $h(x)=ax^3$, and $a$ is a constant. This is equivalent to start with a truncated Gumbel distribution whose logistic transformation is $\alpha(x;\mu,\sigma)$, and the constant $a$ controls the departure from the truncated Gumbel distribution.  We will consider the  additive model \eqref{additive} with $\alpha(x;\mu,\sigma)=-{(x-\mu)}/{\sigma}-\exp({(x-\mu)}/{\sigma})$ and $h \in W_2^2[0,1] \ominus \{1\}$. Note that $\alpha$ is non-linear in $\bftheta$. Therefore, the method in \citeasnoun{yang2009penalized} does not apply.  We compare our proposed method with the kernel, cubic spline and HG's method, where $f_0$ in the HG's approach is a truncated Gumbel distribution. We use the profile likelihood approach as described in Section \ref{sec::additive}  to compute estimates of $\bftheta=(\mu,\sigma)$ and $h$.

We set $\mu=0.5$ and $\sigma=0.2$. We consider three choices of $a$: $a=0.25,1,4$,  and three sample sizes $n=100,200,500$. For each simulation setting, we generate 100 simulated data sets. Table \ref{CompareGumbelDen}  shows that the semiparametric methods has smaller KL distances when the true density is close to truncated Gumbel (i.e. $a=0.25$ and $1$). When the true density if far away from the  truncated Gumbel ($a=4$), cubic spline has smaller KL distances.

\begin{table}[H]
\center
\begin{tabular}{ccc|c|c}
\hline
$a$&Method& KL & KL &KL\\
\hline
&&$n=100$&$n=200$&$n=500$\\\hline
\multirow{3}{*}{0.25}& Kernel & 3.50(0.41, 3.09) & 2.59(0.36, 2.23)     &  1.72(0.21, 1.51)   \\ 
&  HG & 3.20(1.12, 2.08) & 1.89(0.68, 1.21) &  1.10(0.49, 0.61)  \\ 
&  Cubic & 2.12(0.59, 1.53) &   1.32(0.42, 0.90)  & 0.63(0.19, 0.44)    \\ 
&  SEMI & 1.58(0.16, 1.42) &  0.64(0.01, 0.63) &  0.45(0.04, 0.41) \\ \hline
\multirow{ 3}{*}{1}&Kernel & 4.32(0.47, 3.85) & 3.22(0.49, 2.73)   & 2.35(0.31, 2.04)   \\ 
&  HG & 3.07(1.00, 2.07) &   2.04(1.02, 1.02) & 1.19(0.61, 0.58)    \\ 
&  Cubic & 2.17(0.69, 1.48) &   1.21(0.52, 0.69) &   0.64(0.21, 0.43)  \\ 
&  SEMI & 1.28(0.11, 1.17) &  0.64(0.06, 0.58) & 0.45(0.05, 0.40) \\ \hline
\multirow{ 3}{*}{4}& Kernel & 11.47(1.86, 9.61) &   9.68(1.88, 7.80) &  7.68(1.53, 6.15)  \\ 
&  HG & 7.83(5.85, 1.98) &  5.27(4.08, 1.19)  &  3.50(3.04, 0.46)  \\ 
&  Cubic & 2.00(0.65, 1.35) &  1.08(0.37, 0.71)  & 0.52(0.19, 0.33)     \\ 
&  SEMI & 2.33(0.58, 1.75) & 1.30(0.32, 0.98)& 0.66(0.13, 0.53) \\ \hline
\end{tabular}
\caption{Kullback-Leibler distances of different methods. Numbers inside parentheses of KLs are bias and variances. All numbers are in $10^{-2}$.}\label{CompareGumbelDen}
\end{table}

\subsection{Nonadditive cases}
\subsubsection{Power transformation}\label{sec::powerTrans}
The truncated Weibull distribution has density function
\begin{equation}
f(x;s,  \gamma)=\frac{ \gamma}{s}\left(\frac{x}{s}\right)^{ \gamma-1}\exp\left\lbrace -\left(\frac{x}{s} \right)^ \gamma\right\rbrace,
\end{equation}
where $x\in \CX=[0,1]$, $s>0$ is the scale parameter and $ \gamma>0$ is the shape parameter. Since $Y=X^\gamma \sim \text{truncated Exp} (s)$, we consider the transformation model in Section \ref{sec::powEst} with 
\begin{equation}
t(x;\theta)=x^{\theta} ,\quad \theta>0,
\label{PowerTrans}
\end{equation}
and $h \in W_2^2[0,1] \ominus \{1\}$. The algorithm in Section \ref{sec::powEst} is used to compute the estimates of $\theta$ and $h$. We also compute the cubic spline estimate of the density function (Cubic) and the  kernel density estimation after transformation (TransKernel) \cite{wand1991transformations}, where the transformation form is known as in (\ref{PowerTrans}) up to the parameter $\theta$.
In the simulations, we set $s=1$, and consider $ \gamma=1,2, 3$ where $ \gamma=1$ corresponding to the truncated exponential distribution for which the cubic spline is well-suited since the null space corresponds to the truncated exponential density.  We consider two sample sizes $n= 100$ and $200$. For each simulation setting, we generate 100 data sets. The simulation results are shown in Table \ref{Power100200}. 
As expected, cubic spline performs best when $ \gamma=1$. Our semiparametric approach provides  smaller  KL  distances in other cases. As sample size increases, both the performances of density and parameter estimation improved, except for the transformation kernel density estimation, the biases of the parameter estimates actually increased although the variances reduced, leading to slightly larger MSEs.

\subsubsection{Two-sample density estimation}
The Gumbel and logistic distributions are members of  location scale family with the location parameter $\mu$ and scale parameter $\sigma$. They have been extensively used in many different areas.


\begin{table}[H]
\center
\setlength{\tabcolsep}{3pt}
\begin{tabular}{cc|cc|cc}

\hline
$ \gamma$&Method& KL &MSE ($\hat\gamma$) &KL &MSE ($\hat\gamma$)\\
\hline
&&\multicolumn{2}{c|}{$n=100$}&\multicolumn{2}{c}{$n=200$}\\
\hline
\multirow{ 3}{*}{1}& Cubic & 0.65(0.01, 0.64)& &0.43(0.01, 0.42)& \\ 
& TransKernel&1.87(0.69, 1.18)&13.40(11.44, 1.96)&1.57(0.74, 0.83)&15.95(15.05, 0.90) \\
&SEMI& 1.09(0.02, 1.07)&3.07(0.01, 3.06)&0.61(0.00, 0.61)&1.33(0.01, 1.32)\\
\hline
\multirow{ 3}{*}{2}&Cubic&1.90(0.67, 1.23)& & 1.26(0.43, 0.83)&\\ 
& TransKernel&2.05(0.52, 1.53)&58.35(50.44, 7.91)&1.64(1.04, 0.60)&61.80(58.02, 3.78)\\
& SEMI &0.94(0.02, 0.92)&8.63(0.17, 8.46)&0.62(0.01, 0.61)&6.02(0.18, 5.84)\\ 
 \hline
\multirow{ 3}{*}{3}& Cubic &1.72(0.55, 1.17)&  &0.96(0.35, 0.61) &\\ 
 & TransKernel&1.92(0.74, 1.18)&123.05(106.61, 16.44)&1.37(0.71, 0.66)&135.24(128.98, 6.26)\\
 &SEMI&0.95(0.03, 0.92)&21.07(0.48, 20.59)&0.48(0.01, 0.47)&10.75(0.34, 10.41)\\ 
\hline
\end{tabular}
\caption{Kullback-Leibler distances and mean squared errors of different methods. Numbers inside parentheses of KLs are bias and variances. Numbers inside parentheses of MSEs are squared biases and variances. All numbers are in $10^{-2}$.}\label{Power100200}
\end{table}

\noindent
We generate two independent samples, $X_1,\ldots,X_{n_1}\overset{iid}{\sim} f(x;0,1)$ and $Y_1,\ldots,Y_{n_2}\overset{iid}{\sim}f(x;\mu,\sigma)$ where $f$ is either the Gumbel distribution or the logistic distribution.  We set $\mu =2$, $\sigma=1$, and consider four sample sizes $(n_1, n_2)=(100,100),(100,200),(200,100)$ and $(200,200)$. We fit model \eqref{nonadditive} with $\alpha_1(x;\theta)=x$ and $\alpha_2(x;\theta)=\frac{x-\mu}{\sigma}$ and $h$ belongs to the RKHS for the univariate thin-plate splines with $\CX=\mathbb{R}$ and
\begin{equation}
\CH=\left\lbrace h:\int_{-\infty}^{\infty}(h^{\prime\prime})^2dx<\infty \right\rbrace\ominus {1}. \label{TPspace}
\end{equation}
We estimate the density functions using the procedure described in Section \ref{sec::powEst}, and report KL$(f(\cdot;0,1),\hat{f}(\cdot;0,1))+\text{KL}(f(\cdot;\mu,\sigma),\hat{f}(\cdot;\hat{\mu},\hat{\sigma}))$ for the overall performance of density estimation (overall KL). For comparison, we consider the characteristic function based two sample transformation density estimation method (CHAR) in \citeasnoun{potgieter2012nonparametric}, and another method which estimates the densities for $X_i$'s and $Y_i$'s using thin-plate spline models separately (TP), with logistic of densities belong to $ \CH$ in \eqref{TPspace}. Denote the separate thin-plate estimates for $X$ and $Y$ samples as $\hat{f}_1$ and $\hat{f}_2$ respectively. We report KL$( f(\cdot;0,1),\hat{f}_1)+\text{KL}(f(x;\mu,\sigma),\hat{f}_2)$ for the overall performance (overall KL) of the density estimation. For the method CHAR and our method SEMI, we report the MSEs, squared biases and variances as measures for the performance of parameter estimations. We report $\hbox{KL}_s=\hbox{KL}( f(\cdot;0,1),\exp(\hat{h})/\int \exp(\hat{h})dx)$ as a measure of the performance for estimating the function $h$. \\
\noindent \textbf{Gumbel distribution}
The Gumbel distribution has the density 
$
f(x)={\frac  {1}{\sigma }}e^{{-( {(x-\mu) }/{\sigma}}+e^{{-{(x-\mu)}/{\sigma }}})}$, $\quad x\in\mathbb{R}$.
Table \ref{TwoGumbel} shows although that the CHAR method estimates the parameters slightly better, our semiparametric method provides better density 
estimation than the method in \citeasnoun{potgieter2012nonparametric} and  the separate thin-spline estimates. All methods improved as the sample size increases.

\begin{table}[H]
\center
\setlength\tabcolsep{1.5pt}
\begin{tabular}{c|c|c|c|c|c|c}
\hline
$n_1$&$n_2$&Method& Overall KL & $\hbox{KL}_s$ &MSE($\hat{\mu}$) &MSE($\hat{\sigma}$)\\
\hline
\multirow{ 6}{*}{100}&\multirow{ 3}{*}{100}& SEMI& 5.36(0.92, 4.44)&2.56(0.49, 2.07)&3.58(0.00, 3.58)&2.34(0.00, 2.34)\\
&&CHAR &13.39(10.06, 3.33)&6.86(3.53, 3.33)&2.24(0.00, 2.24)&1.95(0.00, 1.95)\\
&&TP&5.81(1.95, 3.86)&&&\\\cline{2-7}
&\multirow{ 3}{*}{200}&SEMI &4.27(0.79, 3.48)&2.56(0.41, 2.15)&2.31(0.01, 2.30)&1.41(0.00, 1.41)\\
&&CHAR & 11.07(8.82, 2.25)&6.26(4.02, 2.24)&1.58(0.00, 1.58)&1.34(0.00, 1.34)\\
&&TP&4.95(1.61, 3.34)&&&\\\hline
\multirow{ 6}{*}{200}&\multirow{ 3}{*}{100}&SEMI &3.84(0.87, 2.97)&1.60(0.41, 1.19)&2.63(0.11, 2.52)&1.29(0.07, 1.22)\\
&&CHAR&10.67(8.83, 1.84)&4.69(2.86, 1.83)&1.52(0.04, 1.48)&0.97(0.03, 0.94)\\
&&TP&4.58(1.57, 3.01)&&&\\\cline{2-7}
&\multirow{ 3}{*}{200}&SEMI &3.21(0.85, 2.36)&1.46(0.32, 1.14)&1.78(0.00, 1.78)&0.86(0.00, 0.86)\\
&&CHAR&8.34(6.82, 1.52)&4.07(2.55, 1.52)&1.65(0.00, 1.65)&0.74(0.01, 0.73)\\
&&TP&3.60(1.16, 2.44)&&&\\\hline
\end{tabular}
\caption{Kullback-Leibler distances and mean squared errors of different methods. Numbers inside parentheses of KLs are bias and variances. Numbers inside parentheses of MSEs are squared biases and variances. All numbers are in $10^{-2}$.}\label{TwoGumbel} 
\end{table}

\noindent \textbf{Logistic distribution}
The pdf of the logistic distribution is given by
$
f(x;\mu ,\sigma)={\frac {e^{-{{(x-\mu)}/{\sigma}}}}{\sigma\left(1+e^{-{{(x-\mu)}/{\sigma}}}\right)^{2}}}$, $\quad x\in\mathbb{R}.
$
The simulation results are in Table \ref{TwoLogistic}. We observe the same relative behavior as in the case of two-sample Gumbel distributions.
\begin{table}[H]
\center
\setlength\tabcolsep{1.5pt}
\begin{tabular}{c|c|c|c|c|c|c}
\hline
$n_1$&$n_2$&Method& Overall KL & $\hbox{KL}_s$ &MSE($\hat{\mu}$) &MSE($\hat{\sigma}$)\\
\hline
\multirow{ 6}{*}{100}&\multirow{ 3}{*}{100}& SEMI& 3.19(0.29, 2.90)&1.49(0.13, 1.36)&7.25(0.10, 7.15)&1.58(0.00, 1.58)\\
&&CHAR&10.52(7.48, 3.04)&5.24(2.20, 3.04)&6.03(0.09, 5.94)&1.41(0.00, 1.41)\\
&&TP&3.71(0.59, 3.12)&&&\\\cline{2-7}
&\multirow{ 3}{*}{200}&SEMI &2.41(0.15, 2.26) & 1.53(0.04, 1.49)&4.67(0.15, 4.52) &1.44(0.00, 1.44)\\
&&CHAR&8.20(5.92, 2.28)&4.74(2.49, 2.25)&3.94(0.15, 3.79) &1.36(0.00, 1.36)\\
&&TP&2.69(0.36, 2.33)&&&\\\hline
\multirow{ 6}{*}{200}&\multirow{ 3}{*}{100}& SEMI&2.17(0.12, 2.05)&0.82(0.09, 0.73)&5.53(0.00, 5.53)&1.16(0.04, 1.12)\\
&&CHAR&8.20(6.41, 1.79)&3.70(1.91, 1.79) &5.34(0.01, 5.33)&1.04(0.02, 1.02)\\
&&TP&2.67(0.34, 2.33)\\\cline{2-7}
&\multirow{ 3}{*}{200}&SEMI &1.54(0.17, 1.37)&0.78(0.05, 0.73)&3.70(0.01, 3.69)&0.58(0.01, 0.57)\\
&&CHAR&6.10(4.87, 1.23)&3.15(1.92, 1.23)&3.53(0.01, 3.52)&0.54(0.00, 0.54)\\
&&TP&1.86(0.36, 1.50)&&&\\\hline
\end{tabular}
\caption{Kullback-Leibler distances and mean squared errors of different methods. Numbers inside parentheses of KLs are bias and variances. Numbers inside parentheses of MSEs are squared biases and variances. All numbers are in $10^{-2}$.}\label{TwoLogistic} 
\end{table}

\section{Application to the Old Faithful data }\label{sec::realAnalysis}
The data set consists of waiting time between eruptions and the duration of the eruption gathered from 272 consecutive eruptions of the Old Faithful geyser in Yellowstone National Park, Wyoming, USA \cite{hardle2012smoothing}. Histograms of these two variables suggest that both of the two variables have two modes. Therefore, we consider a semiparametric model with a mixture Normal as the parametric component:
\begin{equation}
\eta(x;\mu_1,\sigma_1,\mu_2,\sigma_2,p,h)=\log f_0(x;\mu_1,\sigma_1,\mu_2,\sigma_2,p)+h(x),\label{faithfulModel}
\end{equation}
where $h\in W_2^2[40,100] \ominus \{1\}$ for the waiting time and $h\in W_2^2[1.5,5.5]\ominus \{1\}$ for the duration time, and $f_0$ is a mixture Normal density
\begin{equation}
f_0(x)=p \frac{1}{\sqrt{2\pi \sigma_1^2}}\exp\left(-\frac{(x-\mu_1)^2}{2\sigma_1^2}\right)+(1-p)\frac{1}{\sqrt{2\pi \sigma_2^2}}\exp\left(-\frac{(x-\mu_2)^2}{2\sigma_2^2}\right).\label{mixnorm}
\end{equation}
Model \eqref{faithfulModel} is a special case of the additive model \eqref{additive}. We compute density estimates for both waiting time and duration using the profiled likelihood  estimation procedure in Section \ref{sec::additive}.  We also compute the estimates of the mixture Normal 
density \eqref{mixnorm} and cubic spline for both the waiting time and the duration. 


The estimated densities are shown in Figures \ref{faithfulwaiting} and \ref{faithfulDuration} for the waiting time and the duration respectively.  For the  waiting time, the estimates from the semiparametric model and mixture Normal are quite close suggesting that the parametric mixture Normal fits data adequately. Both of them follow the data on the left  slightly better than the  cubic spline estimate. For the duration, the estimates from the semiparametric model and cubic spline are close. And both of them portray both modes more precisely.
\begin{figure}[H]
\center
\includegraphics[scale=0.8]{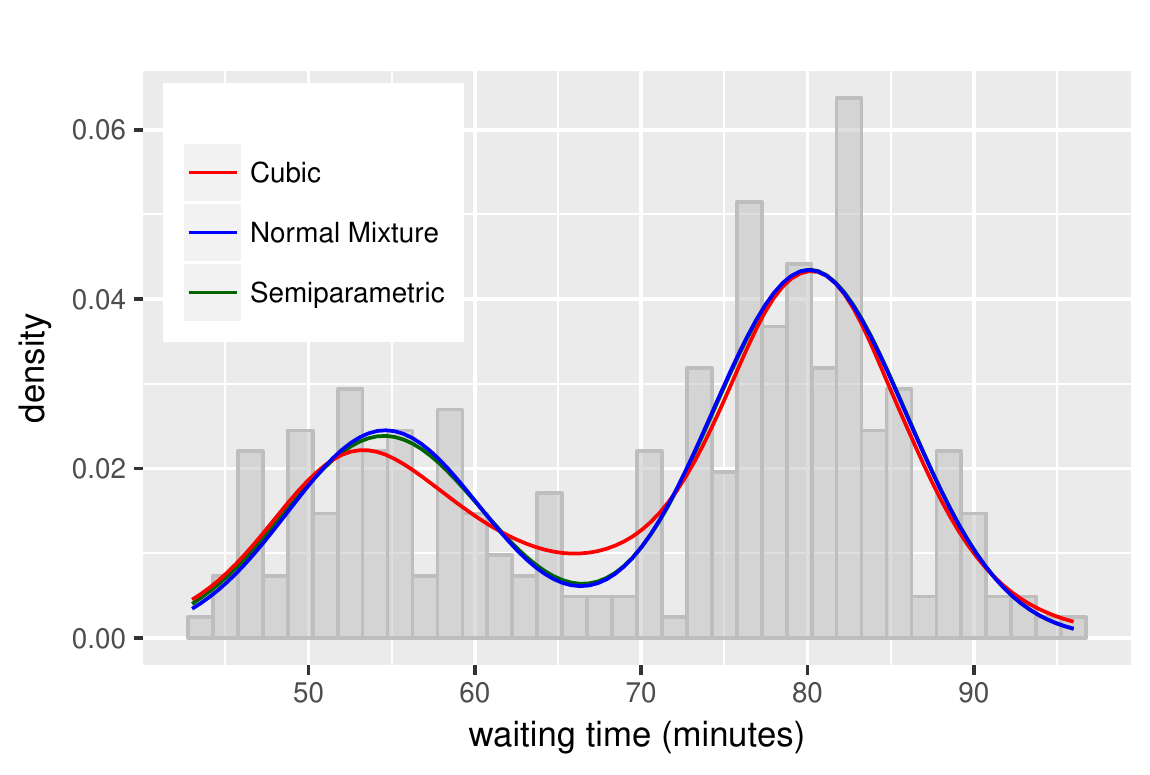}
\caption{Density estimates for the waiting time of faithful data.}\label{faithfulwaiting}
\end{figure}
\begin{figure}[H]
\center
\includegraphics[scale=0.8]{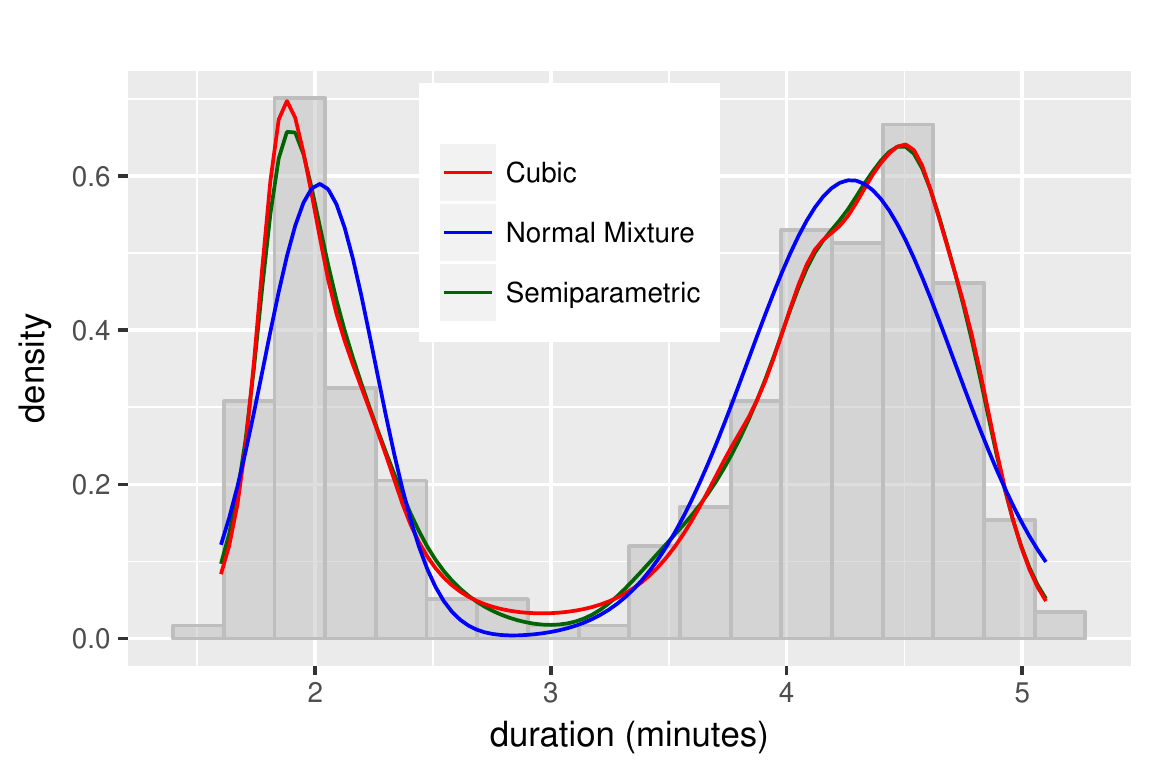}
\caption{Density estimates for the duration  of faithful data.}\label{faithfulDuration}
\end{figure}


\section*{Acknowledgements}
We gratefully acknowledge support from the National Science Foundation (DMS-1507078 for Anna Liu and DMS-1507620 for Yuedong Wang). We acknowledge support from the Center for Scientific Computing from the CNSI, MRL: an NSF MRSEC (DMR-1720256).

\section*{Appendix}

\subsection*{Verification of inner product $\left\langle \cdot,\cdot \right\rangle_{\Q}$}
We now turn to the discussion of the validity of \eqref{inner_product}. For any $p_1 \times 1$ and $p_2 \times 1$ vectors of functions in $\CH$, say $G_1 = [G_1^k]_{k=1}^{p_1}$ and  $G_2 = [G_2^k]_{k=1}^{p_2}$, we use the vector form of the inner product $\langle G_1, G_2 \rangle_{\CH}$ to denote a $p_1 \times p_2$ matrix in which the $(i,j)$th entry is $\langle G_1^i, G_2^j \rangle_{\CH}$. For any $g \in \CH$, let $\mathcal{F}^{k}g = \sum_{l=1}^{m} V_{l,\tau_0}[L_{\l,\theta}^k, L_{l,h}g]$. Since 
\[ \abs{\mathcal{F}^{k} g} \leq \left[\sum_{l=1}^{m} V_{l,\tau_0}(L_{\l,\theta}^k)\right]^{\frac{1}{2}} \left[ \sum_{l=2}^{m} V_{l,\tau_0}(L_{l,h} g) \right]^{\frac{1}{2}} \leq C \norm{g}_{\CH}\]
for some positive constant $C$, $\mathcal{F}^k$ is a bounded linear functional on $\CH$. By the Riesz representation theorem, there exists a $F^k \in \CH$ such that for any $g \in \CH$,
\[\mathcal{F}^{k}g = \sum_{l=1}^{m} V_{l,\tau_0}[L_{l,\theta}^k, L_{l,h} g] = \langle F^k, g \rangle_{\CH}.\]
Let $F = \left[F^k\right]_{k=1}^p$. We define $V_{l,\tau_0}(L_{l,\theta}, L_{l,h} g)$, $V_{l,\tau_0}(L_{l,h} F,L_{l,h} g)$, and $J(F,g)$ to be $p \times 1$ vectors whose $k$th entries are $V_{l,\tau_0}(L_{l,\theta}^k, L_{l,h} g)$, $V_{l,\tau_0}(L_{l,h} F^k,L_{l,h} g)$, and $J(F^k,g)$, respectively. Therefore, 
\[  \sum_{l=1}^{m} V_{l,\tau_0}(L_{l,\theta}, L_{l,h} g)= \sum_{l=1}^{m}V_{\tau_0}(L_{l,h} F,L_{l,h} g) + \lambda J(F,g) = \langle F, g \rangle_{\CH}
.\]
We also define the $p \times p$ matrix $\Omega_F = \sum_{l=1}^{m} V_{l,\tau_0}(L_{l,\theta} - L_{l,h}F , L_{l,\theta} - L_{l,h}F)$, whose $(i,j)$th entry is $\sum_{l=1}^{m} V_{l,\tau_0}(L^i_{l,\theta} - L_{l,h}F^i , L^j_{l,\theta} - L_{l,h}F^j)$.
\begin{alemma}Under Assumption \ref{A3}, $\Omega_F$ is positive definite and the eigenvalues of $\Omega_F$ are greater than $c_{\delta}$, which is as defined in Assumption \ref{A3}$(ii)$.
\label{positive_definite_Omega}
\end{alemma}
\begin{proof}Fix a non-zero vector $\zeta \in \R^p$ and write $\zeta_{*} = \zeta/\norm{\zeta}_{l^2}$. We have 
\begin{align*}
\zeta^{T} \Omega_F \zeta &=  \sum_{l=1}^{m}\zeta^{T} V_{l,\tau_0}(L_{l,\theta} - L_{l,h}F , L_{l,\theta} - L_{l,h}F) \zeta\\
&=\norm{\zeta}_{l^2}^2 \sum_{l=1}^{m}V_{l,\tau_0}[L_{l,\theta}\zeta_{*} - L_{l,h}(\zeta_{*}^T F) , L_{l,\theta}\zeta_{*} - L_{l,h}(\zeta_{*}^T F)]\\
&> c_{\delta}\norm{\zeta}_{l^2}^2,
\end{align*}
where the last inequality holds by Assumption \ref{A3}$(ii)$ because $\norm{\zeta_*}_{l^2}=1$ and $\zeta_*^{T}F \in \CH$. Therefore, $\Omega_F$ is positive definite.

Let $\delta$ be any eigenvalue of $\Omega_F$, and let $\zeta_{\delta} \in \R^p$ be a unit eigenvector associated with $\delta$. By definition, we have $\Omega_F \zeta_{\delta} = \delta \zeta_{\delta}$. We have
\begin{align*}
\delta &= \zeta_{\delta}^{T}\delta\zeta_{\delta} = \zeta_{\delta}^{T}\Omega_F \zeta_{\delta} =\sum_{l=1}^{m}V_{l,\tau_0}[L_{l,\theta}\zeta_{\delta} - L_{l,h}(\zeta_{\delta}^T F) , L_{l,\theta}\zeta_{*} - L_{l,h}(\zeta_{*}^T F)] > c_{\delta}.
\end{align*}
\end{proof} 
\begin{atheorem}
Suppose Assumption \ref{A3} holds. Then $\langle \cdot, \cdot \rangle_{\Q}$ given by \eqref{inner_product} is a well-defined inner product on $\Q$, and $\Q$ is complete with respect to the norm $\norm{\cdot}_{\Q}$ induced by this inner product. Hence, $\Q$ is a Hilbert space. 
\label{inner_product_theorem}
\end{atheorem}
\begin{proof} It is easy to check that \eqref{inner_product} satisfies symmetry, linearity and positive semi-definiteness for an inner product. If $(\zeta,g)=0$, $\langle (\zeta,g),(\zeta,g)\rangle_{\Q}=0$ is obvious. We will now show that $\langle (\zeta,g),(\zeta,g)\rangle_{\Q}=0$ implies $(\zeta,g)=0$. We see that 
\begin{align}
\langle (\zeta,g),(\zeta,g) \rangle_{\Q} &= \sum_{l=1}^{m} V_{l,\tau_0}(L_{l,\theta}\zeta + L_{l,h}g, L_{l,\theta}\zeta + L_{l,h}g) + \lambda J(g,g) \nonumber\\
&= \sum_{l=1}^{m} \left[ \zeta^{T} V_{l,\tau_0}(L_{l,\theta},L_{l,\theta})\zeta + 2 \zeta^{T}V_{l,\tau_0}(L_{l,\theta},L_{l,h}g)+V_{l,\tau_0}(L_{l,h}g,L_{l,h}g)\right] + \lambda J(g,g) \nonumber\\
& = \sum_{l=1}^{m} \zeta^{T} \left[ V_{l,\tau_0}(L_{l,\theta} - L_{l,h}F,L_{l,\theta} - L_{l,h}F) -  V_{l,\tau_0}(L_{l,h}F,L_{l,h}F)  + 2 V_{l,\tau_0}(L_{l,\theta},L_{l,h}F)\right]\zeta \nonumber \\
& \quad + 2 \zeta^{T} \sum_{l=1}^{m} V_{l,\tau_0}(L_{l,\theta},L_{l,h}g) + \langle g,g \rangle_{\CH} \nonumber\\
&=\zeta^T \Omega_F \zeta +\zeta^T \langle F,F \rangle_{\CH} \zeta +2 \zeta^T \langle F,g \rangle_{\CH} + \langle g, g \rangle_{\CH} + \lambda J(\zeta^T F, \zeta^T F) \nonumber\\
&=\zeta^T \Omega_F \zeta + \langle \zeta^T F + g, \zeta^T F +g \rangle_{\CH}+ \lambda J(\zeta^T F, \zeta^T F) \geq 0,
\label{positive_inner}
\end{align}
and every term in \eqref{positive_inner} is non-negative. If $\langle (\zeta,g),(\zeta,g) \rangle=0$, the first term in \eqref{positive_inner} implies $\zeta=0$ by Lemma \ref{positive_definite_Omega}. This further implies that 
\[\langle \zeta^T F + g, \zeta^T F +g \rangle_{\CH} = \langle g, g \rangle_{\CH} =0.\]
Therefore, $g=0$ because $\langle \cdot, \cdot \rangle_{\CH}$ is an inner product on $\CH$. Hence, $\langle \cdot,\cdot \rangle_{\Q}$ is a well-defined inner product on $\Q$. 

Next, we want to show that $\Q$ is complete with respect to the norm $\norm{\cdot}_{\Q}$. Let $\{(\zeta_i,g_i)\}_{i=1}^{\infty} \subset \Q$ be a Cauchy sequence. For any $\epsilon >0$, there exist a positive integer $M$ such that for all $i,j>M$, we have 
\[\norm{(\zeta_i,g_i)-(\zeta_j,g_j)}_{\Q}^2 = \sum_{l=1}^{m} V_{l,\tau_0}[L_{l,\theta}(\zeta_i - \zeta_j) + L_{l,h}(g_i-g_j)] + \lambda J(g_i-g_j) \leq \epsilon. \]
This implies that 
\[\sum_{l=1}^{m} V_{l,\tau_0}[L_{l,\theta}(\zeta_i - \zeta_j) + L_{l,h}(g_i-g_j)] = \norm{\zeta_i - \zeta_j}_{l^2}^2 \sum_{l=1}^{m} V_{l,\tau_0}[L_{l,\theta}(\zeta_i - \zeta_j)^* + L_{l,h}(g_i-g_j)^*] \leq \epsilon,\]
where $(\zeta_i - \zeta_j)^* = (\zeta_i - \zeta_j)/\norm{\zeta_i - \zeta_j}_{l^2}$, and $(g_i-g_j)^* = (g_i-g_j)/\norm{\zeta_i - \zeta_j}_{l^2}$. By Assumption \ref{A3}$(ii)$, $\sum_{l=1}^{m} V_{l,\tau_0}[L_{l,\theta}(\zeta_i - \zeta_j)^* + L_{l,h}(g_i-g_j)^*]>c_{\delta}$ for some positive constant $c_{\delta}$ as defined in \eqref{lower_bound_c_delta}. Therefore,
\begin{equation}
\norm{\zeta_i - \zeta_j}_{l^2}^2  \leq \frac{\epsilon}{c_{\delta}},
\label{zeta_cauchy}
\end{equation}
and $\{\zeta_i\}_{i=1}^{\infty}$ is a Cauchy sequence in $\R^p$ under the Euclidean norm, which therefore converges to some limit $\zeta_{\infty} \in \R^p$.

To find a limit for the sequence $\{g_i\}_{i=1}^{\infty}$ in $\CH$, we consider
\begin{align}
&\norm{\zeta_i - \zeta_j}_{\R^p}^2 + \norm{g_i - g_j}_{\CH}^2 - 2\norm{\zeta_i - \zeta_j}_{\R^p}\norm{g_i - g_j}_{\CH}\nonumber\\
\leq & \norm{\zeta_i - \zeta_j}_{\R^p}^2 + \norm{g_i - g_j}_{\CH}^2 - 2\abs{\sum_{l=1}^{m} V_{l,\tau_0}[L_{l,\theta}(\zeta_i - \zeta_j),L_{l,h}(g_i-g_j)]}\nonumber\\
\leq & \norm{\zeta_i - \zeta_j}_{\R^p}^2 + \norm{g_i - g_j}_{\CH}^2 + 2\sum_{l=1}^{m} V_{l,\tau_0}[L_{l,\theta}(\zeta_i - \zeta_j),L_{l,h}(g_i-g_j)]\nonumber\\
= & \sum_{l=1}^{m} V_{l,\tau_0}[L_{l,\theta}(\zeta_i - \zeta_j) + L_{l,h}(g_i-g_j)] + \lambda J(g_i-g_j)\nonumber\\
=& \norm{(\zeta_i,g_i)-(\zeta_j,g_j)}_{\Q}^2 \leq \epsilon
\label{tofind_gi_limit}
\end{align}
where the first inequality follows from the Cauchy-Schwarz inequality, and the second inequality follows from the triangle inequality. For $a>0, b>0$, we have $(1/4)a + b - a^{1/2}b^{1/2} = [(1/2)a^{1/2} - b^{1/2}]^2 \geq 0$, and it follows that $2a^{1/2}b^{1/2} \leq (1/2)a + 2b$. For $a = \norm{g_i - g_j}_{\CH}^2$ and $b=\norm{\zeta_i - \zeta_j}_{\R^p}^2$, \eqref{tofind_gi_limit} becomes
\begin{align}
\norm{\zeta_i - \zeta_j}_{\R^p}^2 + \norm{g_i - g_j}_{\CH}^2
\leq \epsilon + 2\norm{\zeta_i - \zeta_j}_{\R^p}\norm{g_i - g_j}_{\CH}
\leq \epsilon + \frac{1}{2}\norm{g_i - g_j}_{\CH}^2 + 2\norm{\zeta_i - \zeta_j}_{\R^p}^2.
\label{gi_cauchy}
\end{align}
Since $\norm{\cdot}_{\R^p}$ is equivalent to $\norm{\cdot}_{l^2}$ on $\R^p$, $\norm{\zeta_i - \zeta_j}_{\R^p}^2 \leq C\epsilon$ for some positive constant $C$ by \eqref{zeta_cauchy}. Therefore, after rearranging \eqref{gi_cauchy}, we get 
$\norm{g_i - g_j}_{\CH}^2 \leq (2+C)\epsilon$.
Hence, $\{g_i\}_{i=1}^{\infty}$ is a Cauchy sequence in $\CH$ under the norm $\norm{\cdot}_{\CH}$. By Assumption \ref{A3}$(i)$, this sequence converges to some limit $g_{\infty} \in \CH$. 

Lastly, we show that $(\zeta_i,g_i)$ converges to $(\zeta_{\infty},g_{\infty})$ in $\norm{\cdot}_{\Q}$. By the Cauchy-Schwarz inequality and triangle inequality, as $i \to \infty$, we have 
\begin{align*}
\norm{(\zeta_i,g_i)-(\zeta_{\infty},g_{\infty})}_{\Q}^2 &= \norm{\zeta_i - \zeta_{\infty}}_{\R^p}^2 + \norm{g_i - g_{\infty}}_{\CH}^2 + 2\sum_{l=1}^{m} V_{l,\tau_0}[L_{l,\theta}(\zeta_i - \zeta_{\infty}),L_{l,h}(g_i-g_{\infty})]\\
& \leq \norm{\zeta_i - \zeta_{\infty}}_{\R^p}^2 + \norm{g_i - g_{\infty}}_{\CH}^2 + 2\abs{\sum_{l=1}^{m} V_{l,\tau_0}[L_{l,\theta}(\zeta_i - \zeta_{\infty}),L_{l,h}(g_i-g_{\infty})]}\\
& \leq \norm{\zeta_i - \zeta_{\infty}}_{\R^p}^2 + \norm{g_i - g_{\infty}}_{\CH}^2 + 2\norm{\zeta_i - \zeta_{\infty}}_{\R^p}\norm{g_i - g_{\infty}}_{\CH} \to 0.
\end{align*}
Therefore, we conclude that $\Q$ is a Hilbert space with respect to the inner product $\langle \cdot,\cdot \rangle_{\Q}$.
\end{proof}

\subsection*{Proof of Lemma \ref{lemma1}}
Recall that $\theta = \left[\theta^k\right]_{k=1}^p$. Let $\mathcal{G}^{k}g = V_{\tau_0}[D_{\theta^k}\eta(\tau_0), D_h \eta(\tau_0) g]$, for any $g \in \CH$. Since 
\[ \abs{\mathcal{G}^{k} g} \leq V_{\tau_0}^{\frac{1}{2}}[D_{\theta^k}\eta(\tau_0)]V_{\tau_0}^{\frac{1}{2}}[D_h \eta(\tau_0) g] \leq C \norm{g}_{\CH}\]
for some positive constant $C$, $\mathcal{G}^k$ is a bounded linear functional on $\CH$. By the Riesz representation theorem, there exists a $G^k \in \CH$ such that for any $g \in \CH$,
\[\mathcal{G}^{k}g = V_{\tau_0}[D_{\theta^k}\eta(\tau_0), D_h \eta(\tau_0) g] = V_{\tau_0}[D_h \eta(\tau_0) G^k,D_h \eta(\tau_0) g] + \lambda J(G^k,g).\]
Let $G = \left[G^k\right]_{k=1}^p$. We define $V_{\tau_0}[D_{\theta}\eta(\tau_0), D_h \eta(\tau_0) g]$, $V_{\tau_0}[D_h \eta(\tau_0) G,D_h \eta(\tau_0) g]$, and $J(G,g)$ to be $p \times 1$ vectors whose $k$th entries are $V_{\tau_0}[D_{\theta^k}\eta(\tau_0), D_h \eta(\tau_0) g]$, $V_{\tau_0}[D_h \eta(\tau_0) G^k,D_h \eta(\tau_0) g]$, and $J(G^k,g)$, respectively. Therefore, 
\[  V_{\tau_0}[D_{\theta}\eta(\tau_0), D_h \eta(\tau_0) g]= V_{\tau_0}[D_h \eta(\tau_0) G,D_h \eta(\tau_0) g] + \lambda J(G,g)
.\]

The Fourier expansion of $G^k$ given the eigensystem discussed in Section \ref{Spectral decomposition} is 
\[G^k = \sum_{\nu} V_{\tau_0}[D_h \eta(\tau_0) G^k,D_h \eta(\tau_0) \phi_{0,\nu}] \phi_{0,\nu}.\] 
Simple calculation shows that 
\[V_{\tau_0}[D_h \eta(\tau_0) G^k,D_h \eta(\tau_0) \phi_{0,\nu}] = (1+\lambda \rho_{0,\nu})^{-1} V_{\tau_0}[D_{\theta^k}\eta(\tau_0), D_h \eta(\tau_0) \phi_{0,\nu}],\]
and hence,
\begin{align} \label{Gk}
G^k = \sum_{\nu} \frac{1}{1+\lambda \rho_{0,\nu}} V_{\tau_0}[D_{\theta^k}\eta(\tau_0), D_h \eta(\tau_0) \phi_{0,\nu}]\phi_{0,\nu}.
\end{align}

Recall from Section \ref{Linear approximation} that 
\[ \tilde{\theta} - \theta_0 = \Omega_{\lambda}^{-1} \brac{\bar{\alpha}_n - \sum_{\nu}\frac{\bar{\beta}_{\nu,n} - \lambda \rho_{0,\nu} h_{0,\nu}}{1+\lambda \rho_{0,\nu}}V_{\tau_0}[D_{\theta}\eta(\tau_0),D_{h}\eta(\tau_0)\phi_{0,\nu}]},\]
where 
\[\Omega_{\lambda} =  V_{\tau_0}[D_{\theta}\eta(\tau_0)] - \sum_{\nu} \frac{V_{\tau_0}[D_{\theta}\eta(\tau_0),D_{h}\eta(\tau_0)\phi_{0,\nu}]^{\otimes 2}}{1+\lambda \rho_{0,\nu}}.\]
The $(i,j)$th entry of this $p \times p$ matrix $\Omega_{\lambda}$ can be written as 
\begin{align*}
\Omega_{\lambda}^{i,j} &= V_{\tau_0}[D_{\theta^i}\eta(\tau_0),D_{\theta^j}\eta(\tau_0)] - \sum_{\nu} \frac{V_{\tau_0}[D_{\theta^i} \eta(\tau_0), D_h \eta(\tau_0) \phi_{0,\nu}]V_{\tau_0}[D_{\theta^j} \eta(\tau_0), D_h \eta(\tau_0) \phi_{0,\nu}]}{1+\lambda \rho_{0,\nu}}\\
&= V_{\tau_0}[D_{\theta^i}\eta(\tau_0),D_{\theta^j}\eta(\tau_0)] -  V_{\tau_0}
\left[ D_{\theta^i} \eta(\tau_0), \sum_{\nu} \frac{V_{\tau_0}[D_{\theta^j} \eta(\tau_0), D_h \eta(\tau_0) \phi_{0,\nu}]}{1+\lambda \rho_{0,\nu}}D_h \eta(\tau_0) \phi_{0,\nu} \right]\\
&= V_{\tau_0}[D_{\theta^i}\eta(\tau_0),D_{\theta^j}\eta(\tau_0)] -  V_{\tau_0}
\left[ D_{\theta^i} \eta(\tau_0), D_h \eta(\tau_0) \sum_{\nu} \frac{V_{\tau_0}[D_{\theta^j} \eta(\tau_0), D_h \eta(\tau_0) \phi_{0,\nu}]}{1+\lambda \rho_{0,\nu}} \phi_{0,\nu} \right]\\
&= V_{\tau_0}[D_{\theta^i}\eta(\tau_0),D_{\theta^j}\eta(\tau_0)] - V_{\tau_0}[D_{\theta^i}\eta(\tau_0), D_h \eta(\tau_0) G^j]\\
&= V_{\tau_0}[D_{\theta^i}\eta(\tau_0) -  D_h \eta(\tau_0) G^i,D_{\theta^j}\eta(\tau_0) -  D_h \eta(\tau_0) G^j] + \lambda J(G^i,G^j).
\end{align*}
Let $\Omega = V_{\tau_0}[D_{\theta}\eta(\tau_0) -  D_h \eta(\tau_0) G]$ and $\Sigma_{\lambda} = \lambda J(G)$ be the matrices such that 
\[\Omega^{i,j} = V_{\tau_0}[D_{\theta^i}\eta(\tau_0) -  D_h \eta(\tau_0) G^i,D_{\theta^j}\eta(\tau_0) -  D_h \eta(\tau_0) G^j],\] 
and $\Sigma_{\lambda}^{i,j}=\lambda J(G^i,G^j)$.
Thus, $\Omega_{\lambda} = \Omega + \Sigma_{\lambda}$.


We now prove some properties of $\Omega$ and $\Sigma_{\lambda}$, which will be used to establish the bound for $\E[ (\tilde{\theta}-\theta_0)^{T}(\tilde{\theta}-\theta_0)]$. 
\begin{alemma}\label{lemma3}
$\Sigma_{\lambda} \rightarrow 0$ as $\lambda \rightarrow 0$.
\end{alemma}
\begin{proof}
By \eqref{Gk}, we get that the $(i,j)$th entry of $\Sigma_{\lambda}$ is 
\[\lambda J (G^i,G^j) = \sum_{\nu}  \frac{\lambda \rho_{0,\nu}}{(1+\lambda \rho_{0,\nu})^2} V_{\tau_0}[D_{\theta^i}\eta(\tau_0), D_h \eta(\tau_0) \phi_{0,\nu}] V_{\tau_0}[D_{\theta^j}\eta(\tau_0), D_h \eta(\tau_0) \phi_{0,\nu}].\]
By square summability of $\{V_{\tau_0}[D_{\theta^i}\eta(\tau_0), D_h \eta(\tau_0) \phi_{0,\nu}]\}_{\nu \in \mathbb{N}}$ and the dominated convergence theorem, the above sum converges to 0 as $\lambda \rightarrow 0$. 
\end{proof}

\begin{alemma}\label{lemma4}
Under Assumption \ref{A3}, $\Omega$ is positive definite. Let $c_{\delta}$ and $C_1$ be constants defined in Assumptions \ref{A3} and \ref{A6} respectively, and let $\delta$ be any eigenvalue of $\Omega$. Then $\delta > C_1c_{\delta} = \tilde{c}_{\delta}$.
\end{alemma}
\begin{proof}Using Assumption \ref{A6}, the proof is similar to the one for Lemma \ref{positive_definite_Omega}.
\end{proof}

Note that by Lemma \ref{lemma4}, the eigenvalues of $\Omega$ have a uniform lower bound independent of $\lambda$. Then as $\lambda \rightarrow 0$, we have 
\begin{align*}
\E\left[ (\tilde{\theta} - \theta_{0})^{T}(\tilde{\theta} - \theta_{0}) \right]&= \E\left[ a_n^{T}(\Omega + \Sigma_{\lambda})^{-2}a_n\right]\rightarrow \E\left[ a_n^{T}\Omega^{-2}a_n\right]\leq \tilde{c}_{\delta}^{-2} \E\left[ a_n^{T}a_n \right] = \tilde{c}_{\delta}^{-2} \sum_{i=1}^{p} \E\left[ (a_{n}^i)^2\right],
\end{align*}
where 
\[a_n = \bar{\alpha}_n - \sum_{\nu}\frac{\bar{\beta}_{\nu,n} - \lambda \rho_{0,\nu} h_{0,\nu}}{1+\lambda \rho_{0,\nu}}V_{\tau_0}[D_{\theta}\eta(\tau_0),D_{h}\eta(\tau_0)\phi_{0,\nu}],\]
and $a_{n}^i$ is the $i$th entry of $a_n$.

Before we proceed to derive the bound for $\E\left[( a_{n}^i)^2\right]$, we also need the following lemma, which is the same as Lemma 5.2 in \citeasnoun{gu_qiu_1993}.
\begin{alemma}\label{lemma6}
Under Assumption \ref{A9}, as $\lambda \rightarrow 0$,
\begin{align*}
\sum_{\nu} \frac{\lambda \rho_{0,\nu}}{(1+\lambda \rho_{0,\nu})^2} &= O (\lambda^{-\frac{1}{r}}),\quad \sum_{\nu} \frac{1}{(1+\lambda \rho_{0,\nu})^2} = O (\lambda^{-\frac{1}{r}}), \quad \sum_{\nu} \frac{1}{1+\lambda \rho_{0,\nu}} = O (\lambda^{-\frac{1}{r}}).
\end{align*}
\end{alemma}

We now return to our analysis for $\E\left[(a_{n}^i)^2\right]$. We have 
\begin{align*}
\E\left[ (a_{n}^{i})^2\right] &= \E\brac{\left[ \bar{\alpha}_{n}^{i} - \sum_{\nu}\frac{\bar{\beta}_{\nu,n} - \lambda \rho_{0,\nu} h_{0,\nu}}{1+\lambda \rho_{0,\nu}}V_{\tau_0}[D_{\theta^i}\eta(\tau_0),D_{h}\eta(\tau_0)\phi_{0,\nu}] \right]^2}\\
&\leq 2 \brac{ \E\left[ (\bar{\alpha}_{n}^{i})^2 \right] + \E\left[ \left(\sum_{\nu}\frac{\bar{\beta}_{\nu,n} - \lambda \rho_{0,\nu} h_{0,\nu}}{1+\lambda \rho_{0,\nu}}V_{\tau_0}[D_{\theta^i}\eta(\tau_0),D_{h}\eta(\tau_0)\phi_{0,\nu}]\right)^2 \right]  }.
\end{align*}
Note that $ \E\left[ (\bar{\alpha}_{n}^{i})^2 \right] = O(n^{-1})$, and by square summability of $\{V_{\tau_0}[D_{\theta^i}\eta(\tau_0),D_{h}\eta(\tau_0)\phi_{0,\nu}]\}_{\nu \in \mathbb{N}}$ and $\{h_{0,\nu}\}_{\nu \in \mathbb{N}}$, the dominated convergence theorem, the Cauchy-Schwarz inequality, $\E(\bar{\beta}^2_{\nu,n}) = n^{-1}$, and Lemma \ref{lemma6}, we have
\begin{align*}
&\E\left[ \left(\sum_{\nu}\frac{\bar{\beta}_{\nu,n} - \lambda \rho_{0,\nu} h_{0,\nu}}{1+\lambda \rho_{0,\nu}}V_{\tau_0}[D_{\theta^i}\eta(\tau_0),D_{h}\eta(\tau_0)\phi_{0,\nu}]\right)^2 \right]\\
&\rightarrow \E\left[ \left(\sum_{\nu}\frac{\bar{\beta}_{\nu,n}}{1+\lambda \rho_{0,\nu}}V_{\tau_0}[D_{\theta^i}\eta(\tau_0),D_{h}\eta(\tau_0)\phi_{0,\nu}]\right)^2 \right]\leq C\left[ \sum_{\nu} \frac{\E(\bar{\beta}^2_{\nu,n})}{(1+\lambda \rho_{0,\nu})^2}\right]=O(n^{-1}\lambda^{-\frac{1}{r}}).
\end{align*}
Therefore, we conclude that as $\lambda \rightarrow 0$,  $\E\left[ (a_{n}^{i})^2\right] = O(n^{-1}\lambda^{-\frac{1}{r}})$, which implies that 
\[\E\brac{ V_{\tau_0} [D_{\theta} \eta(\tau_0) (\tilde{\theta} - \theta_0)]} \leq c \E \left[ (\tilde{\theta}-\theta_0)^{T}(\tilde{\theta}-\theta_0) \right] = O(n^{-1}\lambda^{-\frac{1}{r}}).\]
This concludes the proof for the first bound in Lemma \ref{lemma1}. 

Now for the second bound in Lemma \ref{lemma1}, we see that 
\begin{align*}
&V_{\tau_0}[D_{h}\eta(\tau_0)(\tilde{h}-h_0)] = \sum_{\nu} \left( \tilde{h}_{\nu} - h_{0,\nu} \right)^2,\quad \lambda J(\tilde{h}- h_0) = \sum_{\nu}\lambda \rho^0_{\nu} \left( \tilde{h}_{\nu} - h_{0,\nu} \right)^2.
\end{align*}
Plugging in the formula of $\tilde{h}_{\nu}$ given in Section \ref{Linear approximation}, we get 
\begin{align*}
&\sum_{\nu} \left( \tilde{h}_{\nu} - h_{0,\nu} \right)^2 \\
&= \sum_{\nu} \brac{\frac{\bar{\beta}_{\nu,n} -\lambda \rho_{0,\nu} h_{0,\nu}}{1+\lambda \rho_{0,\nu}} - (\tilde{\theta}-\theta_0)^{T} \frac{V_{\tau_0}[D_{\theta}\eta(\tau_0),D_{h}\eta(\tau_0)\phi_{0,\nu}]}{1+\lambda \rho_{0,\nu}}}^2\\
&\leq C \brac{\sum_{\nu} \left(\frac{\bar{\beta}_{\nu,n} -\lambda \rho_{0,\nu} h_{0,\nu}}{1+\lambda \rho_{0,\nu}}\right)^2 + \sum_{\nu} \left[ (\tilde{\theta}-\theta_0)^{T} \frac{V_{\tau_0}[D_{\theta}\eta(\tau_0),D_{h}\eta(\tau_0)\phi_{0,\nu}]}{1+\lambda \rho_{0,\nu}}\right]^2}\\
&=C\left[ (I) + (II) \right].
\end{align*}
Since $\E(\bar{\beta}_{\nu,n}) = 0$, $\E(\bar{\beta}_{\nu,n}^2) = \frac{1}{n}$, and $\sum_{\nu} \rho_{0,\nu}h^2_{0,\nu} = J(h_0) < \infty$, by Lemma \ref{lemma6}, we have  
\begin{align*}
\E[(I)] 
&= n^{-1} \sum_{\nu} \frac{1}{(1+\lambda \rho_{0,\nu})^2} + \lambda \sum_{\nu}\frac{\lambda \rho_{0,\nu}}{(1+\lambda \rho_{0,\nu})^2} \rho_{0,\nu} h^2_{0,\nu}= O(n^{-1}\lambda^{-\frac{1}{r}} + \lambda).
\end{align*} 
If $a,b$ are $p \times 1$ vectors, then $a^Tb$ is $1\times 1$, which implies that $a^Tb = b^Ta$ and $(a^Tb)^2 = b^{T}aa^Tb$. Using this fact and the bound for $\E[ (\tilde{\theta}-\theta_0)^{T}(\tilde{\theta}-\theta_0)]$, we have 
\begin{align*}
\E[(II)] 
&=\E\brac{\sum_{\nu} \frac{V_{\tau_0}[D_{\theta}\eta(\tau_0),D_{h}\eta(\tau_0)\phi_{0,\nu}]^{T}}{1+\lambda \rho_{0,\nu}}(\tilde{\theta}-\theta_0)(\tilde{\theta}-\theta_0)^{T} \frac{V_{\tau_0}[D_{\theta}\eta(\tau_0),D_{h}\eta(\tau_0)\phi_{0,\nu}]}{1+\lambda \rho_{0,\nu}}}\\
&=\sum_{i,j=1}^{p} \left\lbrace \E\left[ (\tilde{\theta} - \theta_0)^i (\tilde{\theta} - \theta_0)^j \right] \right.\\
&\quad \quad \quad \left. \cdot \sum_{\nu}\frac{1}{(1+\lambda \rho_{0,\nu})^2}V_{\tau_0}[D_{\theta^i}\eta(\tau_0),D_{h}\eta(\tau_0)\phi_{0,\nu}]V_{\tau_0}[D_{\theta^j}\eta(\tau_0),D_{h}\eta(\tau_0)\phi_{0,\nu}] \right\rbrace\\
&\leq C \sum_{i,j=1}^{p}\E\left[ (\tilde{\theta} - \theta_0)^i (\tilde{\theta} - \theta_0)^j \right]\\
&\leq C \sum_{i,j=1}^{p}\brac{\E\left[ (\tilde{\theta}^i - \theta_0^i)^2\right]}^{\frac{1}{2}}\brac{\E\left[ (\tilde{\theta}^j - \theta_0^j)^2\right]}^{\frac{1}{2}}\\
&= O(n^{-1}\lambda^{-\frac{1}{r}}).
\end{align*}
Therefore,
\begin{align*}
\E\left[ V_{\tau_0}[D_{h}\eta(\tau_0)(\tilde{h}-h_0)] \right] 
= O(n^{-1}\lambda^{-\frac{1}{r}} + \lambda). 
\end{align*}
Similar analysis shows that  
\begin{align*}
\E\left[ \lambda J(\tilde{h}-h_0)\right] &= \E\left[\sum_{\nu}\lambda \rho_{0,\nu} \left( \tilde{h}_{\nu} - h_{0,\nu} \right)^2 \right] 
= O(n^{-1}\lambda^{-\frac{1}{r}} + \lambda). 
\end{align*} \renewcommand\qedsymbol{$\blacksquare$}
Hence, the second bound in Lemma \ref{lemma1} is established. 

\subsection*{Proof of Theorem \ref{theorem1}}
By Assumption \ref{A6}, Lemma \ref{lemma1} implies that 
\begin{align*}
\E\brac{ V_{\tau_0} [L_{\theta} (\tilde{\theta} - \theta_0)]} = O(n^{-1}\lambda^{-\frac{1}{r}}),\quad \E\brac{ V_{\tau_0}[L_h(\tilde{h}-h_0)] + \lambda J(\tilde{h}-h_0)} = O(n^{-1}\lambda^{-\frac{1}{r}} + \lambda).
\end{align*}
%
By the Cauchy-Schwarz inequality and completing the square, we have 
\begin{align*}
&V_{\tau_0}[L_{\theta} (\tilde{\theta} - \theta_0) + L_h(\tilde{h}-h_0)]\\
&= V_{\tau_0}[L_{\theta} (\tilde{\theta} - \theta_0) ] +V_{\tau_0}[ L_h(\tilde{h}-h_0)] + 2V_{\tau_0}[L_{\theta} (\tilde{\theta} - \theta_0), L_h(\tilde{h}-h_0)]\\
& \leq V_{\tau_0}[L_{\theta} (\tilde{\theta} - \theta_0) ] +V_{\tau_0}[ L_h(\tilde{h}-h_0)] + 2 V_{\tau_0}^{\frac{1}{2}}[L_{\theta} (\tilde{\theta} - \theta_0) ] V_{\tau_0}^{\frac{1}{2}}[ L_h(\tilde{h}-h_0)]\\
&= \left[ V_{\tau_0}^{\frac{1}{2}}[L_{\theta} (\tilde{\theta} - \theta_0) ] + V_{\tau_0}^{\frac{1}{2}}[ L_h(\tilde{h}-h_0)]\right] ^{2}\\
&= O_{p}(n^{-1}\lambda^{-\frac{1}{r}} + \lambda).
\end{align*}
Together with $\lambda J(\tilde{h}-h_0) = O_{p}(n^{-1}\lambda^{-\frac{1}{r}} + \lambda)$, we have the desired result.
\subsection*{Proof of existence and uniqueness of $(\hat{\theta},\hat{h})$}
In this section, adapting the frameworks in \citeasnoun{cox_1988}, \citeasnoun{CoxOsullivan}, \citeasnoun{o’sullivan_1990}, and \citeasnoun{ke2004existence},  we establish the existence and uniqueness of the semiparametric estimator of the penalized likelihood given by (3) for the case when $m=1$ in the main text of the paper. The proof for the case when $m>1$ can be carried out in a similar manner. Note that we have the same assumptions as given in Section 4.2 for the proof of consistency in the main text, with the exception of assuming the existence of $(\hat{\theta},\hat{h})$ in $\N_{\theta_0} \times \N_{h_0}$. For convenience, we drop the subscript $l=1$. Denote the $k$th order partial Fr\'echet derivative operators by $D^k_{a_1\ldots a_k} = D_{a_1} \ldots D_{a_k}$, where $a_i \in \{{\theta},h\}$ for $i = 1, \ldots, k$.
\subsubsection*{Linearization}
We extend the linearization technique used to approximate the systematic and stochastic components of the estimation error in \citeasnoun{CoxOsullivan} and \citeasnoun{o’sullivan_1990} to our semiparametric setting by using a bivariate Taylor series expansion for a nonlinear operator. Note that \citeasnoun{ke2004existence} used a similar analysis to show existence and uniqueness of a smoothing spline nonlinear nonparametric regression model. We first state the following proposition, whose proof is provided in \citeasnoun{ke2004existence}.
\begin{aprop}
Let $f: D(f) \subset X \times Y \to Z$, where $X$, $Y$ and $Z$ are Banach spaces. If $f''$ exists at $(x,y)$, then the partial Fr\'echet derivatives $f_{xx}$, $f_{xy}$, $f_{yx}$ and $f_{yy}$ exist at $(x,y)$. For any $h,a \in X$, $k,b \in Y$, 
\[f''(x,y)(h,k)(a,b) = f_{xx}(x,y)ha +f_{xy}(x,y)ka + f_{yx}(x,y)hb + f_{yy}(x,y)kb.\]
\end{aprop}
By above theorem and the Taylor formula given in Chapter 1 Section 4 in \citeasnoun{ambrosetti1995primer}, we can write the first order Taylor series expansion of $f(x,y)$ as
\[f(x+h,y+k) = f(x,y) + f_x(x,y)h + f_y(x,y)k + R,\]
where $R$ is the remainder, given by
\begin{align*}
R &= \int_0^1 (1-t) f''(x+t h, y+t k)(h,k)(h,k) dt\\
&= \int_0^1 (1-t) \left[f_{xx}(x+t h, y+tk)hh + f_{xy}(x+t h, y+t k)kh\right.\\
&\qquad+ \left.f_{yx}(x+t h, y+t k)hk + f_{yy}(x+t h, y+tk)kk\right]\ dt.
\end{align*}

\subsubsection*{Linear expansions}\label{linear_expansion}
Since $D_{\theta}\mf{L}_{\lambda}(\theta,h)$ is a bounded linear functional on $\R^p$, by the Riesz representation theorem, there exists $Z_{\theta}(\theta,h) \in \R^p$ such that for any $a \in \R^p$, 
\[D_{\theta}\mf{L}_{\lambda}(\theta,h)a = \ip{Z_{\theta}(\theta,h),a}_{\R^p}.\]
Similarly, we can denote the Riesz representer of $D_h\mf{L}_{\lambda}(\theta,h)$ in $\CH$ by $Z_h(\theta,h) \in \CH$. For convenience, we use 
\[Z_{\theta}(\theta,h) = D_{\theta} \mf{L}_{\lambda}(\theta,h) \quad\text{ and }\quad Z_h(\theta,h) = D_h\mf{L}_{\lambda}(\theta,h)\]
to represent either the functionals or their Riesz representers in $\R^p$ and $\CH$, respectively.  
For any $\theta_0+a \in \N_{\theta_0}$ and $h_0+g \in \N_{h_0}$, the first order Taylor series expansions of $Z_h, Z_{\theta}$ at the true parameter $(\theta_0,h_0)$ are
\begin{align*}
Z_h(\theta_0+a, h_0+g) &= Z_h(\theta_0,h_0) + D_{\theta}Z_h(\theta_0,h_0)a + D_h Z_h(\theta_0,h_0)g +R_{h}(\theta_0,h_0)ag,\\
Z_{\theta}(\theta_0+a, h_0+g) &= Z_{\theta}(\theta_0,h_0) + D_{\theta}Z_{\theta}(\theta_0,h_0)a + D_h Z_{\theta}(\theta_0,h_0)g + R_{\theta}(\theta_0,h_0)ag,
\end{align*}
where 
\begin{align*}
R_h(\theta,h)ag &= \int_0^1 (1-t)\left[D_{\theta \theta}^2 Z_h(\theta + ta, h+t g)aa + D^2_{h \theta} Z_h(\theta + t a, h+tg)ag\right.\\
&\qquad+ \left.D^2_{\theta h} Z_h(\theta+ta, h+t g)ga + D^2_{h h} Z_h(\theta + ta, h+t g)gg\right]\ dt,\\
R_{\theta}(\theta,h)ag &= \int_0^1 (1-t)\left[D_{\theta \theta}^2 Z_{\theta}(\theta + t a, h+t g)aa + D^2_{h \theta} Z_{\theta}(\theta + t a, h+t g)ag\right.\\
&\qquad+ \left.D^2_{\theta h} Z_{\theta}(\theta+t a, h+tg)ga + D_{hh}^2 Z_{\theta}(\theta + t a, h+t g)gg\right]\ dt.
\end{align*}
For $Z_h(\theta,h)$, $u,v \in \CH$, $a \in \R^p$, we have
\begin{align*}
Z_h(\theta,h)u &= -\mu_{\tau_0}[D_h\eta(\theta,h)u] + \mu_{\tau}[D_h \eta(\theta,h)u] + \lambda J(h,u),\\
D_h Z_h(\theta,h)uv &= -\left\{\mu_{\tau_0}[D_{hh}^2 \eta(\theta,h)uv] - \mu_{\tau}[D_{hh}^2 \eta(\theta,h)uv]\right\}\\
& \quad + \left\{\mu_{\tau}[D_h \eta(\theta,h)v \cdot D_h \eta(\theta,h)u] - \mu_{\tau}[D_h \eta(\theta,h)v] \mu_{\tau}[D_h \eta(\theta,h)u]\right\}\\
& \quad + \lambda J(v,u),\\
D_{\theta} Z_h(\theta,h)ua &= -\left\{\mu_{\tau_0}[D^2_{\theta h} \eta(\theta,h)ua] - \mu_{\tau}[D^2_{\theta h} \eta(\theta,h)ua]\right\}\\
& \quad + \left\{\mu_{\tau}[D_{\theta} \eta(\theta,h)a \cdot D_h \eta(\theta,h)u] - \mu_{\tau}[D_{\theta} \eta(\theta, h)a] \mu_{\tau}[D_h \eta(\theta, h)u]\right\}.
\end{align*}
For $Z_{\theta}(\theta,h)$, $a,b \in \R^p$, $u \in \CH$, we have
\begin{align*}
Z_{\theta}(\theta,h)a &= -\mu_{\tau_0}[D_{\theta}\eta(\theta,h)a] + \mu_{\tau}[D_{\theta} \eta(\theta,h)a],\\
D_h Z_{\theta}(\theta,h)au &= -\left\{\mu_{\tau_0}[D^2_{h \theta} \eta(\theta,h)au] - \mu_{\tau}[D^2_{h \theta} \eta(\theta,h)au]\right\}\\
&\quad + \left\{\mu_{\tau}[D_{\theta} \eta(\theta,h)a \cdot D_h \eta(\theta,h)u] - \mu_{\tau}[D_{\theta} \eta(\theta,h)a] \mu_{\tau}[D_h \eta(\theta,h)u]\right\},\\
D_{\theta} Z_{\theta}(\theta,h)ab &= -\left\{\mu_{\tau_0}[D_{\theta \theta}^2 \eta(\theta,h)ab] - \mu_{\tau}[D_{\theta \theta}^2 \eta(\theta,h)ab]\right\}\\
&\quad + \left\{\mu_{\tau}[D_{\theta} \eta(\theta,h)a \cdot D_{\theta} \eta(\theta,h)b] - \mu_{\tau}[D_{\theta} \eta(\theta,h)a] \mu_{\tau}[D_{\theta} \eta(\theta,h)b]\right\}.
\end{align*}
For any $u,v \in \CH$, ${a,b \in \R^p}$, and $\tau = (\theta,h) \in \N_{\theta_0} \times \N_{h_0}$, define the operators $U_{\theta}(\theta,h)$ and $U_h(\theta,h)$ on $\R^p$ and $\CH$, respectively, such that
\begin{align*}
\ip{u,U_h(\theta,h)v}_{\CH} &= V_{\tau}[D_h\eta(\theta,h)u, D_h\eta(\theta,h)v],\\
\ip{a,U_{\theta}(\theta,h)b}_{\R^p} &= V_{\tau}[D_{\theta}\eta(\theta,h)a, D_{\theta}\eta(\theta,h)b].
\end{align*}
Note that these operators are well-defined by the Riesz representation theorem applied to the linear functionals
\[v \mapsto V_{\tau}[D_h\eta(\theta,h)u, D_h\eta(\theta,h)v], \qquad b \mapsto V_{\tau}[D_{\theta}\eta(\theta,h)a, D_{\theta}\eta(\theta,h)b],\]
which are bounded in the corresponding norms on $\CH$ and $\R^p$, respectively. Similarly, we also define $U_{h\theta}(\theta,h): \CH \to \R^p$ and $U_{\theta h}(\theta,h): \R^p \to \CH$ by
\begin{align*}
\ip{a, U_{h\theta}(\theta,h)u}_{\R^p} = V_{\tau}[D_{\theta}\eta(\theta,h)a, D_h\eta(\theta,h)u] = \ip{u,U_{\theta h}(\theta,h) a}_{\CH}.
\end{align*}
By Lemma S.2 in the supplement of \citeasnoun{cheng2015joint}, there exists a bounded linear operator $W_{\lambda}$ on $\CH$ such that
\[\ip{u,W_{\lambda}v} = \lambda J(u,v).\]
Therefore, 
\begin{equation}\label{systematic_linearizations}
\begin{aligned}
Z_h(\theta_0+a, h_0+g) &= Z_h(\theta_0,h_0) + G_h(\theta_0,h_0)g + U_{\theta h}(\theta_0,h_0)a + R_h(\theta_0,h_0)ag,\\
Z_{\theta}(\theta_0+a, h_0+g) &= Z_{\theta}(\theta_0,h_0) + U_{\theta}(\theta_0,h_0)a + U_{h\theta}(\theta_0,h_0)g + R_{\theta}(\theta_0,h_0)ag,
\end{aligned}
\end{equation}
where $G_h(\theta,h) = U_h(\theta,h) +  W_{\lambda}$. We provide the presentations of the remainder terms $R_h(\theta,h)ag$ and $R_{\theta}(\theta,h)ag$ in Section \ref{remainder_terms}. 

Suppose $(\theta_{\lambda},h_{\lambda})$ is a solution for $Z_{\theta}(\theta,h)= Z_h(\theta,h) = 0$. We define the systematic error as $(\theta_{\lambda} - \theta_0, h_{\lambda} - h_0)$. Ignoring the remainder terms, we can get an approximation to the systematic error by setting the system of equations \eqref{systematic_linearizations}  to 0 and solving for $\bar{\theta}_{\lambda} - \theta_0$, and $ \bar{h}_{\lambda}- h_0$, i.e.,
\begin{align*}
&Z_h(\theta_0,h_0) + G_h(\theta_0,h_0)(\bar{h}_{\lambda}- h_0) + U_{\theta h}(\theta_0,h_0)(\bar{\theta}_{\lambda} - \theta_0)=0,\\
&Z_{\theta}(\theta_0,h_0) + U_{\theta}(\theta_0,h_0)(\bar{\theta}_{\lambda} - \theta_0) + U_{h \theta}(\theta_0,h_0)(\bar{h}_{\lambda}- h_0)=0.
\end{align*}
By the Lax-Milgram theorem (Section 3.6 of \citeasnoun{aubin_1979}) and Assumptions 2 and 4 in the main text, for any $(\theta,h) \in \N_{\theta_0} \times \N_{h_0}$, the operators $G_h(\theta,h)$, $U_{\theta}(\theta,h)$ have bounded inverses on $\CH$ and $\R^p$, respectively. Let 
\begin{align*}
G_{hh}(\theta,h) &= (G_h-U_{\theta h} U_{\theta}^{-1}U_{h\theta})(\theta,h): \CH \to \CH,\\
G_{\theta\theta}(\theta,h) &= (U_{\theta} - U_{h\theta} G_{h}^{-1}U_{\theta h})(\theta,h) : \R^p \to \R^p.
\end{align*}
Assuming both operators above have bounded inverses for any ${(\theta,h) \in \N_{\theta_0} \times \N_{h_0}}$, we get 
\begin{align*}
\bar{h}_{\lambda} - h_0 &= -G_{hh}^{-1}(\theta_0,h_0) \left[Z_h(\theta_0,h_0) - U_{\theta h}(\theta_0,h_0) U_{\theta}^{-1}(\theta_0,h_0) Z_{\theta}(\theta_0,h_0)\right],\\
\bar{\theta}_{\lambda} - \theta_0 &= -G_{\theta\theta}^{-1}(\theta_0,h_0) \left[Z_{\theta}(\theta_0,h_0) - U_{h \theta}(\theta_0,h_0) G_h^{-1}(\theta_0,h_0) Z_h(\theta_0,h_0)\right].
\end{align*}

Next, we define the stochastic error as $(\hat{\theta} - \theta_{\lambda}, \hat{h} - h_{\lambda})$. Similar to the definition of $Z_{\theta}$ and $Z_{h}$, we let  
\[Z_{n \theta}(\theta,h) = D_{\theta} \mf{L}_{n,\lambda}(\theta,h)\quad\text{ and }\quad Z_{n h}(\theta,h) = D_{h}\mf{L}_{n,\lambda}(\theta,h).\]
The approximation of the stochastic errors can be obtained by the linearizations of $Z_{n\theta}$ and $Z_{nh}$. Since for $u,v \in \CH$ and $a,b \in \R^p$, we have
\begin{align*}
Z_{n\theta}(\theta,h)a  &= -\frac{1}{n} \sum_{i=1}^n D_{\theta}\eta(x_i;\theta,h)a + \mu_{\tau}\left[D_{\theta}\eta(\theta,h)a\right],\\
D_{\theta}Z_{n\theta}(\theta,h)ab &= -\left\{\frac{1}{n} \sum_{i=1}^n D_{\theta \theta}^2 \eta(x_i;\theta,h)ab - \mu_{\tau}\left[D_{\theta \theta}^2\eta(\theta,h)ab\right]\right\}\\
&\quad + \left\{\mu_{\tau}\left[D_{\theta}\eta(\theta,h)a \cdot D_{\theta}\eta(\theta,h)b\right] - \mu_{\tau}\left[D_{\theta}\eta(\theta,h)a\right] \mu_{\tau}\left[D_{\theta}\eta(\theta,h)b\right]\right\},\\
D_hZ_{n\theta}(\theta,h)au &= -\left\{\frac{1}{n} \sum_{i=1}^n D^2_{h \theta} \eta(x_i;\theta,h)au - \mu_{\tau}\left[D^2_{h \theta}\eta(\theta,h)au\right]\right\}\\
&\quad + \left\{\mu_{\tau}\left[D_{\theta}\eta(\theta,h)a \cdot D_h\eta(\theta,h)u\right] - \mu_{\tau}\left[D_{\theta}\eta(\theta,h)a\right] \mu_{\tau}\left[D_h\eta(\theta,h)u\right]\right\},\\
Z_{nh}(\theta,h)u  &= -\frac{1}{n} \sum_{i=1}^n D_h\eta(x_i;\theta,h)u + \mu_{\tau}\left[D_h\eta(\theta,h)u\right] + \lambda J(h,u),\\
D_{\theta}Z_{nh}(\theta,h)ua &= -\left\{\frac{1}{n} \sum_{i=1}^n D_{\theta h}^2 \eta(x_i;\theta,h)ua - \mu_{\tau}\left[D_{\theta h}^2 \eta(\theta,h)ua\right]\right\}\\
&\quad + \left\{\mu_{\tau}\left[D_h\eta(\theta,h)u \cdot D_{\theta}\eta(\theta,h)a\right] - \mu_{\tau}\left[D_h\eta(\theta,h)u\right] \mu_{\tau}\left[D_{\theta}\eta(\theta,h)a\right]\right\},\\
D_hZ_{n\theta}(\theta,h)uv &= -\left\{\frac{1}{n} \sum_{i=1}^n D_{hh}^2 \eta(x_i;\theta,h)uv - \mu_{\tau}\left[D_{hh}^2\eta(\theta,h)uv\right]\right\}\\
&\quad + \left\{\mu_{\tau}\left[D_h\eta(\theta,h)u \cdot D_h\eta(\theta,h)v\right] - \mu_{\tau}\left[D_h\eta(\theta,h)u\right] \mu_{\tau}\left[D_h\eta(\theta,h)v\right]\right\}\\
&\quad + \lambda J(v,u),
\end{align*}
for any $\theta_{\lambda}+a \in \N_{\theta_0}, h_{\lambda}+g \in \N_{h_0}$, the first order Taylor series expansions of $Z_{n \theta}$ and $Z_{n h}$ at $(\theta_{\lambda},h_{\lambda}) \in \N_{\theta_0} \times \N_{h_0}$ can be written as   
\begin{equation}\label{stochastic_linearization}
\begin{aligned}
Z_{n\theta}(\theta_{\lambda}+a, h_{\lambda}+g) &= Z_{n\theta}(\theta_{\lambda}, h_{\lambda}) + U_{\theta}(\theta_{\lambda}, h_{\lambda})a + U_{h\theta}(\theta_{\lambda}, h_{\lambda})g + e_{\theta}(\theta_{\lambda}, h_{\lambda})ag + R_{n\theta}(\theta_{\lambda}, h_{\lambda})ag,\\
Z_{nh}(\theta_{\lambda}+a, h_{\lambda}+g) &= Z_{nh}(\theta_{\lambda}, h_{\lambda}) + G_h(\theta_{\lambda}, h_{\lambda})g + U_{\theta h}(\theta_{\lambda}, h_{\lambda})a+ e_h(\theta_{\lambda}, h_{\lambda})ag + R_{nh}(\theta_{\lambda}, h_{\lambda})ag,
\end{aligned}
\end{equation}
where the error terms are given by
\begin{align*}
&e_{\theta}(\theta_{\lambda}, h_{\lambda})ag  = e_{\theta\theta}(\theta_{\lambda},h_{\lambda})a + e_{\theta h}(\theta_{\lambda},h_{\lambda})g,\\
&e_h (\theta_{\lambda},h_{\lambda})ag = e_{h\theta}(\theta_{\lambda},h_{\lambda})a + e_{hh}(\theta_{\lambda},h_{\lambda})g,\\
&e_{\theta\theta}(\theta,h)a =  -\frac{1}{n} \sum_{i=1}^{n} D_{\theta \theta}^2 \eta(X_{i};\theta,h)a + \mu_{\tau}\left[D_{\theta \theta}^2\eta(\theta,h)a\right],\\
&e_{\theta h}(\theta,h)g =  -\frac{1}{n} \sum_{i=1}^{n} D^2_{h \theta} \eta(X_{i};\theta,h)g + \mu_{\tau}\left[D_{\theta \theta}^2\eta(\theta,h)g\right],\\
&e_{h\theta}(\theta,h)a = -\frac{1}{n} \sum_{i=1}^{n} D^2_{\theta h}\eta(X_{i};\theta,h)a + \mu_{\tau}\left[D^2_{\theta h}\eta(\theta,h)a\right],\\
&e_{hh}(\theta,h)g = -\frac{1}{n} \sum_{i=1}^{n} D_{hh}^2\eta(X_{i};\theta,h)g + \mu_{\tau}\left[D_{hh}^2\eta(\theta,h)g\right],
\end{align*}
and $R_{nh}, R_{n\theta}$ are defined similarly to $R_{h}, R_{\theta}$ by replacing $Z_{h}, Z_{\theta}$ with $Z_{nh},Z_{n\theta}$, respectively. 
Recall that $(\hat{\theta},\hat{h})$ is the solution for $Z_{n\theta}(\theta,h)= Z_{nh}(\theta,h) = 0$. Dropping the error terms and the remainder terms, we can get an approximation to the stochastic error ${(\hat{\theta}-\theta_{\lambda},\hat{h} - h_{\lambda})}$ by setting the linearizations \eqref{stochastic_linearization} to 0 and solving for $\bar{\theta}_{n\lambda} - \theta_{\lambda}$, and $ \bar{h}_{n\lambda}- h_{\lambda}$. We get 
\begin{align*}
\bar{h}_{n\lambda} - h_{\lambda} = -G_{hh}^{-1}(\theta_{\lambda}, h_{\lambda})\left[Z_{nh}(\theta_{\lambda}, h_{\lambda}) - U_{\theta h}(\theta_{\lambda}, h_{\lambda}) U_{\theta}^{-1}(\theta_{\lambda}, h_{\lambda}) Z_{n\theta}(\theta_{\lambda}, h_{\lambda})\right],\\
\bar{\theta}_{n\lambda} - \theta_{\lambda} = -G_{\theta\theta}^{-1}(\theta_{\lambda}, h_{\lambda})\left[Z_{n\theta}(\theta_{\lambda}, h_{\lambda}) - U_{h\theta}(\theta_{\lambda}, h_{\lambda}) G_h^{-1}(\theta_{\lambda}, h_{\lambda}) Z_{nh}(\theta_{\lambda}, h_{\lambda})\right].
\end{align*}

\subsubsection*{Remainder terms} \label{remainder_terms}
By definition of $R_{h}(\theta,h)$ and $R_{\theta}(\theta,h)$ given in Section \ref{linear_expansion},  we need to find the second partial Fr\'echet derivatives of $Z_{\theta}(\theta,h)$ and $Z_{h}(\theta,h)$. 

For $u,v,w \in \CH$, $a,b,c \in \R^p$, we have
\begin{align*}
D_{hh}^2 Z_h(\theta,h)uvw &= -\left\{\mu_{\tau_0}\left[D_{hhh}^2 \eta(\theta,h) uvw\right] - \mu_{\tau}\left[D_{hhh}^2 \eta(\theta,h) uvw\right]\right\}\\
& \quad + \left\{\mu_{\tau}\left[D_{hh}^2\eta(\theta,h)uv \cdot D_h\eta(\theta,h)w\right] - \mu_{\tau}\left[D_h\eta(\theta,h)w\right] \mu_{\tau}\left[D_{hh}^2\eta(\theta,h)uv\right]\right\}\\
& \quad + \left\{\mu_{\tau}\left[D_{hh}^2\eta(\theta,h)vw \cdot D_h\eta(\theta,h)u\right] - \mu_{\tau}\left[D_h\eta(\theta,h)u\right] \mu_{\tau}\left[D_{hh}^2\eta(\theta,h)vw\right]\right\}\\
& \quad + \left\{\mu_{\tau}\left[D_{hh}^2\eta(\theta,h)uw \cdot D_h\eta(\theta,h)v\right] - \mu_{\tau}\left[D_h\eta(\theta,h)v\right] \mu_{\tau}\left[D_{hh}^2\eta(\theta,h)uw\right]\right\}\\
& \quad + \left\{\mu_{\tau}\left[D_h\eta(\theta,h)v \cdot D_h\eta(\theta,h)u \cdot D_h\eta(\theta,h)w\right]\right.\\
&\qquad - \mu_{\tau}\left[D_h\eta(\theta,h)v\right] \mu_{\tau}\left[D_h\eta(\theta,h)u \cdot D_h\eta(\theta,h)w\right]\\
&\qquad -\mu_{\tau}\left[D_h\eta(\theta,h)v \cdot D_h\eta(\theta,h)w\right] \mu_{\tau}\left[D_h\eta(\theta,h)u\right]\\
&\qquad -\mu_{\tau}\left[D_h\eta(\theta,h)v \cdot D_h\eta(\theta,h)u\right] \mu_{\tau}\left[D_h\eta(\theta,h)w\right]\\
&\qquad + \left.2\mu_{\tau}\left[D_h\eta(\theta,h)v\right] \mu_{\tau}\left[D_h\eta(\theta,h)u\right] \mu_{\tau}\left[D_h\eta(\theta,h)w\right]\right\},
\end{align*}
\begin{align*}
D_{\theta h}^2 Z_h(\theta,h)uva &= D_{h \theta}^2 Z_h(\theta,h)uav = D_{hh}^2 Z_{\theta}(\theta,h)auv\\
&= -\left\{\mu_{\tau_0}\left[D_{\theta}D_{hh}^2 \eta(\theta,h) uva\right] - \mu_{\tau}\left[D_{\theta}D_{hh}^2 \eta(\theta,h) uva\right]\right\}\\
& \quad + \left\{\mu_{\tau}\left[D_{hh}^2\eta(\theta,h)uv \cdot D_{\theta}\eta(\theta,h)a\right] - \mu_{\tau}\left[D_{hh}^2\eta(\theta,h)uv\right] \mu_{\tau}\left[D_{\theta}\eta(\theta,h)a\right]\right\}\\
& \quad + \left\{\mu_{\tau}\left[D_{\theta h}^2\eta(\theta,h)va \cdot D_h\eta(\theta,h)u\right] - \mu_{\tau}\left[D_{\theta h}^2\eta(\theta,h)va\right] \mu_{\tau}\left[D_h\eta(\theta,h)u\right]\right\}\\
& \quad + \left\{\mu_{\tau}\left[D_{\theta h}^2\eta(\theta,h)ua \cdot D_h\eta(\theta,h)v\right] - \mu_{\tau}\left[D_{\theta h}^2\eta(\theta,h)va\right] \mu_{\tau}\left[D_h\eta(\theta,h)v\right]\right\}\\
& \quad + \left\{\mu_{\tau}\left[D_h\eta(\theta,h)u \cdot D_h\eta(\theta,h)v \cdot D_{\theta}\eta(\theta,h)a\right]\right.\\
&\qquad - \mu_{\tau}\left[D_{\theta}\eta(\theta,h)a\right] \mu_{\tau}\left[D_h\eta(\theta,h)v \cdot D_h\eta(\theta,h)u\right]\\
&\qquad -\mu_{\tau}\left[D_h\eta(\theta,h)u\right]\mu_{\tau}\left[D_h\eta(\theta,h)v \cdot D_{\theta}\eta(\theta,h)a\right] \\
&\qquad -\mu_{\tau}\left[D_h\eta(\theta,h)v\right]\mu_{\tau}\left[D_{\theta}\eta(\theta,h)a \cdot D_h\eta(\theta,h)u\right]\\
&\qquad + \left.2\mu_{\tau}\left[D_h\eta(\theta,h)u\right] \mu_{\tau}\left[D_h\eta(\theta,h)v\right] \mu_{\tau}\left[D_{\theta}\eta(\theta,h)a\right]\right\},
\end{align*}
\begin{align*}
D_{\theta \theta}^2 Z_h(\theta,h)uab &=D_{\theta h}^2 Z_{\theta}(\theta,h)aub = D_{h \theta}^2 Z_{\theta}(\theta,h)abu\\
&= - \left\{\mu_{\tau_0}\left[D_{\theta \theta}^2D_h \eta(\theta,h) uab\right] - \mu_{\tau}\left[D_{\theta \theta}^2D_h \eta(\theta,h) uab\right]\right\}\\
& \quad + \left\{\mu_{\tau}\left[D^2_{\theta h}\eta(\theta,h)ua \cdot D_{\theta}\eta(\theta,h)b\right] - \mu_{\tau}\left[D_{\theta h}^2\eta(\theta,h)ua\right] \mu_{\tau}\left[D_{\theta}\eta(\theta,h)b\right]\right\}\\
&\quad + \left\{\mu_{\tau}\left[D_{\theta \theta}^2\eta(\theta,h)ab \cdot D_h\eta(\theta,h)u\right] - \mu_{\tau}\left[D_{\theta \theta}^2\eta(\theta,h)ab\right] \mu_{\tau}\left[D_h\eta(\theta,h)u\right]\right\}\\
&\quad + \left\{\mu_{\tau}\left[D_{\theta h}^2\eta(\theta,h)ub \cdot D_{\theta}\eta(\theta,h)a\right] - \mu_{\tau}\left[D_{\theta h}^2\eta(\theta,h)ub\right] \mu_{\tau}\left[D_{\theta}\eta(\theta,h)a\right]\right\}\\
&\quad + \left\{\mu_{\tau}\left[D_{\theta}\eta(\theta,h)a \cdot D_{\theta}\eta(\theta,h)b \cdot D_h\eta(\theta,h)u\right]\right.\\
&\qquad - \mu_{\tau}\left[D_{\theta}\eta(\theta,h)a\right] \mu_{\tau}\left[D_{\theta}\eta(\theta,h)b \cdot D_h\eta(\theta,h)u\right]\\
&\qquad -\mu_{\tau}\left[D_{\theta}\eta(\theta,h)b\right]\mu_{\tau}\left[D_{\theta}\eta(\theta,h)a \cdot D_h\eta(\theta,h)u\right] \\
&\qquad -\mu_{\tau}\left[D_h\eta(\theta,h)u\right]\mu_{\tau}\left[D_{\theta}\eta(\theta,h)a \cdot D_{\theta}\eta(\theta,h)b\right]\\
&\qquad + \left.2\mu_{\tau}\left[D_h\eta(\theta,h)u\right] \mu_{\tau}\left[D_{\theta}\eta(\theta,h)a\right] \mu_{\tau}\left[D_{\theta}\eta(\theta,h)b\right]\right\}.
\end{align*}
\begin{align*}
D_{\theta \theta}^2 Z_{\theta}(\theta,h)abc 
&= - \left\{\mu_{\tau_0}\left[D_{\theta \theta \theta}^3\eta(\theta,h) abc\right] - \mu_{\tau}\left[D_{\theta \theta \theta}^3 \eta(\theta,h) abc\right]\right\}\\
&\quad + \left\{\mu_{\tau}\left[D_{\theta \theta}^2\eta(\theta,h)ab \cdot D_{\theta}\eta(\theta,h)c\right] - \mu_{\tau}\left[D_{\theta}\eta(\theta,h)c\right] \mu_{\tau}\left[D_{\theta \theta}^2\eta(\theta,h)ab\right]\right\}\\
&\quad + \left\{\mu_{\tau}\left[D_{\theta \theta}^2\eta(\theta,h)bc \cdot D_{\theta}\eta(\theta,h)a\right] - \mu_{\tau}\left[D_{\theta}\eta(\theta,h)a\right] \mu_{\tau}\left[D_{\theta \theta}^2\eta(\theta,h)bc\right]\right\}\\
&\quad + \left\{\mu_{\tau}\left[D_{\theta \theta}^2\eta(\theta,h)ac \cdot D_{\theta}\eta(\theta,h)b\right] - \mu_{\tau}\left[D_{\theta}\eta(\theta,h)b\right] \mu_{\tau}\left[D_{\theta \theta}^2\eta(\theta,h)ac\right]\right\}\\
&\quad + \left\{\mu_{\tau}\left[D_{\theta}\eta(\theta,h)a \cdot D_{\theta}\eta(\theta,h)b \cdot D_{\theta}\eta(\theta,h)c\right]\right.\\
&\qquad - \mu_{\tau}\left[D_{\theta}\eta(\theta,h)a\right] \mu_{\tau}\left[D_{\theta}\eta(\theta,h)b \cdot D_{\theta}\eta(\theta,h)c\right]\\
&\qquad -\mu_{\tau}\left[D_{\theta}\eta(\theta,h)b\right]\mu_{\tau}\left[D_{\theta}\eta(\theta,h)a \cdot D_{\theta}\eta(\theta,h)c\right] \\
&\qquad -\mu_{\tau}\left[D_{\theta}\eta(\theta,h)c\right]\mu_{\tau}\left[D_{\theta}\eta(\theta,h)a \cdot D_{\theta}\eta(\theta,h)b\right]\\
&\qquad + \left.2\mu_{\tau}\left[D_{\theta}\eta(\theta,h)a\right] \mu_{\tau}\left[D_{\theta}\eta(\theta,h)b\right] \mu_{\tau}\left[D_{\theta}\eta(\theta,h)c\right]\right\},
\end{align*}

By replacing the terms $\mu_{\tau_0}[\cdot(x)]$ with $\frac{1}{n}\sum_{i=1}^n \cdot(x_i)$ in each term above, we can have the second partial Fr\'echet derivatives of $Z_{n\theta}(\theta,h)$ and $Z_{n h}(\theta,h)$  for the remainder terms  $R_{nh}(\theta,h)$ and $R_{n\theta}(\theta,h)$. 

\subsubsection*{Bounds for the remainders}
We see that the magnitude of the remainder terms $R_{\theta}$, $R_{h}$, $R_{n\theta}$,  $R_{nh}$, $e_{\theta}$, and $e_{h}$ determine how accurate $(\bar{\theta}_{\lambda} - \theta_0, \bar{h}_{\lambda} - h_0)$ and $(\bar{\theta}_{n\lambda} - \theta_{\lambda}, \bar{h}_{n\lambda} - h_{\lambda})$ are as approximations of the systematic error and the stochastic error, respectively. To obtain measures of these terms, we first define that for $\lambda > 0$, $\theta_1,\theta_2 \in \N_{\theta_0}$, $h_1,h_2 \in \N_{h_0}$, and unit elements $u_1,u_2 \in \R^p$ and $v_1, v_2 \in \CH$, 
\begin{align*}
K_h^1 &= \sup_{\substack{\theta_1,\theta_2\\h_1,h_2}} \sup_{v_1,v_2} \left\|G_{hh}^{-1}(\theta_1,h_1) \left[D_{hh}^2Z_h(\theta_2,h_2)v_1v_2 - U_{\theta h}(\theta_1,h_1) U_{\theta}^{-1}(\theta_1,h_1) D_{hh}^2Z_{\theta}(\theta_2,h_2)v_1v_2\right]\right\|_{H^{s}},\\
K_h^2 &= \sup_{\substack{\theta_1,\theta_2\\h_1,h_2}} \sup_{v_1,u_1} \left\|G_{hh}^{-1}(\theta_1,h_1) \left[D^2_{\theta h} Z_h(\theta_2,h_2)v_1u_1 - U_{\theta h}(\theta_1,h_1) U_{\theta}^{-1}(\theta_1,h_1) D^2_{\theta h} Z_{\theta}(\theta_2,h_2)v_1u_1\right]\right\|_{H^{s}},\\
K_h^3 &= \sup_{\substack{\theta_1,\theta_2\\h_1,h_2}} \sup_{u_1,v_1} \left\|G_{hh}^{-1}(\theta_1,h_1) \left[D_{h \theta}^2 Z_h(\theta_2,h_2)u_1v_1 - U_{\theta h}(\theta_1,h_1) U_{\theta}^{-1}(\theta_1,h_1) D_{h \theta}^2 Z_{\theta}(\theta_2,h_2)u_1v_1 \right] \right\|_{H^{s}},\\
K_h^4 &= \sup_{\substack{\theta_1,\theta_2\\h_1,h_2}} \sup_{u_1,u_2} \left\|G_{hh}^{-1}(\theta_1,h_1) \left[D^2_{\theta \theta} Z_h(\theta_2,h_2)u_1u_2 - U_{\theta h}(\theta_1,h_1) U_{\theta}^{-1}(\theta_1,h_1) D^2_{\theta \theta} Z_{\theta}(\theta_2,h_2)u_1u_2\right]\right\|_{H^{s}},\\
K_{\theta}^1 &= \sup_{\substack{\theta_1,\theta_2\\h_1,h_2}} \sup_{v_1,v_2} \left\|G_{\theta\theta}^{-1}(\theta_1,h_1) \left[D^2_{h h} Z_{\theta}(\theta_2,h_2)v_1v_2 - U_{h \theta}(\theta_1,h_1) G_h^{-1}(\theta_1,h_1) D^2_{h h} Z_{h}(\theta_2,h_2)v_1v_2\right]\right\|_{\R^{p}},\\
K_{\theta}^2 &= \sup_{\substack{\theta_1,\theta_2\\h_1,h_2}} \sup_{v_1,u_1} \left\|G_{\theta\theta}^{-1}(\theta_1,h_1) \left[D^2_{\theta h}Z_{\theta}(\theta_2,h_2)v_1u_1 - U_{h \theta}(\theta_1,h_1) G_h^{-1}(\theta_1,h_1) D^2_{\theta h}Z_{h}(\theta_2,h_2)v_1u_1\right]\right\|_{\R^{p}},\\
K_{\theta}^3 &= \sup_{\substack{\theta_1,\theta_2\\h_1,h_2}} \sup_{u_1,v_1} \left\|G_{\theta\theta}^{-1}(\theta_1,h_1) \left[D^2_{h \theta }Z_{\theta}(\theta_2,h_2)u_1v_1 - U_{h \theta}(\theta_1,h_1) G_h^{-1}(\theta_1,h_1) D^2_{h \theta }Z_{h}(\theta_2,h_2)u_1v_1\right]\right\|_{\R^p},\\
K_{\theta}^4 &= \sup_{\substack{\theta_1,\theta_2\\h_1,h_2}} \sup_{u_1,u_2} \left\|G_{\theta\theta}^{-1}(\theta_1,h_1) \left[D^2_{\theta \theta }Z_{\theta}(\theta_2,h_2)u_1u_2 - U_{h \theta}(\theta_1,h_1) G_h^{-1}(\theta_1,h_1) D^2_{\theta \theta }Z_{h}(\theta_2,h_2)u_1u_2\right]\right\|_{\R^p}.
\end{align*}
For ${i=1,2,3,4}$, we also define $K^{i}_{nh}, K^i_{n\theta}$ by replacing $Z_{\theta},Z_{h}$ with $Z_{n\theta}, Z_{nh}$ in $K^i_{h}$ and $K^i_{\theta}$, respectively. In addition, for $\lambda > 0$, $\theta_1,\theta_2 \in \N_{\theta_0}$, $h_1,h_2 \in \N_{h_0}$, and unit elements $u_1,u_2 \in \R^p$ and $v_1, v_2 \in \CH$, we define
\begin{align*}
E_{nh}^{12} &= \sup_{\theta_1,h_1} \sup_{u_1} \left\|G_{hh}^{-1}(\theta_1,h_1) e_{h\theta}(\theta_1,h_1)u_1\right\|_{\CH},\\
E_{nh}^{11} &= \sup_{\theta_1,h_1} \sup_{v_1} \left\|G_{hh}^{-1}(\theta_1,h_1) e_{hh}(\theta_1,h_1)v_1\right\|_{\CH},\\
E_{nh}^{22} &= \sup_{\theta_1,h_1} \sup_{u_1} \left\|G_{hh}^{-1}(\theta_1,h_1) U_{\theta h}(\theta_1,h_1) U_{\theta}^{-1}(\theta_1,h_1) e_{\theta\theta}(\theta_1,h_1)u_1\right\|_{\CH},\\
E_{nh}^{21} &= \sup_{\theta_1,h_1} \sup_{v_1} \left\|G_{hh}^{-1}(\theta_1,h_1) U_{\theta h}(\theta_1,h_1) U_{\theta}^{-1}(\theta_1,h_1) e_{\theta h}(\theta_1,h_1)v_1\right\|_{\CH},\\
E_{n\theta}^{22} &= \sup_{\theta_1,h_1} \sup_{u_1} \left\|G_{\theta\theta}^{-1}(\theta_1,h_1) e_{\theta\theta}(\theta_1,h_1)u_1\right\|_{\R^p},\\
E_{n\theta}^{21} &= \sup_{\theta_1,h_1} \sup_{v_1} \left\|G_{\theta\theta}^{-1}(\theta_1,h_1) e_{\theta h}(\theta_1,h_1)v_1\right\|_{\R^p},\\
E_{n\theta}^{12} &= \sup_{\theta_1,h_1} \sup_{u_1} \left\|G_{\theta\theta}^{-1}(\theta_1,h_1) U_{h\theta}(\theta_1,h_1) G_h^{-1}(\theta_1,h_1) e_{h\theta}(\theta_1,h_1)u_1\right\|_{\R^p},\\
E_{n\theta}^{11} &= \sup_{\theta_1,h_1} \sup_{v_1} \left\|G_{\theta\theta}^{-1}(\theta_1,h_1) U_{h\theta}(\theta_1,h_1) G_h^{-1}(\theta_1,h_1) e_{hh}(\theta_1,h_1)v_1\right\|_{\R^p}.
\end{align*}
Therefore, for any $a \in \R^p$ and $g \in \CH$, standard analysis allows us to have the following bounds for the remainder terms for the systematic error and the stochastic error,
\begin{equation}\label{bound_remainder_Zh}
\begin{aligned}
&\norm{G_{hh}^{-1}(\theta_0,h_0) \left[R_h(\theta_0,h_0)ag - U_{\theta h}(\theta_0,h_0) U_{\theta}^{-1}(\theta_0,h_0) R_{\theta}(\theta_0,h_0)ag\right]}_{\CH}\\
&\qquad \le \frac{1}{2}\left[\left(K_h^1\norm{g}_{\CH} + K_h^2\norm{a}_{\R^p}\right) \norm{g}_{\CH} + \left(K_h^3\norm{g}_{\CH} + K_h^4\norm{a}_{\R^p}\right) \norm{a}_{\R^p}\right],
\end{aligned}
\end{equation}
\begin{equation}\label{bound_remainder_Ztheta}
\begin{aligned}
&\norm{G_{\theta\theta}^{-1}(\theta_0,h_0) \left[R_{\theta}(\theta_0,h_0)ag - U_{h\theta}(\theta_0,h_0) G_h^{-1}(\theta_0,h_0) R_h(\theta_0,h_0)ag\right]}_{\R^p}\\
&\qquad \le \frac{1}{2}\left[\left(K_{\theta}^1\norm{g}_{\CH} + K_{\theta}^2\norm{a}_{\R^p}\right) \norm{g}_{\CH} + \left(K_{\theta}^3\norm{g}_{\CH} + K_{\theta}^4\norm{a}_{\R^p}\right) \norm{a}_{\R^p}\right],
\end{aligned}
\end{equation}
\begin{equation}\label{bound_remainder_Znh}
\begin{aligned}
&\norm{G_{hh}^{-1}(\theta_{\lambda},h_{\lambda}) \left\{\left[e_h(\theta_{\lambda},h_{\lambda}) + R_{nh}(\theta_{\lambda},h_{\lambda})\right]ag - U_{\theta h}(\theta_{\lambda},h_{\lambda}) U_{\theta}^{-1}(\theta_{\lambda},h_{\lambda}) \left[e_{\theta}(\theta_{\lambda},h_{\lambda}) + R_{n\theta}(\theta_{\lambda},h_{\lambda})\right]ag\right\}}_{\CH}\\
&\qquad \le E_{nh}^1\norm{g}_{\CH} + E_{nh}^2\norm{a}_{\R^p} + \frac{1}{2}\left[(K_{nh}^1\norm{g}_{\CH} + K_{nh}^2\norm{a}_{\R^p})\norm{g}_{\CH} + (K_{nh}^3\norm{g}_{\CH} + K_{nh}^4\norm{a}_{\R^p})\norm{a}_{\R^p}\right],
\end{aligned}
\end{equation}
where $E_{nh}^1 = E_{nh}^{11} + E_{nh}^{21}$ and $E_{nh}^2 = E_{nh}^{12} + E_{nh}^{22}$, and
\begin{equation}\label{bound_remainder_Zntheta}
\begin{aligned}
&\norm{G_{\theta\theta}^{-1}(\theta_{\lambda},h_{\lambda}) \left\{\left[e_{\theta}(\theta_{\lambda},h_{\lambda}) + R_{n\theta}(\theta_{\lambda},h_{\lambda})\right]ag - U_{h\theta}(\theta_{\lambda},h_{\lambda}) G_h^{-1}(\theta_{\lambda},h_{\lambda}) \left[e_h(\theta_{\lambda},h_{\lambda}) + R_{nh}(\theta_{\lambda},h_{\lambda})\right]ag\right\}}_{\R^p}\\
&\qquad \le E_{n\theta}^1\norm{g}_{\CH} + E_{n\theta}^2\norm{a}_{\R^p} + \frac{1}{2}\left[(K_{n\theta}^1\norm{g}_{\CH} + K_{n\theta}^2\norm{a}_{\R^p})\norm{g}_{\CH} + (K_{n\theta}^3\norm{g}_{\CH} + K_{n\theta}^4\norm{a}_{\R^p})\norm{a}_{\R^p}\right],
\end{aligned}
\end{equation}
where $E_{n\theta}^1 = E_{n\theta}^{21} + E_{n\theta}^{11}$ and $E_{n\theta}^2 = E_{n\theta}^{22} + E_{n\theta}^{12}$.

\subsubsection*{Proof of Existence and Uniqueness}
We are now ready to show the existence and uniqueness of $(\theta_{\lambda}, h_{\lambda})$ and $(\hat{\theta},\hat{h})$ in the neighborhood $\N_{\theta_0} \times \N_{h_0}$. Let 
\begin{align*}
d_{\theta}(\lambda) &= \norm{\bar{\theta}_{\lambda} - \theta_0}_{\R^p}, \\
d_{h}(\lambda) &= \norm{\bar{h}_{\lambda}- h_0}_{\CH},\\
r_{\theta}(\lambda) &= (K_h^3 + K_{\theta}^3)d_h(\lambda) + (K_h^4 + K_{\theta}^4) d_{\theta}(\lambda),\\
r_{h}(\lambda) &= (K_h^1 + K_{\theta}^1)d_h(\lambda) + (K_h^2 + K_{\theta}^2) d_{\theta}(\lambda),\\
S_{\theta,\theta_1}(\gamma) &= \{a \in \R^{p}: \norm{a-\theta_1}_{\R^p} \le \gamma\} \text{ for }\theta_1 \in \R^p, \\
S_{h,h_1}(\gamma) &= \{u \in \CH: \norm{g-h_1}_{\CH} \le \gamma\} \text{ for }h_1 \in \CH,\\
S_{\theta}(\gamma) &= S_{\theta,0}(\gamma),\\
S_{h}(\gamma) &= S_{h,0}(\gamma).
\end{align*}
One can get the following theorem for the existence and uniqueness of $(\theta_{\lambda},h_{\lambda})$ via a contraction mapping argument. 
\begin{atheorem}(Existence of $(\theta_{\lambda},h_{\lambda})$ and the bias approximation) \label{existence of systematic estimator}
If $d_{\theta}(\lambda)\rightarrow 0, d_{h}(\lambda) \rightarrow 0, r_{\theta}(\lambda) \rightarrow 0, r_{h}(\lambda) \rightarrow 0$ as $\lambda \rightarrow 0$, there exist $\lambda_0 >0$ such that for $\lambda \in [0,\lambda_0]$, there are unique $\theta_{\lambda} \in S_{\theta,\theta_0}(2d_{\theta}(\lambda))$ and $h_{\lambda} \in S_{h,h_0}(2d_h(\lambda))$ satisfying $Z_{\theta}(\theta_{\lambda},h_{\lambda})=0, Z_h(\theta_{\lambda},h_{\lambda})=0$, and $(\theta_{\lambda},h_{\lambda}) \in \N_{\theta_0} \times \N_{h_0}$. In addition, as $\lambda \to 0$,
\[\left\|\bar{\theta}_{\lambda} - \theta_{\lambda}\right\|_{\R^p} + \left\|\bar{h}_{\lambda} - h_{\lambda}\right\|_{\CH} \le 4 \left[r_h(\lambda) d_h(\lambda) + r_{\theta}(\lambda) d_{\theta}(\lambda)\right].\] 
\end{atheorem}

\begin{proof}
Let $t_{\theta\lambda} = 2d_{\theta}(\lambda)$, $t_{h\lambda} = 2d_h(\lambda)$.
Define
\begin{align*}
F_{\theta}(\zeta,g) &= \zeta - G_{\theta\theta}^{-1}(\theta_0,h_0) \left[Z_{\theta}(\theta_0+\zeta,h_0+g)\right.\\
&\qquad\qquad \left. - U_{h\theta}(\theta_0,h_0) G_h^{-1}(\theta_0,h_0) Z_h(\theta_0+\zeta,h_0+g)\right],\\
F_h(\zeta,g) &= g - G_{hh}^{-1}(\theta_0,h_0) \left[Z_h(\theta_0+\zeta,h_0+g)\right.\\
&\qquad\qquad \left. - U_{\theta h}(\theta_0,h_0) U_{\theta}^{-1}(\theta_0,h_0) Z_{\theta}(\theta_0+\zeta,h_0+g)\right].
\end{align*}

Let $\vec{F}(\zeta,g) = (F_{\theta}(\zeta,g), F_h(\zeta,g))$ be a function on $\Q = \R^p \times \CH$, and for any subset $\Q_1 \subset \Q$, denote by $\vec{F}(\Q_1)$ the image of $\Q_1$ under $\vec{F}$. The proof has three steps:
\begin{enumerate}
\item
$\vec{F}(S_{\theta}(t_{\theta\lambda}) \times S_h(t_{h\lambda})) \subset S_{\theta}(t_{\theta\lambda}) \times S_h(t_{h\lambda})$.
\item
$\vec{F}$ is a contraction map on $S_{\theta}(t_{\theta\lambda}) \times S_h(t_{h\lambda})$.
\item
Obtaining the estimate for the bias approximation, $\left\|\bar{\theta}_{\lambda} - \theta_{\lambda}\right\|_{\R^p} + \left\|\bar{h}_{\lambda} - h_{\lambda}\right\|_{\CH}$.
\end{enumerate}

For~(a), by our assumption, we can choose $\lambda_0$ small enough that $S_{\theta,\theta_0}(t_{\theta\lambda}) \subset \N_{\theta_0}$, $S_{h,h_0}(t_{h\lambda}) \subset \N_{h_0}$, and $r_{\theta}(\lambda) < 1/2$ for all $\lambda \in (0,\lambda_0]$. For every $(\theta,h) \in \Q$, we denote 
\[\norm{(\theta,h)}_{{\R^p}\times \CH} = \norm{\theta}_{\R^p} + \norm{h}_{\CH}.\] 
For $(\zeta,g) \in S_{\theta}(t_{\theta\lambda}) \times S_h(t_{h\lambda})$, we have
\[\norm{\vec{F}(\zeta,g)}_{\R^p \times \CH} = \norm{F_{\theta}(\zeta,g)}_{\R^p} + \norm{F_h(\zeta,g)}_{\CH}.\]
For $\norm{F_{\theta}(\zeta,g)}_{\R^p}$, by triangle inequality, we have
\begin{align*}
\norm{F_{\theta}(\zeta,g)}_{\R^p} &\le \left\|\zeta - G_{\theta\theta}^{-1}(\theta_0,h_0) \left[Z_{\theta}(\theta_0+\zeta,h_0+g)\right.\right.\\
&\qquad \left.\left. - U_{h\theta}(\theta_0,h_0) G_h^{-1}(\theta_0,h_0) Z_h(\theta_0+\zeta,h_0+g)\right] - (\bar{\theta}_{\lambda} - \theta_0)\right\|_{\R^p}\\
&+ \norm{\bar{\theta}_{\lambda} - \theta_0}_{\R^p}.
\end{align*}
By the definition of $\bar{\theta}_{\lambda} - \theta_0$ and $G_{\theta\theta}(\theta,h)$, Taylor series expansion of $Z_{\theta}(\theta_0+\zeta, h_0+g)$ and $Z_h(\theta_0+\zeta, h_0+g)$, and the remainder bound \eqref{bound_remainder_Ztheta}, we can get
\begin{align*}
&\left\|\zeta - G_{\theta\theta}^{-1}(\theta_0,h_0)\left[Z_{\theta}(\theta_0+\zeta,h_0+g) - U_{h\theta}(\theta_0,h_0)G_h^{-1}(\theta_0,h_0)Z_h(\theta_0+\zeta,h_0+g)\right] - (\bar{\theta}_{\lambda} - \theta_0)\right\|_{\R^p}\\
&\qquad = \left\|\zeta - G_{\theta\theta}^{-1}(\theta_0,h_0)\left[Z_{\theta}(\theta_0+\zeta,h_0+g) - U_{h\theta}(\theta_0,h_0)G_h^{-1}(\theta_0,h_0)Z_h(\theta_0+\zeta,h_0+g)\right]\right.\\
&\qquad\qquad \left. + G_{\theta\theta}^{-1}(\theta_0,h_0)\left[Z_{\theta}(\theta_0,h_0) - U_{h\theta}(\theta_0,h_0)G_h^{-1}(\theta_0,h_0)Z_h(\theta_0,h_0)\right]\right\|_{\R^p}\\
&\qquad = \left\|\zeta - G_{\theta\theta}^{-1}(\theta_0,h_0)\left\{\left[Z_{\theta}(\theta_0+\zeta,h_0+g) - Z_{\theta}(\theta_0,h_0)\right]\right.\right.\\
&\qquad\qquad \left.\left.- U_{h\theta}(\theta_0,h_0)G_h^{-1}(\theta_0,h_0) \left[Z_h(\theta_0+\zeta,h_0+g) - Z_h(\theta_0,h_0)\right]\right\}\right\|_{\R^p}\\
&\qquad = \left\|\zeta - G_{\theta\theta}^{-1}(\theta_0,h_0)\left\{\left[U_{\theta}(\theta_0,h_0)\zeta + R_{\theta}(\theta_0,h_0)\zeta g\right] \right.\right.\\
&\qquad \qquad \left. \left. - U_{h\theta}(\theta_0,h_0)G_h^{-1}(\theta_0,h_0)\left[ U_{\theta h}(\theta_0,h_0)\zeta + R_h(\theta_0,h_0)\zeta g\right]\right\}\right\|_{\R^p}\\
&\qquad = \left\|\zeta - G_{\theta\theta}^{-1}(\theta_0,h_0)\left\{\left[U_{\theta}(\theta_0,h_0) - U_{h\theta}(\theta_0,h_0)G_h^{-1}(\theta_0,h_0)U_{\theta h}(\theta_0,h_0)\right]\zeta \right. \right.\\
&\qquad \qquad \left. \left. + \left[R_{\theta}(\theta_0,h_0) - U_{h\theta}(\theta_0,h_0)G_h^{-1}(\theta_0,h_0)R_h(\theta_0,h_0)\right]\zeta g\right\}\right\|_{\R^p}\\
&\qquad = \left\|G_{\theta\theta}^{-1}(\theta_0,h_0) \left[R_{\theta}(\theta_0,h_0)\zeta g - U_{h\theta}(\theta_0,h_0) G_h^{-1}(\theta_0,h_0) R_h(\theta_0,h_0) \zeta g\right]\right\|_{\R^p}\\
&\qquad \le \frac{1}{2} \left(K^1_{\theta}\norm{g}_{\CH} + K_{\theta}^2\norm{\zeta}_{\R^p}\right) \norm{g}_{\CH} + \frac{1}{2} \left(K_{\theta}^3\norm{g}_{\CH} + K_{\theta}^4\norm{\zeta}_{\R^p}\right)\norm{\zeta}_{\R^p}.
\end{align*}

Similarly, by the definition of $\bar{h}_{\lambda} - h_0$ and $G_{hh}(\theta,h)$, Taylor series expansion of $Z_{\theta}(\theta_0+\zeta, h_0+g)$ and $Z_h(\theta_0+\zeta, h_0+g)$, and the remainder bound \eqref{bound_remainder_Zh}, we also have
\begin{align*}
\norm{F_h(\zeta,g)}_{\CH} &\le \left\|g - G_{hh}^{-1}(\theta_0,h_0) \left[Z_h(\theta_0+\zeta,h_0+g)\right.\right.\\
&\qquad \left.\left. - U_{\theta h}(\theta_0,h_0)U_{\theta}^{-1}(\theta_0,h_0)Z_{\theta}(\theta_0+\zeta,h_0+g)\right] - (\bar{h}_{\lambda} - h_0)\right\|_{\CH}\\
&\qquad + \norm{\bar{h}_{\lambda} - h_0}_{\CH},
\end{align*}
and
\begin{align*}
&\left\|g - G_{hh}^{-1}(\theta_0,h_0) \left[Z_h(\theta_0+\zeta,h_0+g) - U_{\theta h}(\theta_0,h_0)U_{\theta}^{-1}(\theta_0,h_0) Z_{\theta}(\theta_0+\zeta,h_0+g)\right] - (\bar{h}_{\lambda} - h_0)\right\|_{\CH}\\
&\quad = \left\|G_{hh}^{-1}(\theta_0,h_0) \left[R_h(\theta_0,h_0)\zeta g - U_{\theta h}(\theta_0,h_0) U_{\theta}^{-1}(\theta_0,h_0) R_{\theta}(\theta_0,h_0) \zeta g\right]\right\|_{\CH}\\
&\quad \le \frac{1}{2}\left(K_h^1\norm{g}_{\CH} + K_h^2\norm{\zeta}_{\R^p}\right)\norm{g}_{\CH} + \frac{1}{2}\left(K_h^3\norm{g}_{\CH} + K_h^4\norm{\zeta}_{\R^p}\right)\norm{\zeta}_{\R^p}.
\end{align*}
Since $t_{\theta\lambda} =  2\norm{\bar{\theta}_{\lambda} - \theta_0}_{\R^p}$, $t_{h\lambda} = 2\norm{\bar{h}_{\lambda} - h_0}_{\CH}$, $r_h(\lambda)< 1/2$, and $r_{\theta}(\lambda)<1/2$, for $(\zeta,g) \in {S_{\theta}(t_{\theta\lambda}) \times S_h(t_{h\lambda})}$, we have
\begin{align*}
\norm{\vec{F}(\zeta,g)}_{\R^p \times \CH} &\le \frac{1}{2}\left[\left(K_h^1+K_{\theta}^1\right)\norm{g}_{\CH} + \left(K_h^2+K_{\theta}^2\right)\norm{\zeta}_{\R^p}\right] \norm{g}_{\CH}\\
&\qquad + \frac{1}{2}\left[\left(K_h^3+K_{\theta}^3\right)\norm{g}_{\CH} + \left(K_h^4+K_{\theta}^4\right)\norm{\zeta}_{\R^p}\right] \norm{\zeta}_{\R^p}\\
&\qquad + \norm{\bar{\theta}_{\lambda} - \theta_0}_{\R^p} + \norm{\bar{h}_{\lambda} - h_0}_{\CH}\\
& \le \frac{1}{2}\left[\left(K_h^1+K_{\theta}^1\right)t_{h\lambda} + \left(K_h^2+K_{\theta}^2\right)t_{\theta\lambda}\right]t_{h\lambda}\\
&\qquad + \frac{1}{2}\left[\left(K_h^3+K_{\theta}^3\right)t_{h\lambda} + \left(K_h^4+K_{\theta}^4\right)t_{\theta\lambda}\right]t_{\theta\lambda}\\
&\qquad + \frac{1}{2} t_{\theta\lambda} + \frac{1}{2} t_{h\lambda}\\
&= r_h(\lambda)t_{h\lambda} + r_{\theta}(\lambda)t_{\theta\lambda} + \frac{1}{2}t_{\theta\lambda} + \frac{1}{2}t_{h\lambda}\\
&= \left(r_h(\lambda) + \frac{1}{2}\right)t_{h\lambda} + \left(r_{\theta}(\lambda) + \frac{1}{2}\right)t_{\theta\lambda} \\
&< t_{h\alpha} + t_{\theta\alpha}.
\end{align*}

Now for step (b):  For $\zeta_1, \zeta_2 \in S_{\theta}(t_{\theta\lambda})$, $g_1,g_2 \in S_h(t_{h\lambda})$, by Taylor expansion, we get
\begin{align*}
Z_{\theta}(\theta_0+\zeta_2 &, h_0+g_2) = Z_{\theta}(\theta_0+\zeta_1, h_0+g_1)\\
&+ \int_0^1 D_{\theta} Z_{\theta}\left[\theta_0+\zeta_1+t(\zeta_2 - \zeta_1), h_0+g_1+t(g_2-g_1)\right](\zeta_2-\zeta_1)\\
&+ D_hZ_{\theta}\left[\theta_0+\zeta_1+t(\zeta_2 - \zeta_1), h_0+g_1+t(g_2-g_1)\right](g_2-g_1)\ dt.
\end{align*}
Applying Taylor expansion again to the terms inside the integral and letting $\zeta^{*} = \zeta_1 + t(\zeta_2-\zeta_1)$, ${g^{*} = g_1+t(g_2-g_1)}$, we have
\begin{align*}
Z_{\theta}(\theta_0&+\zeta_2,h_0+g_2) - Z_{\theta}(\theta_0+\zeta_1,h_0+g_1)\\
&= U_{\theta}(\theta_0,h_0)(\zeta_2-\zeta_1) + U_{h\theta}(\theta_0,h_0)(g_2-g_1)\\
&\qquad + \int_0^1\int_0^1 \left[D_{\theta \theta}^2Z_{\theta}(\theta_0+t'\zeta^{*},h_0+t'g^{*})\zeta^{*}\right.\\
&\qquad\qquad\qquad\qquad \left.+ D^2_{h \theta}Z_{\theta}(\theta_0+t'\zeta^{*},h_0+t'g^{*})g^{*}\right](\zeta_2-\zeta_1)\ dt'dt\\
&\qquad + \int_0^1\int_0^1 \left[D_{\theta h}^2 Z_{\theta}(\theta_0+t'\zeta^{*},h_0+t'g^{*})\zeta^{*}\right.\\
&\qquad\qquad\qquad\qquad \left.+ D_{hh}^2 Z_{\theta}(\theta_0+t'\zeta^{*},h_0+t'g^{*})g^{*}\right](g_2-g_1)\ dt'dt.
\end{align*}

Since
\begin{align*}
&F_{\theta}(\zeta_1,g_1) - F_{\theta}(\zeta_2,g_2)\\
&= (\zeta_1-\zeta_2) - G_{\theta\theta}^{-1}(\theta_0,h_0) \left\{\left[Z_{\theta}(\theta_0+\zeta_1,h_0+g_1) - Z_{\theta}(\theta_0+\zeta_2,h_0+g_2)\right]\right.\\
&\qquad\qquad \left.- U_{h\theta}(\theta_0,h_0)G_h^{-1}(\theta_0,h_0) \left[Z_h(\theta_0+\zeta_1,h_0+g_1) - Z_h(\theta_0+\zeta_2,h_0+g_2)\right]\right\},
\end{align*}
and for $0 \le t \le 1$, $\zeta^{*} = \zeta_1 + t(\zeta_2-\zeta_1) \in S_{\theta}(t_{\theta\lambda})$, $g^{*} = g_1+t(g_2-g_1) \in S_h(t_{h\lambda})$ by convexity of $S_{\theta}(t_{\theta\lambda})$ and  $S_h(t_{h\lambda})$, similar algebraic manipulations as in the proof of (a) show that 
\begin{align*}
\norm{F_{\theta}(\zeta_1,g_1) - F_{\theta}(\zeta_2,g_2)}_{\R^p} &\le \left(K_{\theta}^3\norm{g\str}_{\CH} + K_{\theta}^4\norm{\zeta\str}_{\R^p}\right)\norm{\zeta_2-\zeta_1}_{\R^p}\\
&\quad + \left(K_{\theta}^2\norm{\zeta\str}_{\R^p} + K_{\theta}^1\norm{g\str}_{\CH}\right)\norm{g_2-g_1}_{\CH}.
\end{align*}
Similarly for $F_h$, we get 
\begin{align*}
\norm{F_h(\zeta_1,g_1) - F_h(\zeta_2,g_2)}_{\CH} &\le \left(K_h^3\norm{g\str}_{\CH} + K_h^4\norm{\zeta\str}_{\R^p}\right)\norm{\zeta_2-\zeta_1}_{\R^p}\\
&\quad + \left(K_h^2\norm{\zeta\str}_{\R^p} + K_h^1\norm{g\str}_{\CH}\right)\norm{g_2-g_1}_{\CH}.
\end{align*}
Therefore,
\begin{align*}
&\norm{\vec{F}(\zeta_1,g_1) - \vec{F}(\zeta_2,g_2)}_{\R^s \times \CH}\\
&\qquad \le (K_h^1+K_{\theta}^1) \norm{g\str}_{\CH} \norm{g_2-g_1}_{\CH} + (K_h^2+K_{\theta}^2)\norm{\zeta\str}_{\R^p} \norm{g_2-g_1}_{\CH}\\
&\qquad \quad + (K_h^3+K_{\theta}^3) \norm{g\str}_{\CH} \norm{\zeta_2-\zeta_1}_{\R^p} + (K_h^4+K_{\theta}^4)\norm{\zeta\str}_{\R^p} \norm{\zeta_2-\zeta_1}_{\R^p}\\
&\qquad = 2\left[(K_h^1+K_{\theta}^1) d_h(\lambda) + (K_h^2 + K_{\theta}^2) d_{\theta}(\lambda)\right] \norm{g_2-g_1}_{\CH}\\
&\qquad\quad + 2\left[(K_h^3+K_{\theta}^3) d_h(\lambda) + (K_h^4 + K_{\theta}^4) d_{\theta}(\lambda)\right] \norm{\zeta_2-\zeta_1}_{\R^p}\\
&\qquad = 2r_h(\lambda)\norm{g_2-g_1}_{\CH} + 2r_{\theta}(\lambda)\norm{\zeta_2-\zeta_1}_{\R^p}\\
&\quad \le C_1\norm{g_2-g_1}_{\CH} + C_2\norm{\zeta_2-\zeta_1}_{\R^p},
\end{align*}
where $0<C_1 < 1$, $0<C_2 < 1$, so $\vec{F} = (F_{\theta},F_h)$ is a contraction map on $S_{\theta}(t_{\theta\lambda}) \times S_h(t_{h\lambda})$. By the contraction mapping theorem (see Theorem 9.23 in \citeasnoun{rudin1976principles}), there exists a unique ${(\zeta_{\lambda},g_{\lambda}) \in S_{\theta}(t_{\theta\lambda}) \times S_h(t_{h\lambda})}$ such that
\[\vec{F}(\zeta_{\lambda},g_{\lambda}) = (\zeta_{\lambda},g_{\lambda}).\]

Let $\theta_{\lambda} = \theta_0+\zeta_{\lambda}$, $h_{\lambda} = h_0+g_{\lambda}$. Then $\theta_{\lambda} \in S_{\theta,\theta_0}(t_{\theta\lambda})$, $h_{\lambda} \in S_{h,h_0}(t_{\theta\lambda})$, and $(\theta_{\lambda},h_{\lambda})$ are the unique solutions to $Z_{\theta}(\theta_{\lambda},h_{\lambda}) = 0$, $Z_h(\theta_{\lambda},h_{\lambda}) = 0$. 

For part (c), note that
\begin{align*}
(\bar{\theta}_{\lambda} - \theta_{\lambda}, \bar{h}_{\lambda} - h_{\lambda}) &= (\bar{\theta}_{\lambda}-\theta_0,\bar{h}_{\lambda}-h_0) - (\theta_{\lambda}-\theta_0,h_{\lambda}-h_0)\\
&= \vec{F}(0,0) - \vec{F}(\zeta_{\lambda},g_{\lambda}).
\end{align*}
Thus,
\begin{align*}
&\norm{\bar{\theta}_{\lambda}-\theta_{\lambda}}_{\R^p} + \norm{\bar{h}_{\lambda}-h_{\lambda}}_{\CH}\\
&\qquad\qquad = \norm{\vec{F}(\zeta_{\lambda},g_{\lambda}) - \vec{F}(0,0)}_{\R^p \times \CH}\\
&\qquad\qquad \le 2r_h(\lambda)\norm{g_{\lambda}}_{\CH} + 2r_{\theta}(\lambda)\norm{\zeta_{\lambda}}_{\R^p}\\
&\qquad\qquad \le 4\left[r_h(\lambda) d_h(\lambda) + r_{\theta}(\lambda) d_{\theta}(\lambda)\right].
\end{align*}

This completes the proof of Theorem A.\ref{existence of systematic estimator}.
\end{proof}
Next, we consider the existence of $(\hat{\theta},\hat{h}) \in \N_{\theta_0} \times \N_{h_0}$.
Define
\begin{align*}
d_{n\theta}(\lambda) &= \norm{\bar{\theta}_{n\lambda} - \theta_{\lambda}}_{\R^p},\\
d_{nh}(\lambda) &= \norm{\bar{h}_{n\lambda} - h_{\lambda}}_{\CH},\\
r_{n\theta}(\lambda) &= E_{n\theta}^2 + E_{nh}^2 + (K_{n\theta}^3+K_{nh}^3)d_{nh}(\lambda) + (K_{n\theta}^4+K_{nh}^4) d_{n\theta}(\lambda),\\
r_{nh}(\lambda) &= E_{n\theta}^1 + E_{nh}^1 + (K_{n\theta}^1+K_{nh}^1)d_{nh}(\lambda) + (K_{n\theta}^2+K_{nh}^2) d_{n\theta}(\lambda).
\end{align*}
We get the following existence theorem for $(\hat{\theta},\hat{h}) \in \N_{\theta_0} \times \N_{h_0}$. 
\begin{atheorem}(Existence of $(\hat{\theta},\hat{h})$ and the variability approximation)
Suppose $\lambda_n$ is a sequence such that for all $n$ sufficiently large, $\theta_{\lambda_n} \in \N_{\theta_0}$, $h_{\lambda_n} \in \N_{h_0}$, and
\begin{align*}
d_{n\theta}(\lambda_n) &\overset{P}{\rightarrow} 0, &d_{nh}(\lambda_n) &\overset{P}{\rightarrow} 0,\\
r_{n\theta}(\lambda_n) &\overset{P}{\rightarrow} 0, &r_{nh}(\lambda_n) &\overset{P}{\rightarrow} 0.
\end{align*}
Then, with probability tending to unity as $n \to \infty$, there is a unique root $(\hat{\theta}, \hat{h})$ of $Z_{n\theta}(\hat{\theta},\hat{h}) = 0$, $Z_{nh}(\hat{\theta},\hat{h}) = 0$ in ${S_{\theta,\theta_{\lambda_n}}(2d_{n\theta}(\lambda_n)) \times S_{h,h_{\lambda_n}}(2d_{nh}(\lambda_n))} \subset \N_{\theta_0} \times \N_{h_0}$. In addition, as $n \to \infty$, $\lambda_n \to 0$, 
\[\norm{\hat{\theta} - \bar{\theta}_{n\lambda_n}}_{\R^p} + \norm{\hat{h} - \bar{h}_{n \lambda_n}}_{\CH} \le 4r_{n\theta}(\lambda_n)d_{n\theta}(\lambda_n) + 4r_{nh}(\lambda_n)d_{nh}(\lambda_n).\]
\end{atheorem}

\begin{proof}
For convenience, we drop the subscript on $\lambda_n$ and let $t_{n\theta\lambda} = 2d_{n\theta}(\lambda)$, $t_{nh\lambda} = 2d_{nh}(\lambda)$. Let
\begin{align*}
F_{n\theta}(\zeta,g) &= \zeta - G_{\theta\theta}^{-1}(\theta_{\lambda},h_{\lambda})\left[Z_{n\theta}(\theta_{\lambda}+\zeta,h_{\lambda}+g)\right.\\
&\qquad\qquad \left.- U_{h\theta}(\theta_{\lambda},h_{\lambda}) G_h^{-1}(\theta_{\lambda},h_{\lambda}) Z_{nh}(\theta_{\lambda}+\zeta,h_{\lambda}+g)\right]\\
F_{nh}(\zeta,g) &= g - G_{hh}^{-1}(\theta_{\lambda},h_{\lambda})\left[Z_{nh}(\theta_{\lambda}+\zeta,h_{\lambda}+g)\right.\\
&\qquad\qquad \left.- U_{\theta h}(\theta_{\lambda},h_{\lambda}) U_{\theta}^{-1}(\theta_{\lambda},h_{\lambda}) Z_{n\theta}(\theta_{\lambda}+\zeta,h_{\lambda}+g)\right].
\end{align*}
The proof proceeds in three steps, similar to the proof of Theorem \ref{existence of systematic estimator}, with additional terms introduced in approximating $D_{\theta}Z_{n\theta}$ and $D_hZ_{nh}$ by $D_{\theta}Z_{\theta}$ and $D_hZ_{h}$, respectively. Take $n$ large enough so that $S_{\theta,\theta_{\lambda}}(t_{n\theta\lambda}) \subset \N_{\theta_0}$, $S_{h,h_{\lambda}}(t_{nh\lambda}) \subset \N_{h_0}$ and $r_{n\theta}(\lambda) < \frac{1}{2}$, $r_{nh}(\lambda) < \frac{1}{2}$. 

First, we show that $\vec{F}_n(\zeta,g) = (F_{n\theta}(\zeta,g), F_{nh}(\zeta,g))$ maps $S_{\theta}(t_{n\theta\lambda}) \times S_h(t_{nh\lambda})$ to itself, i.e.,
\[\vec{F}_n\left(S_{\theta}(t_{n\theta\lambda}) \times S_h(t_{nh\lambda})\right) \subset S_{\theta}(t_{n\theta\lambda}) \times S_h(t_{nh\lambda}).\]
By definition, for $(\zeta,g) \in S_{\theta}(t_{n\theta\lambda}) \times S_h(t_{nh\lambda})$, we have
\[\norm{\vec{F}_n(\zeta,g)}_{\R^p \times \CH} = \norm{F_{n\theta}(\zeta,g)}_{\R^p} + \norm{F_{nh}(\zeta,g)}_{\CH}.\]
For $F_{n\theta}$, by the triangle inequality, we get
\begin{align*}
\norm{F_{n\theta}(\zeta,g)}_{\R^p} &\le \left\|\zeta - G_{\theta\theta}^{-1}(\theta_{\lambda},h_{\lambda})\left[Z_{n\theta}(\theta_{\lambda}+\zeta,h_{\lambda}+g)\right.\right.\\
&\qquad \left.- U_{h\theta}(\theta_{\lambda},h_{\lambda})G_h^{-1}(\theta_{\lambda},h_{\lambda}) Z_{nh}(\theta_{\lambda}+\zeta,h_{\lambda}+g)\right]\\
&\qquad \left.- (\bar{\theta}_{n\lambda} - \theta_{\lambda})\right\|_{\R^p} + \norm{\bar{\theta}_{n\lambda} - \theta_{\lambda}}_{\R^p}
\end{align*}

Using the definition of $\bar{\theta}_{n\lambda} - \theta_{\lambda}$, $G_{\theta \theta}(\theta,h)$, Taylor expansion of $Z_{n\theta}(\theta_{\lambda} + \zeta, h_{\lambda} + g)$ and $Z_{nh}(\theta_{\lambda} + \zeta, h_{\lambda} + g)$, and the remainder bound \eqref{bound_remainder_Zntheta}, we get that
\begin{align*}
&\left\|\zeta - G_{\theta\theta}^{-1}((\theta_{\lambda},h_{\lambda})\left[Z_{n\theta}(\theta_{\lambda}+\zeta,h_{\lambda}+g) - U_{h\theta}(\theta_{\lambda},h_{\lambda})G_h^{-1}(\theta_{\lambda},h_{\lambda})Z_{nh}(\theta_{\lambda}+\zeta,h_{\lambda}+g)\right] - (\bar{\theta}_{n\lambda}-\theta_{\lambda})\right\|_{\R^p}\\
&\qquad = \left\|\zeta - G_{\theta\theta}^{-1}(\theta_{\lambda},h_{\lambda})\left\{\left[Z_{n\theta}(\theta_{\lambda}+\zeta,h_{\lambda}+g) - Z_{n\theta}(\theta_{\lambda},h_{\lambda})\right]\right.\right.\\
&\qquad\qquad \left.\left.- U_{h\theta}(\theta_{\lambda},h_{\lambda})G_h^{-1}(\theta_{\lambda},h_{\lambda})\left[Z_{nh}(\theta_{\lambda}+\zeta,h_{\lambda}+g) - Z_{nh}(\theta_{\lambda},h_{\lambda})\right]\right\}\right\|_{\R^p}\\
&\qquad = \left\|\zeta - G_{\theta\theta}^{-1}(\theta_{\lambda},h_{\lambda})\left\{\left[U_{\theta}(\theta_{\lambda},h_{\lambda})\zeta + U_{h\theta}(\theta_{\lambda},h_{\lambda})g + (e_{\theta} + R_{n\theta})(\theta_{\lambda},h_{\lambda})\zeta g\right]\right.\right.\\
&\qquad\qquad \left.\left.- U_{h\theta}(\theta_{\lambda},h_{\lambda})G_h^{-1}(\theta_{\lambda},h_{\lambda})\left[G_h(\theta_{\lambda},h_{\lambda})g - U_{\theta h}(\theta_{\lambda},h_{\lambda})\zeta + (e_h + R_{nh})(\theta_{\lambda},h_{\lambda})\zeta g\right]\right\}\right\|_{\R^p}\\
&\qquad = \left\|\zeta - G_{\theta\theta}^{-1}(\theta_{\lambda},h_{\lambda})\left[ U_{\theta}(\theta_{\lambda},h_{\lambda})-(U_{h\theta}(\theta_{\lambda},h_{\lambda})G_h^{-1}(\theta_{\lambda},h_{\lambda})U_{h\theta}(\theta_{\lambda},h_{\lambda}) \right]\zeta\right.\\
&\qquad\qquad \left.- G_{\theta\theta}^{-1}(\theta_{\lambda},h_{\lambda})\left[(e_{\theta} + R_{n\theta})(\theta_{\lambda},h_{\lambda}) - U_{h\theta}(\theta_{\lambda},h_{\lambda})G_h^{-1}(\theta_{\lambda},h_{\lambda})(e_h + R_{nh})(\theta_{\lambda},h_{\lambda})\right]\zeta g\right\|_{\R^p}\\
&\qquad = \left\|G_{\theta\theta}^{-1}(\theta_{\lambda},h_{\lambda})\left\{\left[e_{\theta}(\theta_{\lambda},h_{\lambda}) + R_{n\theta}(\theta_{\lambda},h_{\lambda})\right]\zeta g\right.\right.\\
&\qquad\qquad \left.\left.- U_{h\theta}(\theta_{\lambda},h_{\lambda})G_h^{-1}(\theta_{\lambda},h_{\lambda})\left[e_h(\theta_{\lambda},h_{\lambda}) + R_{nh}(\theta_{\lambda},h_{\lambda})\right]\zeta g\right\}\right\|_{\R^p}\\
&\qquad \le E_{n\theta}^1 \norm{g}_{\CH} + E_{n\theta}^2\norm{\zeta}_{\R^p} + \frac{1}{2}\left(K_{n\theta}^1\norm{g}_{\CH} + K_{n\theta}^2\norm{\zeta}_{\R^p}\right)\norm{g}_{\CH}\\
&\qquad\qquad+ \frac{1}{2}\left(K_{n\theta}^3\norm{g}_{\CH} + K_{n\theta}^4\norm{\zeta}_{\R^p}\right)\norm{\zeta}_{\R^p}.
\end{align*}
Similarly, we have 
\begin{align*}
\norm{F_{nh}(\zeta,g)}_{\CH} \le \Vert g &- G^{-1}_{hh}(\theta_{\lambda},h_{\lambda}) \left[Z_{h}(\theta_{\lambda}+\zeta,h_{\lambda}+g) \right.\\
&- \left. U_{\theta h}(\theta_{\lambda},h_{\lambda})U^{-1}_{\theta}(\theta_{\lambda},h_{\lambda})Z_{\theta}(\theta_{\lambda}+ \zeta,h_{\lambda}+g)\right]\\
&-(\bar{h}_{n\lambda} - h_{\lambda}) \Vert_{\CH} + \norm{\bar{h}_{n\lambda} -h_{\lambda}}_{\CH},
\end{align*}
and
\begin{align*}
&\Vert g - G^{-1}_{hh}(\theta_{\lambda},h_{\lambda}) \left[Z_{h}(\theta_{\lambda}+\zeta,h_{\lambda}+g) \right. \\
& \quad \left. - U_{\theta h}(\theta_{\lambda},h_{\lambda})U^{-1}_{\theta}(\theta_{\lambda},h_{\lambda})Z_{\theta}(\theta_{\lambda}+ \zeta,h_{\lambda}+g)\right] - (\bar{h}_{n\lambda} - h_{\lambda}) \Vert_{\CH}\\ 
& \le E^{1}_{nh} \norm{g}_{\CH} + E^{2}_{nh} \norm{\zeta}_{\R^p} + \frac{1}{2}(K^{1}_{nh}\norm{g}_{\CH} + K^{2}_{nh}\norm{\zeta}_{\R^p})\norm{g}_{\CH} + \frac{1}{2}(K^{3}_{nh}\norm{g}_{\CH} + K^{4}_{nh}\norm{\zeta}_{\R^p})\norm{\zeta}_{\R^p}.
\end{align*}
Thus, for $(\zeta,g)\in S_\theta(t_{n \theta \lambda}) \times S_{h}(t_{n h \lambda})$, 
\begin{align*}
\norm{\vec{F}(\zeta,g)}_{\R^p \times \CH} &\le \left[(E^{1}_{n\theta} + E^{1}_{nh}) + \frac{1}{2}(K^{1}_{n\theta} + K^{1}_{nh})\norm{g}_{\CH} + \frac{1}{2}(K^{2}_{n\theta} + K^{2}_{nh})\norm{\zeta}_{\R^p} \right] \norm{g}_{\CH}\\
& \quad +\left[ (E^{2}_{n\theta} + E^{2}_{nh}) + \frac{1}{2}(K^{3}_{n\theta} + K^{3}_{nh})\norm{g}_{\CH} + \frac{1}{2}(K^{4}_{n\theta} + K^{4}_{nh})\norm{\zeta}_{\R^p} \right] \norm{\zeta}_{\R^p}\\
& \quad + \norm{\bar{\theta}_{n\lambda} - \theta_{\lambda}}_{\R^p} + \norm{\bar{h}_{n\lambda} - h_{\lambda}}_{\CH}\\
&\le r_{nh}(\lambda) t_{nh\lambda} + r_{n\theta}(\lambda)t_{n\theta \lambda} + \frac{1}{2}t_{n\theta \lambda} + \frac{1}{2} t_{nh\lambda}\\
&=\left[r_{nh}(\lambda) + \frac{1}{2}\right]t_{nh\lambda} + \left[ r_{n\theta}(\lambda) + \frac{1}{2}\right] t_{n\theta \lambda}\\
& < t_{nh\lambda} + t_{n\theta \lambda}.
\end{align*}
Therefore, we have shown that $\vec{F}_{n}(S_\theta(t_{n \theta \lambda}) \times S_{h}(t_{n h \lambda})) \subset S_\theta(t_{n \theta \lambda}) \times S_{h}(t_{n h \lambda})$.

Next, we show that $\vec{F}_{n}$ is a contraction map. By similar calculations as in the proof for Theorem 1, after applying Taylor expansion twice, for $\zeta_1,\zeta_2 \in S_\theta(t_{n \theta \lambda})$, $ g_1, g_2 \in S_h(t_{n h \lambda})$, we get
\begin{align*}
\norm{F_{n\theta}(\zeta_1,g_1) - F_{n\theta}(\zeta_2,g_2)}_{\R^p} &\le \left( E^{2}_{n\theta} + K^{3}_{n\theta} t_{nh\lambda} + K^{4}_{n\theta} t_{n\theta \lambda} \right) \norm{\zeta_1 - \zeta_2}_{\R^p}\\
& \quad +\left( E^{1}_{n\theta} + K^{2}_{n\theta} t_{n\theta \lambda} + K^{1}_{n\theta} t_{nh \lambda} \right) \norm{g_1 - g_2}_{\CH},
\end{align*} 
\begin{align*}
\norm{F_{nh}(\zeta_1,g_1) - F_{nh}(\zeta_2,g_2)}_{\CH} &\le \left( E^{2}_{nh} + K^{3}_{nh} t_{nh\lambda} + K^{4}_{nh} t_{n\theta \lambda} \right) \norm{\zeta_1 - \zeta_2}_{\R^p}\\
&\quad +\left( E^{1}_{nh} + K^{2}_{nh} t_{n\theta \lambda} + K^{1}_{nh} t_{nh \lambda} \right) \norm{g_1 - g_2}_{\CH}.
\end{align*}
Thus,
\begin{align*}
\norm{\vec{F}_{n}(\zeta_1,g_1) - \vec{F}_{n}(\zeta_2,g_2)}_{\R^p \times \CH} &\le \left[ E^2_{n\theta} + E^2_{nh} + (K^{3}_{n\theta} + K^{3}_{nh})t_{nh\lambda} +(K^{4}_{n\theta} +K^{4}_{nh})t_{n\theta \lambda} \right] \norm{\zeta_1 - \zeta_2}_{\R^p}\\
&\quad  + \left[ E^1_{n\theta} + E^1_{nh} + (K^{1}_{n\theta} + K^{1}_{nh})t_{nh\lambda} +(K^{2}_{n\theta} +K^{2}_{nh})t_{n\theta \lambda} \right] \norm{g_1 - g_2}_{\CH}\\
& \le 2 r_{n\theta}(\lambda)\norm{\zeta_1-\zeta_2}_{\R^p} + 2r_{nh}(\lambda)\norm{g_1-g_2}_{\CH}.
\end{align*}
Since $r_{n\lambda}(\lambda) < \frac{1}{2}$, $r_{nh}(\lambda) < \frac{1}{2}$, we have shown that $\vec{F}(\zeta,g)$ is a contraction on $S_\theta(t_{n \theta \lambda}) \times S_{h}(t_{n h \lambda})$. By the contraction mapping theorem, there exists a unique $(\zeta_{n \lambda}, g_{n\lambda}) \in S_\theta(t_{n \theta \lambda}, \alpha) \times S_{h}(t_{n h \lambda}, \alpha)$ such that 
\[\vec{F}_{n}(\zeta_{n \lambda}, g_{n\lambda}) = (\zeta_{n \lambda}, g_{n\lambda}).\]
Let $\hat{\theta} = \theta_{\lambda} + \zeta_{n\lambda} \in S_{\theta, \theta_{\lambda}}(2d_{n\theta}(\lambda))$ and $\hat{h} = h_{\lambda} + g_{n\lambda} \in S_{h, h_{\lambda}}(2d_{nh}(\lambda))$. Then $(\hat{\theta},\hat{h})$ is the unique root of $Z_{n\theta}(\hat{\theta},\hat{h}) = 0$ and $Z_{nh}(\hat{\theta},\hat{h}) = 0$. 

To get the upper bound, we observe that 
\begin{align*}
(\bar{\theta}_{n\lambda}-\hat{\theta}, \bar{h}_{n\lambda}-\hat{h}) &= (\bar{\theta}_{n\lambda}-\theta_{\lambda}, \bar{h}_{n\lambda}-h_{\lambda})-(\hat{\theta}-\theta_{\lambda}, \hat{h}-h_{\lambda})\\
&= \vec{F}_{n}(0,0) - \vec{F}_{n}(\zeta_{n\lambda},g_{n\lambda}).
\end{align*}
Therefore,
\begin{align*}
\norm{\bar{\theta}_{n\lambda}-\theta_{n\lambda}}_{\R^p} + \norm{\bar{h}_{n\lambda}-h_{n\lambda}}_{\CH} &= \norm{\vec{F}_{n}(\zeta_{n\lambda},g_{n\lambda})- \vec{F}_{n}(0,0)}_{\R^p \times \CH} \\
&\le 2 r_{n\theta}(\lambda)\norm{\zeta_{n\lambda}}_{\R^p} + 2 r_{nh}(\lambda)\norm{g_{n\lambda}}_{\CH}\\
&\le 4 \left[ r_{n\theta}(\lambda)d_{n\theta}(\lambda) + r_{nh}(\lambda)d_{nh}(\lambda) \right].
\end{align*}
This completes the proof.
\end{proof}

\bibliographystyle{dcu}
\bibliography{dissertation5}
\end{document}